\documentclass[12pt,dvipsnames]{article} % font size
\usepackage[width=180mm,top=25mm,bottom=25mm]{geometry} % page layout
\usepackage[utf8]{inputenc}  % utf8 support

\usepackage{amssymb} % e.g. mathbb
\usepackage{mathtools} % e.g. equations; loads amsmath
\usepackage{hyperref}
\usepackage{amsthm} % proof environment
\usepackage{cleveref}
\usepackage{IEEEtrantools}
\usepackage{dsfont}

\usepackage{stmaryrd} % norm |||.|||
\usepackage{amsmath}
\usepackage{todonotes}
\usepackage{comment}
\usepackage{xcolor}
\usepackage{subcaption}
\usepackage{enumitem}
\numberwithin{equation}{section}
\theoremstyle{plain}
\newtheorem{thm}{Theorem}
\numberwithin{thm}{section}

\newtheorem{lem}[thm]{Lemma}
\newtheorem{prop}[thm]{Proposition}

\theoremstyle{definition}
\newtheorem{defin}[thm]{Definition}

\newtheorem{rem}[thm]{Remark}
\newtheorem{example}[thm]{Example}
\theoremstyle{remark}

\theoremstyle{plain}
\newtheorem{manualtheoreminner}{Theorem}
\newenvironment{manualtheorem}[1]{%
  \IfBlankTF{#1}
    {\renewcommand{\themanualtheoreminner}{\unskip}}
    {\renewcommand\themanualtheoreminner{#1}}%
  \manualtheoreminner
}{\endmanualtheoreminner}
\newcommand{\W}{\mathcal{W}}
\newcommand{\I}{\mathcal{I}}
\newcommand{\K}{\mathcal{K}}
\newcommand{\F}{\mathcal{F}}
\newcommand{\B}{\mathcal{B}}

\newcommand{\C}{\mathcal{C}}

\newcommand{\M}{\mathcal{M}}
\newcommand{\RR}{\mathcal{R}}
\newcommand{\N}{\mathbb{N}}
\newcommand{\R}{\mathbb{R}}
\newcommand{\x}{\mathbf{x}}
\newcommand{\y}{\mathbf{y}}

\renewcommand{\c}{\mathbf{c}}
\renewcommand{\u}{\mathbf{u}}

\newcommand{\ws}{\stackrel{*}{\rightharpoonup}}
\renewcommand{\leq}{\leqslant}
\renewcommand{\geq}{\geqslant}

\DeclareMathOperator*{\argmin}{arg\,min}

\newcommand{\chris}[1]{\todo[color=blue!40, inline]{C: #1}}

\usepackage[backend=biber,maxnames=8,style=numeric]{biblatex}
\title{Atomic Gradient Flows: Gradient Flows on Sparse Representations}
\author{Christian Amend\thanks{Department of Applied Mathematics, University of Twente, 7500AE Enschede, The Netherlands \\
(\texttt{christian.amend@utwente.nl}, \texttt{m.c.carioni@utwente.nl},},\ \ Marcello Carioni\footnotemark[1],\ \ Konstantinos Zemas\thanks{Institute for Applied Mathematics, University of Bonn, Endenicher Allee 60, 53115 Bonn, Germany
(\texttt{zemas@iam.uni-bonn.de})}}
\date{}

\begin{document}

\maketitle

\begin{abstract}
\noindent One of the most popular approaches for solving total variation-regularized optimization problems in the space of measures are Particle Gradient Flows (PGFs). These restrict the problem to linear combinations of Dirac deltas and then perform a Euclidean gradient flow in the weights and positions, significantly reducing the computational cost while still decreasing the energy. In this work, we generalize PGFs to convex optimization problems in arbitrary Banach spaces, which we call \emph{Atomic Gradient Flows} (AGFs).\\
To this end, the crucial ingredient turns out to be 
the right notion of \emph{particles}, or \emph{atoms}, chosen here as 
the extremal points of the unit ball of the regularizer. This choice is motivated by the \emph{Krein–Milman theorem}, which ensures that minimizers can be approximated by linear combinations of extremal points or, as we call them, \emph{sparse representations}. We investigate metric gradient flows of the optimization problem when restricted to such sparse representations, for which we define a suitable \emph{discretized functional} that we show to be to be consistent with the original problem via the means of $\Gamma$-convergence. We prove that the resulting evolution of the latter is well-defined using a minimizing movement scheme, and we establish conditions ensuring $\lambda$-convexity and uniqueness of the flow. These conditions crucially depend on the geometric properties of the set of extremal points as a metric space.\\
Then, using \textit{Choquet's theorem}, we lift the problem into the \textit{Wasserstein space} on weights and extremal points, and consider Wasserstein gradient flows in this lifted setting. As observed for PGFs, this lifted perspective is essential for understanding stability and convergence properties of AGFs.
Our main result is that the lifting of the AGF evolution is again a metric gradient flow in the Wasserstein space, verifying the consistency of the approach with respect to a Wasserstein-type dynamic. \\
Finally, we illustrate the applicability of AGFs to several relevant infinite-dimensional problems, including optimization of functions of bounded variation and  curves of measures regularized by Optimal Transport-type penalties.
\end{abstract}

\tableofcontents

\section{Introduction}

Solving convex optimization problems posed in infinite-dimensional spaces has long been a central challenge in optimization. Over the years, a wide range of algorithms with rigorous convergence schemes have been  devised in order to efficiently face this challenge, developing theory and techniques that are specifically adapted to infinite-dimensional settings. On top of that, infinite-dimensional modeling has surged in popularity in fields that are traditionally more application-driven, such as data-science, inverse problems and imaging.

In this context, variational problems posed over spaces of measures have attracted growing attention, with a particularly effective approach being the so-called \text{Particle Gradient Flows} (PGFs). Roughly speaking, given an ambient space $X$, PGFs solve a convex optimization problem in the space of measures $M(X)$ of the form 
\begin{align}\label{eq:intro_meas}
\inf_{\mu \in M(X)} F\left(\int_{X} \phi\, d\mu \right) + \|\mu\|_{TV}\,,
\end{align}
by restricting the search to empirical (or, so called, \textit{sparse}) measures  of the form
\begin{align*}
\mu_{\rm sparse} := \sum_{j=1}^{n} (c^j)^2 \delta_{x^j}\,,  \ \  \text{where } (c^j, x^j) \in \R_+\times X\,.
\end{align*}
The resulting optimization problem is then solved by performing a gradient flow with respect to the particle weights and positions $(c^j,x^j)$ driven by the energy \eqref{eq:intro_meas}. Empirically, PGFs exhibit strong performance provided that the number of particles $n$ is big enough. Several works have established theoretical results such as convergence to global minimizers under suitable assumptions, by relating the particle dynamics to the Wasserstein gradient flow of the original functional \cite{chizat2018global, C22,wojtowytsch2020convergence,Figalli_Real2022}. Moreover, variants in the space of measures have been considered  \cite{wang2022exponentially, chizat2022convergence, chizat2026quantitative, barboni2025understanding} and applications to machine learning have been explored \cite{rotskoff2018neural, abbe2022merged, sirignano2020mean, dana2025convergence}. 

Generally, PGFs belong to the class of \textit{sparse optimization methods} in infinite-dimensional spaces; one operates by directly optimizing over sparse measures, thus taking advantage of the sparse structure of the problem. Due to its simplicity, it is natural to ask whether approaches similar to PGFs can be applied to more general infinite-dimensional optimization problems, where it is even unclear what sparse objects should be, which is the guiding question of this work.  

Our goal is to extend the PGF approach beyond measure spaces. In particular, in our setting we consider composite convex minimization problems posed on a Banach space $\mathcal{M}$ of the form
\begin{equation}\label{eq:P_u_intro}
\inf_{u\in \mathcal{M}}J(u)\,,\quad \mathrm{where } \ \  J(u):=\F(Ku)+\mathcal{R}(u)\,,
\end{equation}
where $\F$ is a convex fidelity term, $K : \mathcal{M} \rightarrow Y$ is a linear operator mapping to a Hilbert space and $\mathcal{R}$ is a convex, $1$-homogeneous regularizer, cf. \Cref{sec:setting_prelim} for details. 
A key insight underlying our approach comes from recent developments in infinite-dimensional sparsity and representation theory. Sparsity is traditionally understood as the possibility  to represent solutions using only a small number of simple building blocks (atoms). In the recent works \cite{boyer2019representer, bredies2020sparsity}, it has been shown that for optimization problems of the form \eqref{eq:P_u_intro}, the natural atoms are given by the  extremal  points of the unit ball of the regularizer 
\begin{align*}
{\rm Ext}(B)\,, \quad \text{where} \ B := \{u \in \mathcal{M} : \mathcal{R}(u) \leq 1\}\,.
\end{align*}
In particular, when the operator $K$ maps into a finite-dimensional space, representer theorems guarantee the existence of solutions that can be expressed as finite linear combinations of such extremal points.
Motivated by this theory, we focus on \textit{sparse ansatzes} of the form
\begin{equation}\label{eq:sp_intro}
u_{\rm sparse} := \sum_{j=1}^n (c^j)^2u^j\,, \  \text{where } (c^j, u^j)\in \mathbb{R}_+\times \tilde{\mathcal{B}}\,,
\end{equation}
and $\tilde{\mathcal{B}}:=\overline{{\rm Ext}(B)}^*$ is the weak*-closure of the set of extremal points.
This representation serves as the foundation for a generalized particle-based optimization framework, applicable to a wide class of convex problems. \\
Crucially, on the compact set $\tilde{\mathcal{B}}$, the weak*-topology of $\mathcal{M}$ can be metrized by a  metric $d_{*}$. Thanks to this, we can take advantage of the well-developed theory of gradient flows in metric spaces \cite{AGS05} in order to construct the gradient flow of 
\[(\c,\u):=(c^j,u^j)_{j=1}^n \in \R_+^n \times \tilde{\mathcal{B}}^n\,,\] 
according to the energy in \eqref{eq:P_u_intro} restricted to  sparse elements of the form  \eqref{eq:sp_intro}, that is,
\begin{equation}\label{eq:intro_J_n}
J_n(\c,\u) := J\left(\frac{1}{n}\sum_{j=1}^n (c^j)^2 u^j\right)\,.
\end{equation}
We name  such general optimization schemes Atomic Gradient Flows (in short AGFs) to indicate that the gradient flow is performed at the level of the atoms. More precisely, we construct AGFs through a minimizing movement approach, whose well-posedness will be based on the assumptions for the optimization problem \eqref{eq:P_u_intro}, cf. again \Cref{sec:setting_prelim} for the precise setup.
As detailed also in Subsection \ref{subsec:tv_regularization}, our work directly generalizes PGFs, since for $\mathcal{R}:=\|\cdot\|_{TV}$, and $\M:= M( X%\Omega
)$, the extremal points will be ${\rm Ext}(B) = \{\pm \delta_{x} : x\in X%\Omega
\}$.\\
Importantly, the motivation for the definition of AGFs is that the discretized functional $J_n$ provides a variational approximation of the target functional $J$
in the sense of $\Gamma$-convergence. As a consequence, minimizers of $J_n$ converge, as $n\rightarrow +\infty$, to minimizers of $J$, at least under suitable boundedness assumptions on the weights. This provides a theoretical justification for studying the gradient flow of $J_n$: for large $n$ the sparse dynamics of $J_n$ can be viewed as a tractable proxy for approximating minimizers of the infinite-dimensional problem \eqref{eq:P_u_intro} with respect to $J$.

After introducing the main setup and recalling standard convexity and derivative notions in metric spaces in \Cref{sec:setting_prelim}, in \Cref{sec:AGFminmov} we introduce the discretized functional and first justify the consistency of our approach using $\Gamma$-convergence,  cf. Subsection \ref{subsec:discretization_MM}. In Subsection \ref{subsec:MM_AGF} we establish the well-posedness of AGFs via a minimizing movement scheme, and then turn to the investigation of qualitative properties of the flow. In particular, in Subsections \ref{subsec:lambda_cvx_nonlifted}-\ref{subsec:npc_nonlifted} we identify conditions on the linear operator $K$ and on the interaction between the regularizer and the geometry of extremal points which ensure that $J_n$ is $\lambda$-convex for some $\lambda \in \mathbb{R}$. This property is fundamental in the theory of gradient flows in metric spaces, as it guarantees that minimizing movements are curves of maximal slope with respect to the local slope 
$|\partial J_n|$. We then address uniqueness of the minimizing movements defining AGFs, under the additional assumption that the set of  extremal  points is non-positively curved (NPC). The NPC condition is a classical geometric assumption in the theory of gradient flows in metric spaces and plays a central role in establishing contraction and regularity properties \cite{gradient_flows_npc}.

In Section \ref{sec:3_lifted} we introduce and analyze a suitable lifting of the functional $J$ to the space $\mathcal{P}_2(\mathbb{R}_+ \times \tilde{\mathcal{B}})$ of probability measures on $\mathbb{R}_+ \times \tilde{\mathcal{B}}$, with the aim of studying its dynamics as a \textit{Wasserstein gradient flow}.
This construction is motivated by the observation that, as shown in \cite{chizat2018global}, interpreting PGFs as Wasserstein gradient flows in the space of probability measures over weights and positions is crucial for understanding both their dynamical behavior and their convergence properties.
In our more general setting however, such an interpretation is not immediate, since the underlying variational problem is formulated in a Banach space and is therefore not directly expressible in terms of probability measures. 
Instead, however, here one can rely on \textit{Choqu\'et's theorem} \cite{P01}, which ensures that every $u \in B$ can be represented as the weak barycenter of a positive measure $\mu \in M_+(\tilde{\mathcal{B}})$ supported on the set of extremal points,  \textit{i.e.},
\begin{align*}
u = \int_{\tilde{\mathcal{B}}} v\, d\mu(v)\,.
\end{align*}
This representation provides the bridge between the original Banach-space formulation and its lifted counterpart in Wasserstein space. Thanks to this identification, and further factorizing the mass of the measure $\mu \in M_+(\tilde{\mathcal{B}})$ in the domain, following \cite{BCFW23}, one can define the following  lifted problem as
\begin{align}\label{eq:I_def_intro}
\min_{\nu \in \mathcal{P}_2(\R_+\times \tilde{\B})} \mathcal{J}(\nu), \quad \text{where} \ \ \mathcal{J}(\nu) :=  \F\left(\int_{\R_+\times \tilde{\mathcal{B}}} c^2 Ku \,d \nu(c,u) \right) + \int_{\R_+\times \tilde{\mathcal{B}}} c^2\, d\nu(c,u)\,.
\end{align}
We advocate that the convergence properties of AGFs can be understood by analyzing the metric gradient flow of the lifted functional $\mathcal{J}$. A first crucial observation supporting this perspective is that the minimization problems \eqref{eq:I_def_intro} and \eqref{eq:P_u_intro} are equivalent, thanks to the surjectivity of the Choquet representation. This equivalence ensures that the lifted formulation faithfully captures the original variational problem.
After establishing this equivalence and deriving the fundamental properties of the lifted functional in Subsection \ref{sec:lifting_wasserstein_space}, in the rest of the section we study its minimizing movement scheme. Our setting differs substantially from that of \cite{chizat2018global}: In contrast to problems where the extremal points possess additional properties, here we must consider a gradient flow on the metric space $\mathbb{R}_+ \times \tilde{\mathcal{B}}$, which in general lacks any differential structure. As a consequence, classical tools based on continuity equations and differential calculus are not available, which requires a fully metric approach and significantly complicates the analysis.
We prove well-posedness of the minimizing movement scheme for $\mathcal{J }$ and establish its $\lambda$-convexity under the same assumptions that guarantee $\lambda$-convexity of $J_n$. In particular, under these conditions, every minimizing movement is a curve of maximal slope with respect to the metric slope $|\partial \mathcal{J}|$.

Finally, in \Cref{sec:equivalence_gfs}, which contains the main result of this paper, we establish a precise link between the AGF of $J_n$  and the corresponding metric gradient flow of $\mathcal{J}$. Specifically, we show that if $(\c_t, \u_t)$ is a curve of maximal slope for $J_n$ with respect to $|\partial J_n|$, then the curve of empirical measures
\begin{align}\label{eq:flows}
t \mapsto \frac{1}{n} \sum_{j=1}^n \delta_{(c_t^j,u_t^j)}
\end{align}
is a curve of maximal slope for $\mathcal{J}$ with respect to $|\partial \mathcal{J }|$. The significance of this result is that it shows that the metric gradient flow of $\mathcal{J}$ recovers (at a discrete level) the dynamics of  the initial AGF and thus sets the necessary basis for analyzing convergence of the AGF to minimizers of $J$ when $n\rightarrow +\infty$; this analysis is left to future works. The proof relies on a careful comparison between the metric slopes of $J_n$ and $\mathcal{J}$, using techniques from semi-discrete optimal transport \cite{ Aurenhammer1998, Merigot_Semi_Discrete} generalized to metric spaces. These tools allow us to localize the slope estimates around each Dirac mass appearing in the empirical representation \eqref{eq:flows}.

We conclude the paper by presenting in \Cref{sec:examples} several natural examples illustrating the applicability of AGFs to different variational problems. For each example, we provide the characterization of $\mathrm{Ext}(B)$ and analyze the metric induced by the weak*-topology on $\tilde{\mathcal{B}}$. We investigate whether this metric space is NPC, which in turn would imply uniqueness of the AGFs, and we examine the structure of the lifting in each case.
More precisely, we cover the following examples: optimization problems in the space of measures, recovering the framework of \cite{chizat2018global}; one-dimensional BV functions regularized by their total variation seminorm \cite{bredies2020sparsity, carioni2023general} and dynamic problems in the space of time-dependent measures regularized by the Benamou–Brenier energy \cite{bredies2021extremal, bredies2023generalized, duval2024dynamical, carioni2025sparsity}. Moreover, we briefly mention how AGFs could be applied to optimization problems regularized with KR-norms \cite{carioni2023general, bartolucci2024lipschitz}, scalar differential operators \cite{bredies2020sparsity, unser2017splines} and higher dimensional BV functions \cite{ambrosio2001connected}.

\section{Setting and Preliminaries}\label{sec:setting_prelim}
Throughout this paper we consider minimization problems of the form 
\begin{equation}\label{eq:P_u}
\inf_{u\in \mathcal{M}}J(u)\,,\ \ \mathrm{where } \  J(u):=\F(Ku)+\mathcal{R}(u)\,,
\end{equation}
where $\mathcal{M}$ is the topological dual of a separable Banach space $\mathcal{C}$. The norm on $\C$ will be denoted by $\|\cdot\|_{\C}$ and the duality pairing between $u \in  \M$ and $p\in \C $ by $\langle u,p\rangle$. The space $\M$ is a Banach space when equipped with the canonical dual norm
\begin{equation}\label{eq: M_norm}
\|u\|_{\M} := \sup \{\langle u,p\rangle\colon \|p\|_{\C}\leqslant  1\}.
\end{equation}
The forward operator $K :\mathcal{M} \rightarrow Y$ is a linear operator mapping into a Hilbert space $Y$. The inner product and induced norm on $Y$ will be denoted by $(\cdot,\cdot)_Y$ and $\|\cdot\|_Y$ respectively. 
We make the following standard assumptions on $\mathcal{F}$, $K$ and $\mathcal{R}$.

\begin{enumerate}[label=\textnormal{(A\arabic*)}] 
\item\label{assmpt:a1} The forward operator $K:\M\to Y$ is linear, weak$^\ast$-to-weak continuous from $\M$ into a Hilbert space $Y$. 
\item\label{assmpt:a2} The fidelity term $\F:Y\to\R$ is bounded from below, convex, and twice Fréchet differentiable on $Y$.
\item\label{assmpt:a3} The regularizer $\mathcal{R}:\M\to [0,+\infty]$  is convex and positively one-homogeneous, \textit{i.e.},  
\begin{equation}\label{eq:homogeneous_regularizer}
\mathcal{R}(\lambda u) = \lambda \mathcal{R}(u) \ \ \forall \lambda \geqslant 0,\ u\in \M\,,
\end{equation}
 and is also weak$^\ast$-lower semicontinuous.
\item\label{assmpt:a4} For every $\alpha \geqslant 0$, the sublevel set
\begin{equation}\label{eq:sublevel_sets}
S_\alpha^{-}(\mathcal{R}) := \{ u \in \mathcal{M}\colon \mathcal{R}(u) \leqslant  \alpha\}\,,
\end{equation}
is weak$^*$-compact.
\item\label{assmpt:a5} The forward operator $K: \M \to Y$ is sequentially weak$^\ast$-to-strong continuous in the domain of $\mathcal{R}$, defined by
\begin{equation}\label{eq:def_domain}
\mathrm{Dom}(\mathcal{R}) := \{u \in  \M\colon \mathcal{R}(u) < +\infty\}\,.
\end{equation}
\end{enumerate}
Note that, due to \ref{assmpt:a1}, there exists a linear and continuous operator
$K^\ast:Y\to \C$, the \textit{adjoint operator} of $K$, which satisfies 
\begin{equation}\label{eq:adjoint_of_K}
\langle u, K^\ast y \rangle=(Ku, y)_Y\ \quad \forall u\in \M,\ y\in Y\,;
\end{equation}
see, for example, \cite[Remark 3.2]{bredies2013inverse}. Moreover, the existence of the pre-adjoint $K^*$ implies the strong-to-strong continuity of $K$ on $\M$. Note further that \ref{assmpt:a4} implies that the sublevel sets are weak*-closed and norm bounded.
It is immediate to verify (see also \cite[Proposition 2.3]{BCFW23}) that, under the above assumptions,  the existence of a minimizer to \eqref{eq:P_u} is guaranteed.
In what follows, we will denote the unit ball of $\mathcal{R}$ by
\begin{equation}\label{eq:1_ball_of_R}
B := S^{-}_{1}(\mathcal{R})=\{u\in \M\colon \mathcal{R}(u)\leqslant  1\}\,.
\end{equation}
\begin{defin}\label{def:extreme_points}
An element $u\in B\subset \mathcal{M}$ is called an \emph{extremal point} of $B$ if there exists no choice of $u_1, u_2 \in B$ with $u_1\neq  u_2$, and $s \in (0, 1)$ such that $u=(1-s)u_1 + su_2$. The set of all extremal points of $B$ is denoted by $\mathrm{Ext}(B)$. We also set 
\begin{equation}\label{eq:extreme_ball}
\tilde{\B}:=\overline{\textup{Ext}(B)}^*\,.
\end{equation}
\end{defin}
\noindent A consequence of the \textit{Krein-Milman theorem} \cite{KreinMilman} is that $B$ is the weak*-closure of the convex envelope of $\tilde{\mathcal{B}}$.

\begin{rem}\label{rem:metrizability_of_B}
Note that by \ref{assmpt:a3} and \ref{assmpt:a4}, $\tilde{\B}$ is weak$^\ast$-compact and non-empty.
\noindent
Since the predual space $\C$ is separable, there exists a metric $d_{\tilde{\B}}$ metrizing the weak$^\ast$-convergence on $\tilde{\B}$, \textit{i.e.}, for all sequences $(u_k)_{k\in \N} \subset \tilde{\B}$ and $u\in \tilde{\B}$ we have:
\begin{equation}\label{eq:metriz_of_weak_star}
u_k\overset{\ast}{\rightharpoonup}u \ \text{ as } k\to \infty\iff  \lim_{k\to\infty} d_{\tilde{\B}}(u_k,u) = 0 \,.
\end{equation}
In particular, one has that $(\tilde{\B}, d_{\tilde{\B}})$ is a compact metric space and thus separable. Moreover, compactness of $(\tilde \B, d_{\tilde \B})$ guarantees the existence of $u_1,u_2\in \tilde \B$ such that
\begin{equation}\label{eq:metric_bounded}
d_{\tilde{\B}}(u_1,u_2)=\sup\big\{d_{\tilde{\B}}(u,v)\colon u,v\in \tilde{\B}\big\}<+\infty\,.
\end{equation}
\end{rem}

We next recall some of the necessary terminology for metric spaces. In what follows, given a metric space $(X,d)$, a curve $\gamma\colon [0,1]\to X$ is always intended to be continuous with respect to the metric $d$. Of course, the interval of parametrization $[0,1]$ can also be replaced with any subinterval of $[0,+\infty)$. 
\begin{defin}[\textit{Geodesic Space}]\label{def:Geodesic_space}
Let $(X,d)$ be a metric space.  For $x,y\in X$, a curve $\gamma:[0,1]\rightarrow X$ is called a \emph{geodesic between $\gamma_0:=x$ and $\gamma_1:=y$}, if
\begin{equation}\label{eq:unit_geodesic}
d(\gamma_s,\gamma_t) = |s-t|\,d(\gamma_0,\gamma_1)\quad  \forall s,t\in [0,1]\,.
\end{equation}
Then $(X,d)$ is called \emph{geodesically complete} iff for every $x,y\in X$ there exists a geodesic $\gamma$ between $x$ and $y$. The space of all such constant-speed geodesics in $X$, will be denoted by $\mathrm{Geo}(X)$.
\end{defin}

\noindent Note that for simplicity, we only use constant speed geodesics parametrized on the unit interval.
\begin{defin}[\textit{Convexity and $\lambda$-convexity along curves}]
\label{def:metric_lambda_convexity} A functional $F:X \rightarrow (-\infty,+\infty]$ is said to be \textit{convex} along a curve $\gamma\colon[0,1]\rightarrow X$, $\gamma_t := \gamma(t)$, iff
\begin{equation}\label{eq:convexity}
F(\gamma_t)\leq (1-t)F(\gamma_0)+tF(\gamma_1) \quad \forall t\in[0,1]\,,
\end{equation}
and given $\lambda\in \R$, $F$ is said to be $\lambda$-\textit{convex} along $\gamma$, iff
\begin{equation}\label{eq:lambda_convexity}
F(\gamma_t)\leqslant  (1-t)F(\gamma_0) + tF(\gamma_1) - \frac{\lambda}{2}t(1-t) d^2(\gamma_0,\gamma_1)\quad \forall t\in [0,1]\,.
\end{equation}
\end{defin}

For metric gradient flows, properties of the metric itself are crucial for the well-posedness of gradient flows, for instance in establishing uniqueness 
as in Subsection \ref{subsec:npc_nonlifted}. Moreover, it is well known (cf. \cite{gradient_flows_npc}) that  $(X,d)$ being a so-called \textit{space of global non-positive curvature (NPC) in the sense of Alexandrov} implies regularity of the gradient flows. Recall that there are various equivalent conditions for a geodesic space $(X,d)$ to be NPC (for more details see for instance \cite{bacak2014convex,bridson2013metric}).
\begin{defin}[\textit{NPC space}]\label{def:NPC} A geodesic metric space $(X, d)$ is a space of (global) \textit{non-positive curvature} iff for every triplet of points $\gamma_0, \gamma_1, w \in X$ and geodesics $\gamma:[0,1] \rightarrow X$, the following inequality holds:
\begin{align}
\label{eq:convexity_of_dist_along_curve}
d^2(\gamma_t,w) \leqslant  (1-t)d^2(\gamma_0,w)+td^2(\gamma_1,w)- t(1-t) d^2(\gamma_0,\gamma_1)\quad \forall t\in [0,1]\,.
 \end{align}
 \end{defin}

Next, we collect some basic definitions of  differentiability of curves and functionals on metric spaces, which we will use throughout the sequel, and refer the reader to \cite[Chapter 1]{AGS05} for more details. 

\begin{defin}[\textit{Absolutely continuous curves}]\label{def:abs_cont_curves}
Let $(X,d)$ be a complete metric space, $p\in[1,+\infty]$. We say that a curve $\gamma %(\c_t,\u_t)
:[0,1]\to X$ belongs to the space of $p$-absolutely continuous curves $AC^{p}([0,1];X)$ iff there exists $m\in L^{p}(0,1)$ such that
\begin{equation}\label{eq:ac_curves}
d(\gamma_s,\gamma_t) \leqslant  \int_s^t m(r)\, \mathrm{d}r \quad \forall\ 0 \leq s\leqslant   t  \leq 1\,, 
\end{equation}
and analogously define $AC^{p}([a,b];X)$%ne can then also
\ for all $a,b \in [0,+\infty)$ with $a<b$.
\end{defin}

In the case $p=1$ the above definition reduces to the one of absolutely continuous curves and we will
denote the corresponding space simply with $AC([a, b];X)$. In addition, we set
\begin{equation*}
AC([0,\infty);X) := \bigcap_{n=1}^\infty AC([0,n];X)\,, \ \ AC_{\mathrm{loc}}([0,\infty);X) := \{\gamma\colon \gamma\in AC([a,b];X) \ \forall\ 0\leq a<b<\infty\}\,.
\end{equation*}

\begin{defin}[\textit{Metric derivative}]\label{def:metric_derivative}
 Let $(X, d)$ be a complete metric space. For any $\gamma \in AC([0,1];X)$, %absolutely continuous curve $v:[0,1]\rightarrow X$, 
 we define the metric derivative of $\gamma$ at $t\in (0,1)$ as 
\begin{equation}\label{eq:metric_derivative}
\left|\gamma'\right|(t):=\lim _{h \rightarrow 0} \frac{d(\gamma(t+h), \gamma(t))}{|h|}\,.
 \end{equation}
\end{defin}
\noindent
The above limit is defined for $\mathcal{L}^1$-a.e. $t\in(0,1)$, % we have that 
$|\gamma'|\in L^1(0,1)$, and $|\gamma'|$ is exactly the \textit{minimal} function $m \in$ $L^1(0,1)$ satisfying the inequality \eqref{eq:ac_curves}.
We next turn to the standard definitions of differentials for functionals defined on metric spaces.
\begin{defin}[\textit{Local slope}]\label{def:local_slope}
Let $F: X \rightarrow(-\infty,+\infty]$ with proper effective domain, \textit{i.e.}, 
\[\mathrm{Dom}(F)\colon= \{u\in X\colon F(u)<+\infty\}\neq \emptyset\,.\]
We define the local slope of $F$ at a point $u \in \operatorname{Dom}(F)$ as
\begin{equation}\label{eq:local_slope}
|\partial F|(u):=\limsup _{w \rightarrow u} \frac{(F(u)-F(w))^{+}}{d(u, w)}\,.
\end{equation}
\end{defin}

Under suitable assumptions, the local slope is indeed a kind of metric gradient for the functional, in the following sense.
\begin{defin}[\textit{Strong upper gradient}]\label{def:Strong_upper_gradient}
A function $g: X \rightarrow[0,+\infty]$ is a strong upper gradient for $F: X\rightarrow(-\infty,+\infty]$ iff for every  $\gamma \in AC([0,1];X)$, the function $g \circ \gamma $ is Borel and the following inequality holds:
\begin{equation}\label{def:strong_upper_grad}
|F(\gamma (t))-F(\gamma (s))| \leq \int_s^t g (\gamma (r))\left|\gamma^{\prime}\right|(r)\, d r \quad \forall 0<s \leq t<1\,.
\end{equation}
\end{defin}

\noindent In particular, if $g \circ \gamma\left|\gamma'\right| \in L^1(0, 1)$, then $F \circ \gamma\in AC([0,1];X)$ %is absolutely continuous
and 
\begin{equation}\label{eq:strong_upper_grad}
\left|(F \circ \gamma)^{\prime}\right|(t) \leq g(\gamma(t))\left|\gamma^{\prime}\right|(t) \quad \text { for } \mathcal{L}^1\text{-a.e. } t \in(0, 1)\,.
\end{equation}

 We next recall the concept of curves of maximal slope, which is a suitable generalization of the standard concept of gradient flows in the metric setting.

\begin{defin}[\textit{Curves of maximal slope}]\label{def:curves_max_slope}
A curve $\gamma\in AC_{\mathrm{loc}}([0,+\infty);X)$ is said to be  %A locally absolutely continuous map $u:[0, +\infty) \rightarrow X$ is 
a curve of maximal slope for a functional $F\colon X\to (-\infty,+\infty]$ with respect to its strong upper gradient $g$, iff the function $F \circ \gamma\colon [0,\infty)\to (-\infty,\infty]$ is $\mathcal{L}^1$-a.e. equal to a non-increasing map $\varphi$, and 
\begin{equation}\label{eq:curve_max_slope}
\varphi'(t) \leq-\frac{1}{2}\left|\gamma'\right|^2(t)-\frac{1}{2} g^2(\gamma(t)) \quad \text { for } \mathcal{L}^1\text {-a.e. } t >0\,.
 \end{equation}
\end{defin}

%For curves with values in a Hilbert space $Y$, differentiability can also be defined in a weak sense, as in the following.

%\begin{defin}[\textit{Weak$^*$ differentiability}]\label{def:weak_star_deriv}
%We say that a curve $\gamma:[0,1]\rightarrow Y$ is weakly* differentiable at $t \in (0,1)$ iff there exists an element $\frac{d}{dt}\gamma(t) \in Y$ such that

%\begin{equation}\label{def:weak*_derivative}
%\frac{\gamma(t+h) - \gamma(t)}{h} \stackrel{*}{\rightharpoonup} \frac{d}{dt}\gamma(t)  \quad \text {as } h \rightarrow 0\,,
%\end{equation}
%\textit{i.e.}, 
%\begin{equation*}
%\lim_{h\rightarrow 0}\left(\frac{ \gamma(t+h) - \gamma(t)}{h},y\right)_Y = \left(\frac{d}{dt} \gamma(t),y\right)_Y\, \quad \forall y\in Y\,,
%\end{equation*}
%and $\gamma$ is called  weakly* differentiable iff it is weakly* differentiable at every $t\in(0,1)$.
%\end{defin}

\section{Formulation of the Atomic Gradient Flow}\label{sec:AGFminmov}
In this section we present a discretization approach to problem \eqref{eq:P_u}, which we show to be consistent with the original problem by means of $\Gamma$-convergence. To approximate minimizers of \eqref{eq:P_u}, we then formalize the evolution of the discretized functional through \textit{Atomic Gradient Flows} (in short AGFs) by using a minimizing movement approach.
\subsection{Discretization of the functional $J$ and consistency}\label{subsec:discretization_MM}

We start by setting the necessary notations to define the restriction of the optimization functional to sparse representations. In what follows, we always follow the convention that components of vectors will be denoted as superscripts, while subscripts will typically denote a time-parameter.

\begin{defin}
\label{def:omega_n}
Recalling the definitions of $\tilde{\B}$ in \eqref{eq:extreme_ball} and the functional $J$ in \eqref{eq:P_u},  let us fix $L>0$, and a closed and  geodesically complete  %connected 
subset $\B\subset \tilde \B$ (cf. \Cref{def:Geodesic_space}), and a metric $d_\B$ metrizing the weak$^\ast$-convergence on $\B$. Then, for $n\in\N$, we set
\begin{equation}\label{eq:Omega_n_def}
\Omega^n_L:= [0,L]^n\times \B^n\quad \text{and} \quad \tilde{\Omega}^n_L:= [0,L]^n\times \tilde{\B}^n.
\end{equation}
We endow $\Omega^n_L$ with the distance
\begin{equation}\label{eq: tau_dissip}
{d_n}((\c,\u),(\tilde{\c},\tilde{\u})):=
\left(\frac{1}{n}\sum_{j=1}^n\big(|c^j-\tilde{c}^{j}|^2 + d^2_\mathcal{B} (u^j,\tilde{u}^{j})\big)\right)^{1/2}\,,
\end{equation}
and analogously for $\tilde{\Omega}^n_L$ with $d_{\tilde{\B}}$ in place of $d_\B$ in \eqref{eq: tau_dissip}.
For $n=1$ we write $\Omega_L:=\Omega_L^1$ and $d_\Omega:=d_1$, \textit{i.e.}, 
\begin{equation}\label{eq:d_omega_def}
d_\Omega((c,u),(\tilde{c},\tilde{u})) = \left(|c-\tilde{c}|^2 + d^2_\mathcal{B} (u,\tilde{u})\right)^{1/2}\,.
\end{equation}
\end{defin}

It is well-known that the properties of $(\B,d_\B)$ propagate to $(\Omega_L^n, d_n)$, as the following lemma suggests. 
\begin{lem}\label{lem:coord_geodesics}  The spaces 
$\Omega_L$ and $\Omega_L^n$, $n>1$,  are both compact, separable and geodesic metric spaces.
Moreover, a curve
\[
(\c_t,\u_t)\colon [0,1]%_{t\in[0,1]} \in 
\to \Omega_L^n\,, 
\ \text{with}\
(\c_t,\u_t)=\big((c_t^1,u_t^1),\dots,(c_t^n,u_t^n)\big)
\]
is a geodesic in $(\Omega_L^n,d_n)$ iff for every $j\in\{1,\dots,n\}$ the coordinate curve
\((c_t^j,u_t^j):[0,1]\to\Omega\) is a geodesic in $(\Omega,d_\Omega)$, or equivalently, iff
$c_t^j:[0,1]\to [0,L]$ and $u_t^j:[0,1]\to\B$ are geodesics in their respective spaces.
\begin{proof}
That $\Omega_L$ and $\Omega_L^n$ are compact and separable metric spaces is a simple consequence of equipping them with the $\ell^2$-metric, cf. \eqref{eq: tau_dissip}. 
The existence of geodesics and geodesic completeness follows from the product–geodesic property in, \textit{e.g.}, \cite[Proposition 5.3]{bridson2013metric}. Since we endowed $\Omega_L^n$ with the $\ell_2$-metric $d_n$, $(\c_t,\u_t)$ is a  unit
-speed geodesic in $\Omega_L^n$ iff,  for every $j\in\{1,\dots,n\}$,  each coordinate $(c_t^j,u_t^j)$ is a unit-speed geodesic in $\left(\Omega, d_{\Omega}\right)$, cf. \eqref{eq:d_omega_def}. As $d_{\Omega}$ is the $\ell_2$-metric on $[0,L] \times \B$ and both spaces in the product are geodesic spaces, $(c_t^j, u_t^j)$ is a  unit-speed geodesic in $\Omega$ iff $c_t^j$ and $u_t^j$ are unit-speed geodesics in $[0,L]$ and $\B$, respectively. 
\end{proof}
\end{lem}
Note that above, and also in what follows, with a slight abuse of notation we have reordered the components of a vector $(\c,\u)\in \Omega_L^n$ into an $n$-component vector with the entries being pairs of the form $(c,u)\in \Omega_L$, with the initial enumeration being taken into account, hence identifying 
\begin{equation}\label{eq:coordinates_of_c_u}
(\c,\u):=\big((c^1,\dots,c^n), (u^1,\dots,u^n)\big)=\big((c^1, u^1), \dots , (c^n, u^n)\big)\in  \Omega_L^n\,.
\end{equation}
Adopting this notation throughout, we next introduce a \textit{discretized version} of the initial functional $J$ of \eqref{eq:P_u}, by restricting it to sparse representations.

\begin{defin}[\textit{Discretization by restriction to sparse representations}] We define the functional $J_n:\Omega_L^n\to \R$ as
\begin{equation}\label{eq:discrete_version_of_energy}
J_n(\c,\u):=J\left(\frac{1}{n}\sum_{j=1}^{n}(c^j)^2u^j \right)\quad \forall (\c,\u)\in \Omega_L^n\,.
\end{equation}
One defines $J_n$ on $\tilde{\Omega}_L^n$ in the same way.
\end{defin}
\noindent 
In general, by  \ref{assmpt:a3}  and the fact that $u^j\in \tilde{\B} = \overline{\mathrm{Ext}\{\RR(u)\leq 1\}}^\ast$ for all $j\in\{1,\dots,n\}$ (which in particular implies that either $\R(u^j)=1$ or $u^j=0$), one has
\begin{equation}\label{eq:no_loss_upper_bound}
\RR\left(\sum_{j=1}^n \alpha^ju^j\right)\leq \sum_{j=1}^n \alpha^j\,, \  \text{where } \alpha^j\geq 0\quad \forall j\in\{1,\dots,n\}\,.
\end{equation}
In addition to \ref{assmpt:a3}--
\ref{assmpt:a5}, it is convenient to impose the following extra condition on the regularizer $\RR$, called the no loss of mass condition.

\begin{enumerate}[label=\textnormal{(A\arabic*)}] 
\setcounter{enumi}{5}
\item \label{assmpt:a6} The regularizer $\RR$ has \emph{no loss of mass} on $\B\subset \tilde{\mathcal{B}}$, \textit{i.e.}, for every $n\in \N$, $(\alpha^j)_{j=1}^n \subset \R_+$ and $(u^j)_{j=1}^n \subset \B$, it holds that
\begin{equation}
\label{eq:cond_no_loss_of_mass}
\RR\left(\sum_{j=1}^n \alpha^ju^j\right) = \sum_{j=1}^n \alpha^j \RR(u^j)\,.
\end{equation}
\end{enumerate}

The following assumption will also be necessary for the some of the subsequent results, for which we will mention it explicitly. 
\begin{enumerate}[label=\textnormal{(A\arabic*)}] 
\setcounter{enumi}{6}
\item \label{assmpt:a7}  The closed, geodesically complete set $\B\subset \tilde{\mathcal{B}}$ does not to contain $0$, \textit{i.e.}, $0\notin \B\,.$

\end{enumerate}

\begin{rem} We make the following remarks about the previous definitions and assumptions.

\medskip

(i) The setting described above for the AGF-evolution 
is natural, since firstly, one has that the canonical identification with elements in the lifted Wasserstein-2 space of probability measures (see \Cref{sec:lifting_wasserstein_space} for the respective definitions), given by 
\begin{align*}
(\Omega_L^n,d_n) \ni (\c,\u) \mapsto \frac{1}{n} \sum_{j=1}^n \delta_{(c^j,u^j)} \in (\mathcal{P}_2(\Omega^n_L),\W_2)
\end{align*}
is $1$-Lipschitz, as 
\begin{align*}\W_2\left(\frac{1}{n}\sum_{j=1}^n\delta_{(c^{j},u^{j})},\frac{1}{n}\sum_{j=1}^n\delta_{(\tilde c^j,\tilde u^j)}\right) \leq \left(\frac{1}{n}\sum_{j=1}^n d_\Omega^2\big((c^j,u^j),(\tilde c^{j},\tilde u^{j})\big)\right)^{1/2}.
\end{align*}
Secondly, using the Krein-Milman theorem, we prove in \Cref{prop:gamma_conv} that on $\tilde{\B}$, $J_n {\rightarrow}J$ as $n\to \infty$, in the sense of $\Gamma$-convergence. Afterwards, we restrict to a geodesically connected subset $\B$ of $\tilde{\mathcal{B}}$, which is sensible in the context of gradient flows, and analogous to replacing $\{\pm\delta_x: x\in X\}$ by $\{\delta_x: x\in X\}$ for the optimization problem \eqref{eq:intro_meas}.

\medskip 

 (ii)  In particular, with the choice $\mathcal{R}(\cdot)=\|\cdot\|_{TV}$, the latter being the total variation norm in the space of measures, the no loss of mass property \eqref{eq:cond_no_loss_of_mass} holds.

\medskip

(iii) It is important to note that many of our results can be proved in the more general setting of unbounded weights in $\R_+$ and without assuming the no loss of mass of property. In those cases, we simply state the result for, \textit{e.g.}, $\Omega_\infty:=\R_+\times \B$ instead of using $\Omega_L$. But for ease of presentation, unless specifically mentioned, we assume for simplicity both \ref{assmpt:a6}-\ref{assmpt:a7}  throughout the manuscript.
\end{rem}

Due to the compactness of $\Omega_L^n$ and our constitutive assumptions, one can easily prove that $J_n$ is lower-semicontinuous and thus admits a minimizer in $\Omega_L^n$ by the direct method in the Calculus of Variations.

\begin{lem}[\textit{Existence of minimizers for the discretized problem}]
\label{lem:lsc_ex_mins} Under assumptions \ref{assmpt:a1}-\ref{assmpt:a5}, $J_n$ is lower-semicontinuous on $\tilde{\Omega}_n$, and $\inf_{\tilde{\Omega}_L^n}J_{n}$ admits a minimizer. In addition, the same holds on $\Omega_L^n$.
\begin{proof}
Let $(\c_\ell,\u_\ell)_{\ell\in \N}\subset\tilde{\Omega}_L^n$ be a sequence such that $(\c_\ell,\u_\ell)\overset{d_n}{\longrightarrow} (\c,\u)\in \tilde{\Omega}_L^n$ as $\ell\to \infty$.
In particular, since $d_\mathcal{B}$ metrizes the weak*-convergence in $\mathcal{B}$, recalling \eqref{eq:metriz_of_weak_star}  and \Cref{def:omega_n}, we have that componentwise, 
$$ (c_{\ell}^j)^2u_{\ell}^j\ws (c^j)^2u^j  \ \ \text{as } \ell\to \infty\,, \quad \forall j\in\{1,\dots,n\}\,.$$ 
\noindent
Therefore, denoting
\begin{equation}\label{eq:x_points}
x_\ell:=\frac{1}{n}\sum_{j=1}^{n}(c_{\ell}^j)^2u_{\ell}^j\,,\quad\textup{and}\quad x:=\frac{1}{n}\sum_{j=1}^{n}(c^j)^2u^j\,, 
\end{equation} 
it holds that $x_\ell\ws x$ as $\ell\to \infty$. In view of \ref{assmpt:a1}, \eqref{eq:x_points} implies that $Kx_\ell\rightharpoonup Kx$ in $Y$, and by \ref{assmpt:a2}, $\mathcal{F}(Kx) \leq \liminf_{\ell\to \infty}\mathcal{F}(Kx_\ell)$.
Similarly, \ref{assmpt:a3} combined again with \eqref{eq:x_points} and the super-additivity of the $\liminf$, yield  in total
\begin{equation}\label{eq:J_n_lsc}
J_n(\c,\u)\leq\liminf_{\ell\to \infty}J_n(\c_\ell,\u_\ell)\,,
\end{equation}
 \textit{i.e.}, $J_n$ is lower semicontinuous. The existence for minimizers for $J_n$ follows by the compactness of $\tilde{\Omega}_L^n$ (since $\tilde{\B}$ is weak$\ast$-compact) as a standard application of the direct method in the Calculus of Variations.\\
 For $\Omega_L^n$, the same arguments can be repeated verbatim.
\end{proof}
\end{lem}

The strategy of optimizing $J_n$ as a substitute for  $J$ is based on the intuition that the variational problem associated with $J_n$ provides a consistent approximation of the original problem. More precisely, as $n$ increases, one anticipates that minimizers of $J_n$ will converge, in an appropriate sense, to minimizers of $J$. We will formalize such a result using a $\Gamma$-convergence approach, cf. \cite{dalmaso19993Gammaconvergence} for a detailed treatment.\\
Recalling the notation \eqref{eq:coordinates_of_c_u}, we first note that one needs to work with unbounded weights and the full set $\tilde{\B}$ here in order to compare $J_n$ with $J$ on the domain of $\mathcal{R}$. Thus, 
we extend $J_n$ to $v\in {\rm Dom}(\mathcal{R})$ by setting \begin{align}\label{eq:extended_J_n}
J_n(v) := \begin{cases} J\left(\frac{1}{n} \sum_{j=1}^n (c^j)^2u^j \right), & \text { if } v = \frac{1}{n} \sum_{j=1}^n (c^j)^2u^j\,, \text{ for some } (\c,\u)\in \R^n_+\times \tilde{\mathcal{B}}^n\,, \\ +\infty, & \text { otherwise. }\end{cases}
 \end{align}
\begin{prop}[$\Gamma$\textit{-Convergence result}]\label{prop:gamma_conv}
Assume \ref{assmpt:a1}-\ref{assmpt:a5}. With respect to the weak$^\ast$-topology on ${\rm Dom}(\mathcal{R})$,  we have that $J_n \overset{\Gamma}{\rightarrow} J$ as $n\to\infty$, \textit{i.e.}, 
\begin{itemize}
\item \emph{($\Gamma$-liminf inequality):} For every $(v_n)_{ n\in \N}\subset {\rm Dom}(\mathcal{R})$ with $v_n \overset{\ast}{\rightharpoonup}v\in {\rm Dom}(\mathcal{R})$, it holds that
\[\displaystyle{J(v)\leq \liminf_{n\to\infty} J_n(v_n)}\,.\]
\item \emph{($\Gamma$-limsup inequality):} For each $v\in {\rm Dom}(\mathcal{R})$, there exists a recovery sequence $(v_n)_{n\in \N }\subset {\rm Dom}(\mathcal{R})$ with  $v_n\overset{\ast}{\rightharpoonup} v\in {\rm Dom}(\mathcal{R})$ such that  
\[\limsup_{n\to\infty} J_n(v_n)\leq \displaystyle{J(v)}\,.\]
Actually, for every $n\in \N$, $v_n$ is given by some $(\tilde{c}^j_{n})_{j=1}^n \subset  [0,\sqrt{n\mathcal{R}(v)}]$, and $(\tilde{u}_{n}^j)_{j=1}^n\subset  \tilde{\B}$ such that
\begin{equation*}
v_n=\frac{1}{n}\sum_{j=1}^{n} (\tilde{c}_{n}^j)^2 \tilde{u}_{n}^j,\quad   \text{and }\quad      \frac{1}{n}\sum_{j=1}^{n} (\tilde{c}_{n}^j)^2 \leq \mathcal{R}(v).
\end{equation*}
\end{itemize}
\end{prop}
 \begin{proof}
 (\emph{$\Gamma$-liminf}) Let $(v_n)_{n\in \N}\subset {\rm Dom}(\mathcal{R})$ be such that $v_n \overset{\ast}{\rightharpoonup} v\in{\rm Dom}(\mathcal{R})$ as $n\to \infty$.
If we suppose that $\liminf_{n\to\infty} J_n(v_n)=+\infty$, then there is nothing to prove. Otherwise,  we may without restriction assume that %it is enough to consider subsequences  such that 
$ \sup_{n\in \N}J_n(v_n)<\infty$.  % for all $n$, since this will only decrease the $\liminf$.
By the definition of $J_n$ on ${\rm Dom}(\mathcal{R})$, this implies that each $v_n$ admits a representation as in \eqref{eq:extended_J_n}, and
$
J_n(v_n)=J(v_n).
$
Since $J$ is weak$^\ast$-lower semicontinuous and $v_n\overset{\ast}{\rightharpoonup} v$ as $n\to \infty$, we get
\begin{align*}
J(v) \leq \liminf_{n\rightarrow\infty} J(v_n) \leq \liminf_{n\rightarrow\infty} J_n(v_n)\,,
\end{align*}
which is the desired $\Gamma$-liminf inequality.

\medskip  
      
(\emph{$\Gamma$-limsup}) Take $v\in {\rm Dom}(\mathcal{R})$, for which we can suppose that $\mathcal{R}(v) \neq 0$. Indeed, otherwise, $v=0$, cf. \ref{assmpt:a3}, and thus the recovery sequence would be the trivial one. \\
We claim that for $n\in \N$ there exist $(\beta^j_n)_{j=1}^n\subset \R_+$, $(u^j_n)_{j=1}^n\subset \tilde{\B}$ such that
\begin{equation}\label{eq:recovery_sequence}
v_n:=\frac{1}{n}\sum_{j=1}^{n} (\beta_n^j)^2 u_n^j  \overset{\ast}{\rightharpoonup} v, \quad \text{ and } \quad \frac{1}{n}\sum_{j=1}^n (\beta_n^j)^2 = \mathcal{R}(v)\,.
\end{equation}

Indeed, by the Krein-Milman theorem, there exists a sequence of convex combinations of extremal points $(u_n^j)_{j=1}^{k_n}\subset \tilde \B$, with $k_n \in \mathbb{N}$, so that
\begin{align*}
\sum_{j=1}^{k_n}\tilde{\alpha}_n^j u_n^j \overset{\ast}{\rightharpoonup} \frac{v}{\mathcal{R}(v)},  \ \text{ with } \tilde{\alpha}_n^j \geq 0 \ \text{ and } \sum_{j=1}^{k_n}\tilde{\alpha}_n^j =1\,.
\end{align*}
Note that by adding extremal points with weights equal to zero we can assume that $k_n$ is strictly increasing in $n$ and thus $k_n \rightarrow +\infty$.
Moreover, by a rescaling argument (and multiplying both sides of the above with $\mathcal{R}(v)$), we can find weights $\alpha_n^j \geq 0$ such that
\begin{equation*}
\frac{1}{k_n}\sum_{j=1}^{k_n}(\alpha_n^j)^2 u_n^j \overset{\ast}{\rightharpoonup} v,  \ \text{ with } \alpha_n^j \geq 0 \,, \text{ and } \frac{1}{k_n}\sum_{j=1}^{k_n}(\alpha_n^j)^2 = \mathcal{R}(v)\,,
\end{equation*}
 in particular $\alpha_n^j:=\sqrt{k_n\mathcal{R}(v)\tilde \alpha_n^j}$. 
Define now the following sequence. Set $k_0:=0$, and define 
\begin{equation*}
v_n := \begin{cases}
0\,, &\text{ if } n \leq k_1\,,\\[2pt]
\frac{1}{k_\ell}\sum_{j=1}^{k_\ell}(\alpha_\ell^j)^2 u_\ell^j\,, &\text{ if } k_\ell  <  n  \leq  k_{\ell+1}\,.
\end{cases}
\end{equation*}
Then clearly still $v_n \overset{\ast}{\rightharpoonup} v$ as $n\to \infty$ and for $n\geq k_1$ it holds that for any extremal point $\bar u \in \tilde{\mathcal{B}}$ that  
\begin{align*}
v_n & = \frac{1}{k_\ell}\sum_{j=1}^{k_\ell}(\alpha_\ell^j)^2 u_\ell^j = \frac{1}{k_\ell}\sum_{j=1}^{k_\ell}(\alpha_\ell^j)^2 u_\ell^j + \sum^{n - k_\ell}_{ j=k_\ell+1} 0 \cdot \bar u = \frac{1}{n}\sum_{j=1}^{k_\ell}\Big(\sqrt\frac{n}{k_\ell}\alpha_\ell^j\Big)^2 u_\ell^j + \sum^{n - k_\ell}_{ j=k_\ell+1} 0 \cdot \bar u = \frac{1}{n}\sum_{j=1}^n (\beta^j_n)^2 \tilde u^j_n\,,
\end{align*}
for a suitable choice of $ \beta^j_n \geq 0$ and $(\tilde u^j_n)^n_{j=1} \subset \tilde {\mathcal{B}}$ with $\frac{1}{n}\sum_{j=1}^n (\beta^j_n)^2 = \mathcal{R}(v)$, which proves \eqref{eq:recovery_sequence}.

 It remains to show that the constructed sequence in \eqref{eq:recovery_sequence} is a recovery sequence. By assumption \ref{assmpt:a5}, the weak$^\ast$-to-strong-continuity of $v\mapsto \mathcal{F}(Kv)$ holds, implying that
\begin{align*}
\mathcal{F}(Kv)= \lim_{n\rightarrow\infty} \mathcal{F}\left(\frac{1}{n}\sum_{j=1}^{n} (\beta^j_n)^2 Ku_n^j\right).
\end{align*}
By Jensen's inequality and \ref{assmpt:a3}, it holds that
\begin{align*}
\mathcal{R}\left(\frac{1}{n}\sum_{j=1}^{n} ( \beta^j_n)^2 u_n^j\right)\leq \frac{1}{n}\sum_{j=1}^{n} (\beta^j_n)^2 \mathcal{R}(u_n^j)  \leq \frac{1}{n}\sum_{j=1}^{n} (\beta^j_n)^2 = \mathcal{R}(v)\,.
\end{align*}
 Therefore, it follows by the weak$^\ast$-lower semicontinuity of $\mathcal{R}$  that 
 \begin{align*}
 \mathcal{R}(v) \leq \liminf_{n\rightarrow\infty} \mathcal{R}\left(\frac{1}{n}\sum_{j=1}^{n} (\beta^j_n)^2 u_n^j\right)\leq \mathcal{R}(v)\,.
\end{align*}
As a consequence, equality must hold and thus
\begin{align*}
\mathcal{R}(v) = \lim_{n\rightarrow\infty}\mathcal{R}\left(\frac{1}{n}\sum_{j=1}^{n} (\beta^j_n)^2 u_n^j \right),\quad \text{so also } \
J(v) = \lim_{n\rightarrow\infty}J\left(\frac{1}{n}\sum_{j=1}^{n} (\beta^j_n)^2 u_n^j\right)\,,
\end{align*}
as desired.
 \end{proof}

By the coercivity of $J$ and general properties of $\Gamma$-convergence, every sequence of minimizers of $J_n$ of the form
$
v_n:=\frac{1}{n} \sum_{j=1}^n\left(c_n^j\right)^2 u_n^j
$
admits a weak*-convergent subsequence, and every weak*-limit minimizes $J$.

 \begin{prop}[Convergence of Minimizers]
\label{prop:compactness_of_min_nonlifted}
Assume \ref{assmpt:a1}-\ref{assmpt:a5}. Then, there exists $L>0$ such that for every sequence of minimizers
\begin{align*}
(\c_n,\u_n) \in \argmin_{(\c,\u)\in \tilde{\Omega}_{\sqrt{n L}}^n} J_n(\c,\u)\,,
\end{align*}
the linear combination $\frac{1}{n}\sum_{j=1}^{n} (c_{n}^j)^2u_{n}^j$ admits a convergent subsequence. Moreover,
for every such convergent subsequence it holds that 
\begin{align*}
 &\lim_{k\rightarrow\infty} \frac{1}{n_k}\sum_{j=1}^{n_k} (c_{n_k}^j)^2u_{n_k}^j =  v_0 \in \argmin_{u\in \M} J(u)\,.
\end{align*}
\end{prop}
\begin{proof}
Let us consider a minimizer $v_*$ of $J$ in $\M$ and choose $L > \mathcal{R}(v_*)$. 
Choosing $(\c_n,\u_n)_{n\in \N}$ as in the statement, we obtain
\begin{align*}
\mathcal{R}\left(\frac{1}{n}\sum_{j=1}^n (c^j_n)^2 u^j_n\right) & = J_n(\c_n,\u_n) - \mathcal{F} \left(\frac{1}{n}\sum_{j=1}^n (c^j_n)^2 K u^j_n\right)  \leq \min_{(\c,\u)\in \Omega^n_{\sqrt{nL}}}J_{n}(\c,\u) - \inf_{v\in \M}\F(v)\,.
\end{align*}
The term in the right hand side above remains bounded, since it holds that for every $n\in \N$,
\begin{equation*}
\min_{(\c,\u)\in\tilde{\Omega}^{n+1}_{\sqrt{(n+1)L}}} J_{n+1}(\c,\u) \leq \min_{(\c,\u)\in \tilde{\Omega}^n_{\sqrt{nL}}} J_n(\c,\u)\,,
\end{equation*}
as one can show that all elements in $\tilde{\Omega}_{\sqrt{n L}}^{n}$ are valid competitors for the problem in $\tilde{\Omega}_{\sqrt{(n+1) L}}^{ n+1}$, by adding an extremal point with zero weight, and rescaling.
Therefore, there exists a $C >0$ such that
\begin{align*}
\frac{1}{n}\sum_{j=1}^n (c_n^j)^2 u_n^j \in \{v: \mathcal{R}(v)\leq C\} \,,
\end{align*}
which is weak*-compact by \ref{assmpt:a4}.
Hence, there exists a subsequence $(n_k)_{k\in \N}$ and $v_0\in \M$ such that
\begin{align*}
\frac{1}{n_k}\sum_{j=1}^{n_k} (c_{n_k}^j)^2 u_{n_k}^j \stackrel{*}{\rightharpoonup} v_0\,.
\end{align*}
It remains to prove that $v_0 \in \argmin_{v\in\M}J(v)$. 
Since $v_* \in {\rm Dom}(\mathcal{R})$, by the $\Gamma$-convergence result of Proposition \ref{prop:gamma_conv}, there exists a recovery sequence $(\tilde \c_{n},\tilde \u_{n}) \in \tilde{\Omega}^n_{\sqrt{nL}}$
such that
\begin{align*}
\frac{1}{n}\sum_{j=1}^{n} (\tilde{c}_{n}^j)^2 \tilde{u}_{n}^j \stackrel{*}{\rightharpoonup} v_*\,,\quad         \frac{1}{n}\sum_{j=1}^{n} (\tilde{c}_{n}^j)^2 \leq \mathcal{R}(v_\ast)\,, \quad \text{and} \quad J_{n}\left(\frac{1}{n}\sum_{j=1}^{n} (\tilde{c}_{n}^j)^2 \tilde{u}_{n}^j\right) \longrightarrow J(v_*)\, \ \text{ as } n\to\infty\,.
\end{align*}
In particular, $(\tilde{\c}_n,\tilde{\u}_n)$ is an admissible competitor for $J_n$, for every $n\in \N$, so that
\begin{equation*}
J_n(\c_n,\u_n) \leq J_n(\tilde{\c}_n,\tilde{\u}_n)\,,
\end{equation*}
which, together with the $\Gamma$-liminf inequality implies that 
\begin{align*}
J(v_*)\geq \limsup_{k\rightarrow\infty} J_{n_k}\left(\frac{1}{n_k}\sum_{j=1}^{n_k} (\tilde{c}_{n_k}^j)^2 \tilde{u}_{n_k}^j\right) \geq \liminf_{k\rightarrow\infty} J_{n_k}\left(\frac{1}{n_k}\sum_{j=1}^{n_k} (c_{n_k}^j)^2 u_{n_k}^j\right) \geq J(v_0)\,,
\end{align*}
\textit{i.e.}, $v_0 \in \argmin_{v\in\M}J(v)$, which completes the assertion.
\end{proof}

\subsection{Definition of AGFs through minimizing movements}\label{subsec:MM_AGF}

\textit{Atomic gradient flows (AGFs)} are defined as gradient flows in the metric space $\Omega_L^n$ of the functional $J_n$.  In this subsection, we show the existence of such gradient flows using a minimizing movements approach, also called JKO scheme, cf. \cite{JKO1998}. The minimizing movements setup is also briefly recalled in \Cref{sec:met_th}. Here, considering again $\Omega_n^L$ makes sense to enforce geodesical completeness and compactness, which we will both require in the proofs. We first define the approximation scheme.

\begin{defin}[\textit{Approximation scheme}]
\label{def:discrete_scheme_mms}
Let $\tau>0$ and $(\c^0,\u^0)\in\Omega_L^n$ be a given initial datum. Set $(\c^0_\tau,\u^0_\tau):=(\c^0,\u^0)$ and for every $k\in \N$ choose iteratively
\begin{equation}\label{eq:discrete_minim_scheme}
(\c^{k+1}_\tau,\u^{k+1}_\tau)\in\argmin_{\Omega_L^n}G_{n,\tau}^k\,,
\end{equation}
where the \textit{$\tau$-discretized energy} $G^k_{n,\tau}\colon\Omega_L^n\to\R$ is defined by
\begin{equation}\label{eq:energy_and_dissip_fct}
G_{n,\tau}^k(\c,\u):=J_n(\c,\u)+\frac{1}{2\tau}d^2_n\big((\c,\u),(\c^k_\tau,\u^k_\tau)\big) \quad \forall (\c,\u)\in \Omega_L^n\,.
\end{equation}
\end{defin}

Note that the existence of minimizers of $G^k_{n,\tau}$ follows directly from the compactess of $\Omega_L^n$ and the lower semicontinuity of $J_n$ stated in Lemma \ref{lem:lsc_ex_mins}.

The existence of minimizing movements in our setting can now be obtained by a direct application of \cite[Proposition 2.2.3]{AGS05}. The necessary lower-semicontinuity, coercivity and compactness properties follow by \Cref{lem:lsc_ex_mins}, the fact that $J_n$ is bounded from below and the compactness of $\Omega^L_n$.
%\begin{lem}[Attainment of Minimum for $J_n$]
%\label{prop:minimum_Jn_attained}
%The problem $\inf_{\Omega_L^n}J_{n}$ admits a minimizer.
%\begin{proof}
%Take $(\c_k,\u_k)_{k\in \N}\in \Omega_L^n$ such that
%\begin{align*}
%\lim_{k\rightarrow\infty} J_n(\c_k,\u_k) = \inf_{(c,u)\in\Omega_L^n} J_n(\c,\u).
%\end{align*}
 
 %       Then by  there exists a subsequence $(\c_{k_l},\u_{k_l})$ such that $(\c_{k_l},\u_{k_l}) \rightarrow (\c_*,\u_*)\in \Omega_L^n$ for $l\rightarrow\infty$. Finally, the lower-semicontinuity from Lemma \ref{lem:cptness_lsc} implies that
 %       \begin{align*}
 %           \inf_{(\c,\u)\in\Omega_L^n} J_n(\c,\u) = \liminf_{l\rightarrow\infty} J_n(\c_{k_l},\u_{k_l}) \geq J_n(\c_*,\u_*).
 %       \end{align*}
 %   \end{proof}
%\end{lem}

\noindent

\begin{thm}[\textit{Minimizing movement scheme as limit path for $\tau\to 0$}]
\label{thm:existence_of_GMM}
Let $\tau>0$ and an initial datum $(\c^0,\u^0)\in\textup{Dom}(J_n)$. Define $(\c_\tau,\u_\tau):[0,+\infty)\to\Omega_L^n$, by 
\begin{equation}\label{eq:pw_const_interp}
(\c_\tau,\u_\tau)(0):=(\c^0,\u^0)\,, \ \text{ and  }\ 
(\c_\tau,\u_\tau)(t):= (\c^k_\tau ,\u^k_\tau) \ \  \text{ for } t\in (k\tau,(k+1)\tau] \ \text{and } k\in\mathbb{N}\,,
\end{equation}
where $(\c_\tau^k,\u_\tau^k)$ as in \eqref{eq:discrete_minim_scheme}. There exists a sequence $(\tau_\ell)_{\ell\in \N}\subset (0,1)$ with $\tau_\ell\to 0$ as $\ell \to \infty$, and a curve $(\c,\u)\in AC_{\textup{loc}}^2([0,+\infty);\Omega_L^n)$ such that for all $t\in [0,+\infty)$
\begin{equation}\label{eq:convergence_of_discrete_scheme_1}
d_n((\c_{\tau_\ell,t},\u_{\tau_\ell,t}),(\c_t,\u_t))\to 0\ \ \textup{as}\ \ \ell\to+\infty\,.
\end{equation}
In particular, we have $(\c,\u)(0^+)=(\c^0,\u^0)$ and we call $(\c,\u)$  a \textit{minimizing movement}, see also \Cref{def:minimizing_movements}.
\end{thm}

\begin{rem}
Note that in the comprehensive treatment of the  theory of minimizing movements, detailed in   \cite[Chapter 2]{AGS05}, what we simply called minimizing movements are called \emph{Generalized Minimizing Movements}, denoted by ${\rm GMM}(G,u_0)$ therein. However,  we have decided to simplify the treatment of minimizing movements here for brevity.
\end{rem}

\subsection{$\lambda$-Convexity and curves of maximal slope}
\label{subsec:lambda_cvx_nonlifted}
Next, we study the $\lambda$-convexity of $J_n$ and, whenever this property holds, we derive finer properties of the minimizing movements associated with the AGF.

\begin{rem}[\textit{Scalar and geodesic $\lambda$-convexity}]
We first note the following connection between scalar and geodesic $\lambda$-convexity. Let $(X,d)$ be a geodesic metric space, $f \colon X \rightarrow\R$ and $\gamma \colon [0,1]\to X$ a geodesic. Suppose that $f\circ \gamma\colon [0,1]\to \R$ is $\alpha(\gamma)$\textit{-convex} for some $\alpha(\gamma)\in \R$\,, \ \textit{i.e.}, the map 
\[t\mapsto f(\gamma_t)-\frac{\alpha(\gamma)}{2}t^2\]
is convex. Equivalently, for every $t \in [0,1]$ it holds that
\begin{align*}
f(\gamma_t) \leq (1-t) f(\gamma_0) + t f(\gamma_1) -\frac{\alpha(\gamma)}{2}t(1-t)\,.
\end{align*}
%and the viceversa holds as well.
Comparing the last inequality with the definition of geodesic $\lambda$-convexity of $f$, cf. \eqref{eq:lambda_convexity}, one immediately sees that for scalar $\alpha(\gamma)$-convexity along every geodesic $\gamma$ to imply $\lambda$-convexity of $f$, $\alpha(\gamma)$ needs to be of the form
$\alpha(\gamma) = \lambda d^2(\gamma_0,\gamma_1)$.
\end{rem}

\noindent

To prove the $\lambda$-convexity of $J_n$ for $\lambda \in \mathbb{R}$, we will need an important assumption that asks for a compatibility between the forward operator $K$ and the geodesic structure of $\mathcal{B}$. A consequence of the $\lambda$-convexity will be then that the local slope is a strong upper gradient for $J_n$.
\begin{enumerate}[label=\textnormal{(A\arabic*)}] 
\setcounter{enumi}{7}
\item \label{assmpt:a8} There exists a constant $C:=C(K,\mathcal{B}) > 0$ such that for $m\in \{1,2\}$\,,
\begin{equation}\label{eq:bound_on_geodesics}
\sup_{t\in[0,1]}\left\|\frac{d^m}{dt^m}K(u_t)\right\|_Y \leq\ Cd^m_{\B}(u_0,u_1)\, 
\end{equation}
for all geodesics $u_t:[0,1]\to \B$.
\end{enumerate}
\begin{rem}\label{rem:assumption8}
Note that \eqref{eq:bound_on_geodesics} with $m=1$ implies the existence of $C>0$ such that 
\begin{align}\label{eq:liplike}
\|Ku - Kv\|_Y \leq Cd_{\mathcal{B}}(u,v) \quad \text{for} \quad u,v \in \B\,.
\end{align}
Such Lipschitz-like property of $K$ on $\mathcal{B}$ is inherently of metric nature and can replace \ref{assmpt:a8} in the subsequent Theorem \ref{thm:gradient_equality_many_particle}, if one were to consider curves of maximal slope with regard to so-called weak upper gradients (c.f. \cite[Definition 1.2.2]{AGS05}). However, since \eqref{eq:liplike} alone is not enough to ensure \textit{e.g.}, $\lambda$-convexity of $J_n$, we stick to \ref{assmpt:a8} to simplify the presentation.
\end{rem}
In the following, by $C>0$ we denote a generic constant that depends only on the data (\textit{e.g.} the forward operator $K$, the fidelity term $\mathcal{F}$ and the regularizer $\mathcal{R}$) and whose value is allowed to vary from line to line. The dependence of a constant on a particular parameter will be denoted by a subscript. Moreover, we remark that all derivatives of the forward operator $K$ are intended in a weak sense, \textit{i.e.}, testing against elements in the Hilbert space $Y$. We also recall the definition of $\Omega_L^n$ and $J_n$, \textit{i.e.}, \eqref{eq:Omega_n_def} and \eqref{eq:discrete_version_of_energy}.
\begin{prop}\label{prop:local_convexity}
Assume \ref{assmpt:a1}--\ref{assmpt:a8}
%, and that for the forward operator $K:\M\to Y$ there exists a constant $C:=C(K,\mathcal{B}) > 0$ such that for $m\in \{1,2\}$\,,
%\begin{equation}\label{eq:bound_on_geodesics}
%\sup_{t\in[0,1]}\left\|\frac{d^m}{dt^m}K(u_t)\right\|_Y \leq\ Cd^m_{\B}(u_0,u_1)\, 
%\end{equation}
%for all geodesics $u_t:[0,1]\to \B$. 
Then $J_n$ is $\lambda(L)$-convex on $\Omega_L^n$  for some constant $\lambda(L)\in \R$.
\begin{proof}
Let $(\c_t,\u_t)\colon[0,1]\rightarrow \Omega_L^n$ be a geodesic, cf. \eqref{eq:coordinates_of_c_u}. By Lemma \ref{lem:coord_geodesics}, for every $ j\in\{1,\dots,n\}$, the components 
\[c_t^j=(1-t) c_0^j+t c_1^j\,,\]
and $u_t^j$ are also geodesics in their target, and by the continuity of $K$, \eqref{eq:bound_on_geodesics} also holds for $m=0$.  Recalling \ref{assmpt:a6}--\ref{assmpt:a7}, we denote
\begin{equation*}
\mathcal{R}_n(t) := \RR\left(\frac{1}{n}\sum_{j=1}^n (c^j_t)^2 u_t^j\right)= \frac{1}{n}\sum_{j=1}^n (c_t^j)^2\,.
\end{equation*}
Then, by the Cauchy–Schwarz inequality and \eqref{eq: tau_dissip}, we get
\begin{equation}\label{eq:second_derivative_regularizer}
\frac{d^2}{dt^2} \mathcal{R}_n(t) = \frac{2}{n} \sum_{j=1}^n |c_1^j-c_0^j|^2\leq 2d_n^2((c_0,u_0),(c_1,u_1))\,.
\end{equation}
Next, we compute the derivatives of $K_n:[0,1] \rightarrow Y$, defined as
\begin{equation}\label{eq:K_n_oper}
K_n(t) := \frac{1}{n}\sum_{j=1}^n K((c_t^j)^2u_t^j) = \frac{1}{n}\sum_{j=1}^n (c_t^j)^2 K(u_t^j)\,,
\end{equation}
 namely
\begin{align*}
\frac{d}{dt} K_n(t) &= \frac{1}{n}\sum_{j=1}^n c_t^j\Big(2(c_1^j-c_0^j) K(u_t^j) + c_t^j\frac{d}{dt}K(u_t^j)\Big)\,, \\
\frac{d^2}{dt^2} K_n(t) =& \frac{1}{n} \sum_{j=1}^n\Big[ 2(c_1^j-c_0^j)^2K(u_t^j) +  4 c_t^j(c_1^j-c_0^j)\frac{d}{dt}K(u_t^j) + (c_t^j)^2 \frac{d^2}{dt^2}K(u_t^j)\Big]\,.
\end{align*}
Using Jensen's inequality, that $0\leq c_t^j\leq L$ for every $t\in [0,1]$ and $j\in\{1,\dots,n\}$, \ref{assmpt:a8} and the fact that $K$ is even strong-to-strong continuous (cf. \ref{assmpt:a1} and \cite[Remark 3.2]{bredies2013inverse}), we estimate 
\begin{align}\label{eq:first_order_bd_kn}
\begin{split}
\sup_{t\in[0,1]} \Big\|\frac{d}{dt} K_n(t)\Big\|_Y& \leq C_{ n,L}\sum_{j=1}^n \big(|c_1^j-c_0^j|+d_{\mathcal{B}}(u_0^j,u_1^j)\big) \\ 
&\leq C_{{ n,L}}\Big(\sum_{j=1}^n \big(|c_1^j-c_0^j|^2+d^2_{\mathcal{B}}(u_0^j,u_1^j)\Big)^{\frac{1}{2}}= C_{{ n,L}}\,d_n((c_0,u_0),(c_1,u_1))\,, 
\end{split}
\end{align}
and analogously,
\begin{align}\label{eq:K_n_bounds}
\begin{split}
\sup_{t\in [0,1]} \left\|\frac{d^2}{dt^2} K_n(t)\right\|_Y &\leq\  C_{ n,L}\sum_{j=1}^n\Big((c_1^j-c_0^j)^2+(c_1^j-c_0^j)d_{\mathcal{B}}(u_0^j,u_1^j)+d^2_{\mathcal{B}}(u_0^j,u_1^j)\Big)  \\
&\leq C_{ n,L}\sum_{j=1}^n \left(|c_1^j-c_0^j|^2+d_\B^2(u_0^j,u_1^j)\right)=C_{ n,L}d_n^2((c_0,u_0),(c_1,u_1))\,.
\end{split}
\end{align}
Finally, by the chain rule, we compute  
\begin{align*}
\frac{d^2}{dt^2} J_n(\c_t,\u_t) &  =  \frac{d}{dt} \big(\nabla \F(K_n(t)),\frac{d}{dt} K_n(t)\big)_Y  + \frac{d^2}{dt^2} \mathcal{R}_n(t)\\
&= \nabla^2\F(K_n(t))\Big[\frac{d}{dt}K_n(t),\frac{d}{dt} K_n(t)\Big]_Y +\Big(\nabla \F(K_n(t)),\frac{d^2}{dt^2} K_n(t)\Big)_Y + \frac{d^2}{dt^2} \mathcal{R}_n(t)\,,
\end{align*}
which together with \eqref{eq:second_derivative_regularizer}, \eqref{eq:first_order_bd_kn} and \eqref{eq:K_n_bounds} yields
\begin{align*}
\sup_{t\in [0,1]}\left|\frac{d^2}{dt^2} J_n(\c_t,\u_t)\right| \leq&\,
 2C_{n,L} \|\F\|_{C^{2}}\, d^2_n((c_0,u_0),(c_1,u_1))+ 2d_n^2((c_0,u_0),(c_1,u_1))\\
    \leq&\, C_{n,L,\F}\, d_n^2((c_0,u_0),(c_1,u_1))\,.
\end{align*}
Hence, there exists a $\lambda:=\lambda(n,L,K,\B,\F)<0 $ such that 
\[J_n(\c_t,\u_t)\colon[0,1] \rightarrow  \R_+\ \text{ is } \ \lambda\cdot d_n^2((c_0,u_0),(c_1,u_1))\text{-convex}\,,\] 
implying that $J_n : \Omega_L^n \rightarrow \R_+$ is $\lambda$-convex along all  unit-speed geodesics $(\c_t,\u_t)$.
\end{proof}
\end{prop}

%\begin{rem}\label{eq:interpretation_of_bounds}
%Assumption \eqref{eq:bound_on_geodesics} asks for a suitable compatibility between the forward operator $K$ and the geodesic structure on $\B$ (see also Assumption \ref{assmpt:a8} for comparison). 
%\end{rem}

\begin{example}[$\lambda$-convexity for convolutions of measures] 
In this example, we show that \ref{assmpt:a8} is fulfilled for standard convolution operators in the space of measures. Similar reasoning applies also in more general settings.
Let $\mathbb{T}$ be the unit circle, which we implicitly identify with $\R/\mathbb{Z}$,
and let us consider the linear operator $K: M(\mathbb{T}) \rightarrow L^2(\mathbb{T})$,  where 
\begin{equation}\label{ex:K_mu}
K \mu(s):= \int_{\mathbb{T}} k(s-x)\, d \mu(x) \quad \forall \mu \in M(\mathbb{T})\,,
\end{equation}
with a suitable convolution kernel $k\in C^2(\mathbb{T})$. Here, $M(\mathbb{T})$ is the space of Radon measures on $\mathbb{T}$.  Choosing as regularizer the total variation of measures $\mathcal{R}(\cdot) := \|\cdot\|_{TV}$ , it holds that 
\[\tilde{\mathcal{B}} = \{\pm \delta_x : x\in \mathbb{T}\}\,.\] 
Then take as the weakly*-closed, geodesically complete set $\mathcal{B} := \{\delta_x : x\in \mathbb{T}\}$, and as $d_{\mathcal{B}}$ the standard $2$-Wasserstein distance $\mathcal{W}_2$, which metrizes the weak*-convergence on $\tilde{\mathcal{B}}$.
Note that the unit-speed $\mathcal{W}_2$–geodesic joining $\delta_{x_0}$ to $\delta_{x_1}$ is $t \mapsto \delta_{x_t}$ where $x_t=(1-t) x_0+t x_1$, since we implicitly identified $\mathbb{T}$ with $\R/\mathbb{Z}$.
Along this geodesic we have that for every $s \in \mathbb{T}$,
\begin{equation*}
K\delta_{x_t}(s)=\int_{\mathbb{T}} k(s-x)\, d \delta_{x_t}(x)=k(s-x_t)\,,
\end{equation*}
which is twice differentiable in $t$. In particular, by the chain rule, we obtain
\begin{equation*}
 \frac{d}{dt} K\delta_{x_t}(s)=k'(s-x_t)(x_0-x_1)\,, \quad    \frac{d^2}{d t^2} K\delta_{x_t}(s)=k''(s-x_t)(x_0-x_1)^2\,.
\end{equation*}
Thus, the map $t\mapsto K\delta_{x_t}$ satisfies \eqref{eq:bound_on_geodesics}.
\end{example}

\begin{example}[Necessity of local convexity]
In this example we show that Proposition \ref{prop:local_convexity} cannot, in general, be improved to a global
$\lambda$–convexity statement on $\R_+\times\B$.
 Let $\M:=M([0,1])$ and $\mathcal{R}(\cdot):=\|\cdot\|_{TV}$, and set again $\B:=\{\delta_x:\ x\in[0,1]\}$. We define the linear map $K :\mathcal{M} \rightarrow \mathbb{R}$ as 
\begin{align*}
K\mu := \int_0^1 y \,d \mu(y)\,,    \end{align*}
and consider $\F(y) := y^2$ and $n=1$. 
If  $(c,\mu)\in \R_+\times \B$,  with $ \mu= \delta_x$ for some $x\in [0,1]$, then
\begin{align}
J_1(c, \mu)=\F\!\big(K(c^2\delta_x)\big)+c^2
=\F(c^2 x)+c^2=(c^2 x)^2+c^2=c^4x^2+c^2.
\end{align}
With the $\ell_2$-product metric, geodesics in $\R_+\times\B$ have the form 
\[(c_t,u_t) = \left((1-t)c_0+tc_1, \delta_{(1-t)x_0+tx_1}\right)\,.\]
Along such geodesics we may view $J_1$ as the function
$f(c,x):=c^4 x^2+c^2$ on $\R_+\times[0,1]$. Its Hessian at $(c,x)$ becomes
\begin{align}
\nabla^2 f(c,x)
=
\begin{pmatrix}
12c^2x^2+2\, \ & 8c^3x\\[2pt]
8c^3x\, \ & 2c^4\,,
\end{pmatrix}
\end{align}
so that
\begin{align}
\det  \nabla^2 f(c,x)=\bigl(12c^2x^2+2\bigr)\cdot 2c^4-(8c^3x)^2
=4c^4\bigl(1-10c^2x^2\bigr)\,.
\end{align}
Hence, for every fixed $x\in(0,1]$ and all sufficiently large $c$, one has $\det  \nabla^2 f(c,x)<0$,
so $\nabla^2 f(c,x)$ has a negative eigenvalue. In particular, if $\lambda_{\min}(c,x)$ denotes the minimal eigenvalue of $\nabla^2f(c,x)$, then
$\lambda_{\min}(c,x)\to-\infty$ as $c\to\infty$. Therefore, $J_1$ is \emph{only locally} semiconvex $\R_+\times \B$, \textit{i.e.}, semiconvex on $\Omega_L$ for all $ L>0$.
\end{example}

Proposition \ref{prop:local_convexity} yields as a first consequence that the local slope is a strong upper gradient according to Definition \ref{def:Strong_upper_gradient}, even with non-compact weights, so on
\begin{equation}\label{eq:Omega_infty_B}
\Omega^n_\infty:=[0,\infty)^n\times \B^n\,.
\end{equation}
Recall from \Cref{def:local_slope} that the slope is a local concept, hence there is no need to write, \textit{e.g.}, $\left|\partial J_n|_{\Omega_L^n}\right| $ instead of just always $|\partial J_n|$. \\
\begin{lem}\label{lem:sugJn}
Assume \ref{assmpt:a1}-\ref{assmpt:a5} and that the forward operator $K:\M\to Y$ fulfills \ref{assmpt:a8}. Then, the local slope $|\partial J_n|$ is a strong upper gradient for $J_n$ on $\Omega^n_\infty %[0,\infty)^n\times \B^n
$.
\begin{proof}
By Proposition \ref{prop:local_convexity}, $J_n$ is $\lambda(L)$-geodesically convex, hence by \cite[Corollary 2.4.10]{AGS05} the metric slope 
$|\partial J_n|$ is a strong upper gradient for $J_n$ on $\Omega_L^n$.
Now, let $\mathbf{v}\in AC([0,1]; \Omega_\infty^n)$ and take $L>0$ large enough so that $\mathbf{v}_t \in \Omega_L^{n}$ for every $t\in[0,1]$. 
%Note that such $L$ exists since $\mathbf{v}$ admits a continuous extension to $[a,b]$ \BBB for every $0\leq a<b<+\infty$, \EEE as $\B^n$ is complete.
Then, for all $0\leq s\leq t\leq 1$, locality of the metric slope gives
\begin{align*}
|J_n(\mathbf{v}_t) - J_n(\mathbf{v}_s)|
\leq \int_s^t |\partial J_n|(\mathbf{v}_r)\,|\mathbf{v}'_r|\,\mathrm{d}r,
\end{align*}
since $|\partial J_n|$ is a strong upper gradient for $J_n$ on $\Omega_L^n$. Thus $|\partial J_n|$ is also a strong upper gradient for $J_n$ on $\Omega_\infty^n$.
\end{proof}
\end{lem}
\noindent

This gives the existence of curves of maximal slope, again without the need for compactness, for which we recall \Cref{thm:existence_of_GMM} and \eqref{eq:Omega_infty_B}\,.
\begin{thm}
\label{thm:existence_of_gfs_nonlifted}
Assume \ref{assmpt:a1}-\ref{assmpt:a8}. Then every minimizing movement $(\c_t,\u_t)$ in $\Omega_\infty^n$ is a curve of maximal slope for $J_n$ with regard to $|\partial J_n|$. 
In addition, minimizing movements in $\Omega^n_L$ are also curves of maximal slope.
    \begin{proof}
Let $(\c^0,\u^0) \in \Omega_\infty^n$, and let $(\c_t,\u_t)$ be a minimizing movement for $J_n$ in $\Omega_\infty^n$. First, assume that there exists $L>0$  such that  $c_t^j \in [0,L]$ for all $j\in \{1,\dots,n\}$ and $t\geq 0$. 
 
By \Cref{thm:existence_of_GMM}, there exist discrete solutions $(\c_{\tau_\ell,t},\u_{\tau_\ell, t}) \in \Omega^n_{\infty}$ such that
\begin{equation}\label{eq:l_pw_constant_interpolation}
d_n((\c_{\tau_\ell,t},\u_{\tau_\ell,t}),(\c_t,\u_t))\to 0\ \ \textup{as}\ \ \ell\to+\infty\,.
\end{equation}
Since $c_t^j \leq L$, for $\ell$ big enough, $(\c_{\tau_\ell,t},\u_{\tau_\ell,t}) \in \Omega_{\tilde L}^{n}$ for $\tilde L > L$.
This implies that the limit is a minimizing movement in $\Omega_{\tilde L}^n.$
Now using \cite[Corollary 2.4.11]{AGS05} together with Proposition \ref{prop:local_convexity} gives that $(\c_t,\u_t)$ is also a curve of maximal slope for $J_n|_{\Omega_{\tilde L}^{n}}$ with regard to $|\partial J_n|_{\Omega_{\tilde L}^{n}}|$. Since all of the involved terms are local, $(\c_t,\u_t)$ is a curve of maximal slope for $J_n$ with regard to $|\partial J_n|$.\\
So the only thing left to prove is that for each minimizing movement $(\c_t,\u_t)$ in $\Omega_\infty^n$, there exists some $L>0$ such that  $c_t^j \in[0,L]$ for all $j\in \{1,\dots,n\}$  and $t\geq 0$.  By \ref{assmpt:a6}-\ref{assmpt:a7}, and as $ \mathcal{F}$ is bounded from below by \ref{assmpt:a2}, it holds that 
\begin{equation}\label{eq:c_t_j_bdd_1}
\frac{1}{n}\sum_{j=1}^n (c_t^j)^2  =\mathcal{R}\Big(\frac{1}{n}\sum_{j=1}^n (c_t^j)^2u_t^j\Big)= J_n(\c_t,\u_t) -  \mathcal{F}\left(\frac{1}{n} \sum_{j=1}^n (c_t^j)^2Ku_t^j\right) \leq J_n(\c_t,\u_t) + C(F)\,.
\end{equation}
As for every $\tau>0$, \eqref{eq:discrete_minim_scheme} 
and \eqref{eq:energy_and_dissip_fct} yield 
\begin{align*}
J_n(\c_\tau^1,\u_\tau^1) \leq J_n(\c_\tau^{1},\u_\tau^{1}) + \frac{1}{2\tau}d_n^2((
\c_\tau^{1},\u_\tau^{1})),(\c^0,\u^0)) \leq J_n(\c^0,\u^0)\,,
\end{align*}
and thus iteratively for all $k\in \N$
\begin{align}\label{eq:monot}
J_n(\c_\tau^k,\u_\tau^k) \leq J_n(\c_\tau^{k-1},\u_\tau^{k-1}) \leq J_n(\c^0,\u^0)\,.
\end{align}
In view of \eqref{eq:pw_const_interp}, \eqref{eq:l_pw_constant_interpolation} (taking the limit $\ell \rightarrow \infty$) and the lower semicontinuity of $J_n$ with respect to $d_n$ imply that $J_n(\c_t,\u_t) \leq J_n(\c^0,\u^0)$, which together with \eqref{eq:c_t_j_bdd_1} and the fact that $n$ is fixed imply that the $c_t^j$ stay indeed uniformly bounded for all $j\in \{1,\dots,n\}$ and $t\geq 0$. 
\end{proof}
\end{thm}

\subsection{NPC of the extremal points and uniqueness of the flow}
\label{subsec:npc_nonlifted}
We now turn to the issue of uniqueness of the AGF with respect to the functional $J_n$ introduced in \eqref{eq:discrete_version_of_energy}. Here, we show that if the metric space $\mathcal{B}$ is non-positively curved (NPC), cf. Definition \ref{def:NPC}, we can ensure such a uniqueness property. To this end, in \Cref{sec:met_th} we recall a local uniqueness result, see \Cref{thm:uniqueness+regularity_mm_in_npc_case}, which we extend below to a global one. We remind the reader that uniqueness of gradient flows on NPC metric spaces was extensively studied in \cite{gradient_flows_npc}, and among other works, this analysis was further refined in \cite[Chapter 4]{AGS05}.

First, it is well-known that $\B$ being NPC implies that $\Omega_L^n$ is as well. This is essentially a consequence of \Cref{lem:coord_geodesics}, and of choosing the $\ell_2$-metric on the product space $\Omega_L^n$.

\begin{lem}[NPC of product space]
\label{lem:npc_of_product}
If $(\B,d_\B)$ is a space of NPC according to \Cref{def:NPC}, then so is $(\Omega_L^n,d_n)$.
\end{lem}
%\begin{proof}
%First, $(\Omega_L^n,d_n)$ is a geodesic metric space by \Cref{lem:coord_geodesics}, which also implies that given a geodesic $(\c_t,\u_t)\in \Omega_L^n$, also $ c_t^j\colon [0,1] \to [0,L] $ and $u_t^j\colon [0,1] \to \B$ are geodesics for every $j\in \{1, \dots, n\}$.  \\
%Thus for $(\tilde{c},\tilde{u})\in \Omega_L^n$, the NPC property of $([0,L], |\cdot|)$ and of $(\B,d_\B)$ imply that
%\begin{align*}
%\label{eq:existence_geodesic}
%d_n^2((c_t, u_t),(\tilde{c},\tilde{u})) =&\ \sum_{j=1}^n \Big(|c_t^j-\tilde{c}^{j}|^2 + d_\B^2( u_t^j,\tilde{u}^j)\Big) \\
%\leq& \ \sum_{j=1}^n (1-t)\left(|c_0^j-\tilde{c}^j|^2 + d_\B^2(u_0^j,\tilde{u}^j)\right) + \sum_{j=1}^n t\left(|c_1^j-\tilde{c}^j|^2 + d_\B^2( u_1^j,\tilde{u}^j)\right)\\ 
%&\ - \sum_{j=1}^n t(1-t)\left(|c_0^j-c_1^j|^2 + d_\B^2(u_0^j, u_1^j)\right) \\
%=&\ (1-t)d_n^2(( c_0,u_0),(\tilde{c},\tilde{u})) + td_n^2(( c_1,u_1),(\tilde{c},\tilde{u}))- t(1-t) d_n^2(( c_0,u_0),( c_1,u_1))\,,
%\end{align*}
%so $d_n$ fulfills \eqref{eq:convexity_of_dist_along_curve}, and thus also $\Omega_L^n$ is NPC.
%\end{proof}
We are then ready to show uniqueness, provided that $(\B,d_\B)$  is NPC. Uniqueness also comes with contraction estimates between minimizing movements with different initial points. 
\begin{thm}\label{thm: uniqueness}
\label{cor:global_uniqueness_gf_non_lifted}
Assume \ref{assmpt:a1}-\ref{assmpt:a8} and that $(\B,d_\B)$ is NPC. Then, for every initial point $(\c^0,\u^0)\in \Omega_\infty^n$, there exists a unique minimizing movement for $J_n$. Moreover, for every other initial point  $(\tilde \c^0, \tilde \u^0)\in \Omega_\infty^n$, and corresponding minimizing movement $(\tilde{\c}_t,\tilde{\u}_t)$ the following contraction estimate holds:
\begin{equation}\label{eq:contract_prop_bulk}
d_n\big((\c_t,\u_t),(\tilde{\c}_t,\tilde{\u}_t)\big) \leqslant  e^{-\lambda t}d_n\big((\c^0,\u^0),(\tilde{\c}^0,\tilde{\u}^0)\big) \quad \text{for } \mathcal{L}^1\text{-a.e. }t>0\,,
\end{equation}
for a $\lambda\in \R$ depending only on $|\c^0|$ and $|\tilde{\c}^0|$.
\end{thm}
\begin{proof}
Again by \ref{assmpt:a6}-\ref{assmpt:a7}, and as $\F$ is bounded from below by a constant $-C(\mathcal{F})$, it holds that  
\begin{equation*}
J_n(\c,\u) = \F\left( \frac{1}{n}\sum_{j=1}^n(c^j)^2K(u^j)\right) +  \frac{1}{n}\sum_{j=1}^n (c^j)^2 \geq  \frac{1}{n}\sum_{j=1}^n(c^j)^2 - C(\mathcal{F})\,,
\end{equation*}
\textit{i.e.}, 
\[\frac{1}{n}\sum_{j=1}^n (c^j)^2 \leq J_n(\c,\u) + C(\mathcal{F})\,.\] 
Assume that $(\c_t,\u_t)$ is a minimizing movement with initial point $(\c^0,\u^0)$. Then by the previous estimate and \eqref{eq:monot}, we have that $(\c_t,\u_t) \in \Omega_L^{n}$ for some $L=L(n,(\c^0,\u^0))$ large  enough. Since $J_n$ is $\lambda$-convex on $\Omega_n^L$ by Proposition \ref{prop:local_convexity}, minimizing movements in $\Omega_L^{n}$ exist, and fulfill the contractivity property \eqref{eq:contract_prop_bulk} by \Cref{thm:uniqueness+regularity_mm_in_npc_case}, cf. \eqref{eq:contract_prop}.
\end{proof}

\section{The lifted problem}\label{sec:3_lifted}
To further our analysis, we follow a lifting approach to the Wasserstein space, detailed in \cite{BCFW23}. This leads to a comparison  of Atomic Gradient Flows in $\Omega_L^n$ versus metric gradient flows in $ \mathcal{P}_2(\Omega_L^n)$ in \Cref{sec:equivalence_gfs}, which can be thought in spirit similar to the investigation of Chiz\'at and Bach in \cite{chizat2018global}. A key difference here is that lifted particles are Dirac deltas supported on elements of $\Omega_L^n$. In particular, for $TV$-regularized problems, where extremal points are Dirac deltas, lifted particles would be elements of the form $\delta_{\delta_x}$. However, it is clear that this representation of lifted particles is isometric to $\delta_x$, so that the lifted formulation remains fully consistent with the perspective of Chiz\'at and Bach (more details will be given in Section \ref{sec:examples}).   \\
Note that to keep the lifted problem equivalent to the original problem \eqref{eq:P_u} on $\M$, (only) in this section one has to work again with the full set of extremal points $\tilde{\mathcal{B}}$ of \eqref{eq:extreme_ball}. 
%\begin{equation*}
%\tilde{\B} = \overline{\mathrm{Ext}(\{u\in \M:\mathcal{R}(u)\leq 1\})}^\ast.
%\end{equation*}

\begin{defin}[\textit{Lifting in the space of positive measures}]\label{def:lifting_in_measures}
We consider the lifting of \eqref{eq:P_u} to the space of positive measures on $\tilde{\B}$ as follows:
\begin{equation}\label{eq:P_mu}
\inf_{\mu \in M^+(\tilde{\B})} j(\mu)\,, \quad \text{where} \ \ j(\mu):=\F(\K\mu)+ \|\mu\|_{TV}\,,
\end{equation}
where $M^+(\tilde{\B})$ is the cone of positive (Radon) measures on $\tilde{\B}$. The forward operator $\K:M^+(\tilde{\B})\to Y$ is set as 
\begin{equation}\label{eq:curly_K}
\K\mu:=K\I(\mu)\,,
\end{equation}
where the representative $\I(\mu)\in \M$ is defined via the following version of \textit{Choquet's theorem}.
\end{defin}
\begin{prop}[\textit{Version of Choquet's theorem}, {\cite[Proposition 5.2]{BCFW23} \& \cite[page 14]{P01}}]
\label{Choquet}
Every measure $\mu\in M^+(\tilde{\B})$ defines the linear action of some $\I(\mu)\in\textup{Dom}(\mathcal{R})$ via duality, namely, 
\begin{equation}\label{eq:Choquet}
\langle \I(\mu), p\rangle=\int_{\tilde{\B}}\langle v,p\rangle\,d\mu(v)\,\ \ \forall p\in \C\,.
\end{equation}
Furthermore, the map $\I:M^+(\tilde{\B})\to\textup{Dom}(\mathcal{R})\subset \mathcal{M}$ is a linear surjection and $\I(\mu)$ is called the (weak-) barycenter of $\mu$.
\end{prop}
For $\alpha >0$, we also write $M^+_\alpha(\tilde{\B}) := \{\mu\in M^+(\tilde{\B}) : \|\mu\|_{TV}\leq \alpha\}$. We can now state the equivalence between the optimization of the lifted problem and the original one.

\begin{prop}\label{prop:equivalence_of_problems}
By \cite[page 4 and Theorem 5.4]{BCFW23}, we have the equivalence
\begin{align}\label{eq:equiv_of_both_problems_nonbounded}
\min_{u\in \M}J(u)=\min_{\mu\in M^+(\tilde{\B})} j(\mu)\,.
\end{align}
In addition, there exists an $\alpha = \alpha(J) >0$ such that
\begin{align}\label{eq:equiv_of_both_problems}
\min_{u\in \M}J(u)=\min_{\mu\in M^+_\alpha(\tilde{\B})} j(\mu)\,.
\end{align}
\begin{proof}
The statement in \eqref{eq:equiv_of_both_problems_nonbounded} is already proved in \cite{BCFW23}, while for   \eqref{eq:equiv_of_both_problems} it suffices to consider a minimizer $u_*\in \underset{u\in \M}{\argmin} J(u)$, and  $\alpha > \mathcal{R}(u_\ast)$.
\end{proof}
\end{prop}

\subsection{Lifting the problem in Wasserstein space}
\label{sec:lifting_wasserstein_space}
Recalling \eqref{eq:Omega_n_def} and the subsequent notation, let $\tilde \Omega_L:=[0,L]\times \tilde{\mathcal{B}}$
for a fixed $L> \sqrt{\alpha}$, with $\alpha>0$ as in \Cref{prop:equivalence_of_problems}. Denote by $\mathcal{P}_2(\tilde \Omega_L)$ the space of probability measures in $\tilde \Omega_L$ with finite second moment endowed with the Wasserstein-2 metric. 
In particular, since the metric $d_{\tilde \B}$ is bounded, we can write
\begin{equation}\label{eq:def_P_2_Wasserstein}
\mathcal{P}_2(\tilde \Omega_L):=\Big\{\nu \in M^+(\tilde \Omega_L)\colon \nu(\tilde \Omega_L)=1\Big\}\,.
\end{equation}
For every $\nu_1,\nu_2\in \mathcal{P}_2(\tilde \Omega_L)$, their Wasserstein 2-distance is defined as 
\begin{equation}\label{eq:Wasserstein_2}
\W_2(\nu_1,\nu_2):=\sqrt{\inf_{\gamma\in \Gamma_{\nu_1,\nu_2}}\int_{ \tilde \Omega_L\times  \tilde \Omega_L}d^2_{\Omega}(\omega_1,\omega_2)\,d\gamma(\omega_1,\omega_2)}\,,
\end{equation}
where 
\begin{equation}\label{eq:Gamma}
\Gamma_{\nu_1,\nu_2}:=\{\gamma\in \mathcal{P}_2( \tilde \Omega_L\times  \tilde \Omega_L)\colon (\pi_1)_\#(\gamma)=\nu_1, (\pi_2)_{\#}(\gamma)=\nu_2 \}\,.
\end{equation}
In \eqref{eq:Gamma}, for $i=1,2$, we have denoted by $\pi_i: \tilde \Omega_L\times \tilde \Omega_L\to \tilde \Omega_L$ the $i$-th coordinate projection. As a reminder,  cf. \eqref{eq:d_omega_def}, the distance $d_{\Omega}: \tilde \Omega_L\to \R_+$ is defined as follows. For $\omega_i:=(c_i,u_i)\in  \tilde \Omega_L$, we set 
\begin{equation}\label{eq:d_Omega}
d_{\Omega}(\omega_1,\omega_2):=\left(|c_1-c_2|^2+d^2_{ \tilde\B }(u_1,u_2)\right)^{1/2}\,.
\end{equation}
The infimum in \eqref{eq:Wasserstein_2} is actually a minimum, cf. \cite[Section 7.1]{AGS05},  and the set of all possible corresponding minimizers will be denoted by $\Gamma_0(\nu_1,\nu_2)$. 
\noindent
We define the homogeneous projection operator $\pi:\mathcal{P}_2( \tilde \Omega_L)\to M^+( \tilde\B )$ via 
\begin{equation}\label{eq:pi_nu}
\int_{ \tilde\B }\psi(u)\,d[\pi\nu](u):=\int_{ \tilde \Omega_L}c^2\psi(u)\,d\nu(c,u)\,\quad \forall \nu\in \mathcal{P}_2( \tilde \Omega_L)\,, \psi\in C(\tilde\B)\,.
\end{equation}
\noindent Note that the homogeneous projection is a typical tool for addressing the unbalanced formulation of Optimal Transport and is also one way to define the Hellinger-Kantorovich distance \cite{liero2018optimal}.
Thanks to this homogeneous projection we can then transform the problem in the right hand side of \eqref{eq:P_mu} into a minimization problem in $\mathcal{P}_2( \tilde \Omega_L)$ in the following way.
\begin{defin}[\textit{Lifting in the space of probability measures}]
We consider the problem
\begin{equation}\label{eq: P_Wasserstein}
 \inf_{\nu \in \mathcal{P}_2(\tilde \Omega_L)} \mathcal{J} (\nu)\,, \ \text{where}  \ \mathcal{J} (\nu):=j(\pi\nu)\,.
\end{equation}
\end{defin}
\noindent Then, the functional $\mathcal{J}$ can be written more explicitly as
\begin{equation*}
\mathcal{J}(\nu) = \F(\K[\pi \nu]) + \int_{{\tilde \Omega_L} } c^2\, d\nu(c,u)\,.
\end{equation*}
Using \eqref{eq:curly_K}, \eqref{eq:Choquet},  and \ \eqref{eq:pi_nu}, we can also check that 
\begin{equation*}
\F(\K[\pi \nu]) =\F(K\I[\pi\nu])= \F\left(\int_{\tilde \Omega_L} c^2 Ku \, d \nu(c,u) \right)\,,
\end{equation*}
so that in total
\begin{align}\label{eq:I_def}
\mathcal{J}(\nu) =  \F\left(\int_{ \tilde \Omega_L} c^2 Ku \,d \nu(c,u) \right) + \int_{\tilde \Omega_L} c^2\,d\nu(c,u)\,.
\end{align}
\begin{rem}
The definitions above can be made analogously with $\B$ instead of $\tilde{\B}$, but as mentioned before, the equivalence of the problems that we show in the following only holds with $\tilde{\B}$ in general.
\end{rem}
Next, we show that problems \eqref{eq:P_mu} and \eqref{eq: P_Wasserstein} are equivalent. We first need to prove that the projection operator $\pi$ is surjective.

\begin{lem}\label{lem:surjectivity_of_pi} For every $\alpha>0$ and $L>\sqrt{\alpha}$, we have that the homogeneous projection operator  $\pi: \mathcal{P}_2(\tilde \Omega_L)\to M_\alpha^+(\tilde\B)$ is surjective.
\end{lem}
\begin{proof}
Consider a measure $\mu \in M_\alpha^+(\tilde\B )$. Note that if $\mu = 0$, then $\pi \nu_0 = \mu$ for every $\nu_0$ of the form $\nu_0 = \delta_0\otimes \delta_{u}$, and for any $u \in  \tilde\B $. Therefore, in what follows we will suppose that $\mu( \tilde\B ) > 0$.
Define now 
\begin{equation}\label{eq:measure_nu_mu}
\nu_{\mu}:= \delta_{\sqrt{\mu( \tilde\B )}} \otimes \frac{\mu}{\mu(\tilde\B)} \in M^+_\alpha(\tilde \Omega_L)
\,, 
\end{equation}
since $L > \sqrt{\alpha}$. Note that for every test function $\varphi \in C(\tilde \Omega_L)$,  we have
\begin{align*}
\nu_\mu(\varphi) := \int_{ \tilde\Omega_L} \varphi(c,u) \, d \delta_{\sqrt{\mu( \tilde\B )}}(c) \, d \Big(\frac{\mu}{\mu( \tilde\B )}\Big)(u) = \int_{\tilde \B} \varphi(\sqrt{\mu( \tilde\B )},u)\, d \big(\frac{\mu}{\mu( \tilde\B )}\big)(u)\,.
\end{align*}
Testing the above definition with the constant function $\varphi \equiv 1$, it is clear that the measure $\nu_\mu$ as defined in \eqref{eq:measure_nu_mu} is a probability measure. 
Finally, let us verify that $\pi \nu_\mu = \mu$. For every $\psi \in C( \tilde\B)$ it holds that 
\begin{align*}
\pi\nu_\mu(\psi)&=\int_{ \tilde\B }\psi(u)\,d[\pi\nu_\mu](u) =\int_{ \tilde \Omega_L}c^2\psi(u)\,d\nu_\mu(c,u)\\
&= \int_{[0,L]\times \tilde\B }c^2\psi(u)\,d\delta_{\sqrt{\mu( \tilde\B )}}(c)\,d\big(\frac{\mu}{\mu( \tilde \B)}\big)(u)=\int_{ \tilde\B} \psi(u)\, d \mu(u) = \mu(\psi)\,,
\end{align*}
showing that $\pi\nu_\mu=\mu$.
\end{proof}

\begin{prop}
 For every $\alpha>0$ and $L>\sqrt{\alpha}$,  it holds that 
\begin{align}\label{eq:lifted_and_Wasserstein_equiv}
\min_{\nu \in \mathcal{P}_2( \tilde \Omega_L)} \mathcal{J} (\nu) = \min_{\mu \in M_\alpha^+( \tilde\B )} j(\mu)\,.
 \end{align}
Moreover, if $\nu \in \mathcal{P}_2( \tilde \Omega_L)$ minimizes $ \mathcal{J} $, then $\pi \nu \in M^+_{ \alpha}( \tilde \B)$ minimizes $j$. Conversely, if $\mu  \in M^+_{ \alpha }( \tilde\B )$ minimizes $j$, then every $\nu \in \mathcal{P}_2( \tilde \Omega_L)$ such that $\pi \nu  = \mu$ minimizes $\mathcal{J}$.
\end{prop}

\begin{proof}
 The existence of minimizers in both problems follows by the structural assumptions \ref{assmpt:a1}-\ref{assmpt:a6} and a standard application of the direct method in the Calculus of Variations. The equivalence of the two minimization problems follows immediately from the surjectivity of $\pi$ and the definition of $\mathcal{J}$ in \eqref{eq: P_Wasserstein}. Indeed, let $\nu\in \mathcal{P}_2(\tilde\Omega_L)$ be a minimizer of $ \mathcal{J}$.
 Then, for every $\mu\in M^+_{\alpha}(\tilde\B)$, using Lemma \ref{lem:surjectivity_of_pi}, there exists $\nu_\mu\in \mathcal{P}_2(\tilde \Omega_L)$ such that $\pi \nu_\mu=\mu$, so that by \eqref{eq: P_Wasserstein}, we have
 \[j(\mu)=j(\pi\nu_\mu)= \mathcal{J} (\nu_\mu)\geqslant \mathcal{J}  (\nu)=j(\pi \nu)\,,\]
 \textit{i.e.}, $\pi\nu\in M^+_{\alpha}(\tilde\B)$ is a minimizer for $j$.  The converse follows analogously, showing \eqref{eq:lifted_and_Wasserstein_equiv}. \qedhere
 \end{proof}

\subsection{Minimizing movements for the lifted problem}\label{sec:minmovlifted}
Following the previous discussion and the approach in Subsection \ref{subsec:discretization_MM}, we fix again a closed (so compact), geodesically complete subset $\B\subset \tilde \B$, and recalling \eqref{eq:Omega_n_def}-\eqref{eq:d_omega_def}, we consider
\begin{equation*}
    \Omega_L = [0,L]\times\B\,,
\end{equation*}
and the lifted functional $\mathcal{J}$ on $\mathcal{P}_2(\Omega_L)$, which itself a compact and geodesically connected subset of $\mathcal{P}_2(\tilde \Omega_L)$. We recall that, having restricted to $\Omega_L$, the functional  $\mathcal{J}$ is defined as
\begin{align}\label{eq:repr_I}
\mathcal{J}(\nu) :=   \mathcal{F}\left(\int_{\Omega_L} c^2 Ku \,d \nu(c,u) \right) + \int_{\Omega_L} c^2\, d\nu(c,u)\,.
\end{align}
Again, we use the JKO scheme to show existence of mimizing movements. We note that similarly as in Subsection \ref{subsec:discretization_MM}, we can only hope to recover an approximate minimum of $J$ with $\mathcal{J}$ %on $\Omega_L$ 
in $\mathrm{conv}\,\B$, but here one needs to restrict the support of our measures to $\B$ to obtain existence of, \textit{e.g.},  curves of maximal slope.
\begin{defin}[\textit{Approximation scheme}]
\label{def:discrete_scheme_mms}
Let $\tau>0$ and $\nu^0\in \mathcal{P}_2(\Omega_L)$ be a given initialization. Set $\nu_\tau^0:=\nu^0$, and choose (iteratively) for all $k\in\mathbb{N}_0$,
\begin{align}\label{eq:dsc2}
\nu^{k+1}_\tau \in \underset{\nu \in \mathcal{P}_2(\Omega_L)}{\rm argmin} \left\{\mathcal{J}(\nu) + \frac{1}{2\tau}\W_2^2(\nu,  \nu^k_\tau)\right\}.
\end{align}
Define then the piecewise constant curve $\nu_\tau : [0,+\infty) \rightarrow \mathcal{P}_2(\Omega_L)$ by $\nu^0_\tau = \nu^0$ and for every $k \in \mathbb{N}_0$,
\begin{align}\label{eq:nu_tau_step}
\nu_{\tau,t} := \nu^k_\tau\ \ \text{ for } t \in ((k-1)\tau, k\tau]\,. 
\end{align}
\end{defin}
\noindent
The next theorem shows the existence of minimizing movements for the lifted problem.
\begin{thm}[\textit{Minimizing movements as limit path for $\tau \rightarrow 0$}]
\label{thm:existence_of_gen_MMs}
With $\nu_{\tau,t}$ as in \eqref{eq:nu_tau_step}, there exists $(\tau_\ell)_{\ell\in \N}\subset (0,1)$ with $\tau_\ell\searrow 0$, and a curve $\nu_t \in  AC^2_{loc}([0,+\infty); \mathcal{P}_2(\Omega_L))$  such that for every $t \in [0,+\infty)$ it holds that 
\begin{align*}
\lim_{\ell \to \infty} \mathcal{W}_2(\nu_{\tau_\ell,t},\nu_t)\rightarrow 0 \quad  \text{as } \ell\to\infty\,.
\end{align*}
In particular, we have that $\nu_t$ is a minimizing movement in the sense of \Cref{def:minimizing_movements}.
\end{thm}

\begin{proof}
%We first prove that the sublevels sets of $\BBB \mathcal{J}\EEE$ are relatively compact with respect to the narrow convergence of measures. To this end, consider a sequence $(\nu^m)_{m\in \N} \BBB \subset \EEE \mathcal{P}_2(\Omega_L)$ such that 
%\[\sup_{m\in \N} \BBB \mathcal{J}\EEE(\nu^m) \leq C<+\infty\,,\] 
%for some constant $C>0$. By Prokhorov's theorem, cf. \cite[Theorem 5.1.3]{AGS05}, we only need to prove that the sequence $(\nu^m)_{m\in \N}$ is tight. This is true thanks to the integral criterion for tightness given in \cite[Remark 5.1.5]{AGS05} using the function $\varphi(\theta,u) := \theta^2$, whose sublevel sets are compact on $\Omega_L$. Therefore, $(\nu^m)_{m\in \N}$ is relatively compact in the narrow topology. 

 We first check that $\mathcal{J}:\mathcal{P}_2(\Omega_L) \rightarrow \R$ is lower semicontinuous with respect to the narrow convergence. Indeed, if $(\nu^m)_{m\in \N} \subset \mathcal{P}_2(\Omega_L)$ is narrowly converging to $\nu \in \mathcal{P}_2(\Omega_L)$, and we assume % We aim to prove that 
%\begin{align}\label{eq:I_narrow_lsc}
%\liminf_{m\to \infty} \BBB \mathcal{J}\EEE(\nu^m) \geqslant \BBB \mathcal{J}\EEE(\nu)\,.
%\end{align}
without restriction that \begin{align}\label{eq:M}
\liminf_{m\to \infty}  \mathcal{J}(\nu^m) = \lim_{m\to \infty}  \mathcal{J}(\nu^m) =: M <+\infty\,,
\end{align}
then
$$\liminf_{m\to \infty} \int_{\Omega_L} c^2 d\nu^m(c,u) \geqslant \int_{\Omega_L} c^2 d\nu(c,u)\,,$$
thanks to \cite[Formula 5.1.15]{AGS05}. Since $\mathcal{F}$ is continuous and convex cf. \ref{assmpt:a2}, it is also weakly lower semicontinuous. Therefore, it is enough to prove that 
\[\lim_{m\to \infty} \int_{\Omega_L} c^2 Ku \,d \nu^m(c,u) = \int_{\Omega_L} c^2 Ku \,d \nu(c,u)\]
with respect to the weak convergence in $Y$. To this end,  given $y \in Y$ it holds that 
\begin{equation}\label{eq:weak}
\Big(\int_{\Omega_L} c^2 Ku \,d \nu^m(c,u),y\Big)_Y = \int_{\Omega_L} c^2 (Ku,y)_Y \,d \nu^m(c,u)\,.
\end{equation}
Note that $c^2 |(Ku,y)_Y|$ is uniformly integrable with respect to $\nu^m$. Indeed,
\begin{align*}
\int_{\Omega_L} c^2 |(Ku,y)_Y|\, d\nu^m \leqslant  \big(\|K\|\|y\|_Y \sup_{u\in \mathcal{B}}\|u\|_{\M}\big)  \mathcal{J}(\nu^m)\,,
\end{align*}
 and the right hand side of the last inequality  is uniformly bounded in $m$ thanks to \eqref{eq:M}. Therefore, applying \cite[Lemma 5.1.7]{AGS05} to \eqref{eq:weak} we conclude the lower semicontinuity of $ \mathcal{J}$ with respect to the narrow convergence. Finally, since $\mathcal{W}_2$ metrizes the narrow convergence in $\mathcal{P}_2(\Omega_L)$, we can apply 
\cite[Proposition 2.2.3]{AGS05} to conclude the proof.
\end{proof}

\subsection{$\lambda$-Convexity for the lifted problem}
Analogously to Subsection \ref{subsec:lambda_cvx_nonlifted}, we turn to the study of convexity properties of the functional $\mathcal{J}$ along geodesics in the Wasserstein space $\mathcal{P}_2(\Omega_L)$. In particular, we address the question whether for a geodesic $\gamma_t\in \mathcal{P}_2(\Omega_L)$, the function 
\begin{equation*}
 t\mapsto\mathcal{J}(\gamma_t) := \F\left(\int_{\Omega_L}c^2Ku\,d\gamma_t( c,u)\right) + \int_{\Omega_L}c^2\,d\gamma_t( c,u)
\end{equation*}
is $\lambda$-convex, \textit{i.e.}, fulfills \eqref{eq:lambda_convexity}. In the following, we denote by $C([0,1]; \Omega_L)$ the separable and complete metric space of continuous curves in $\Omega_L$, endowed with the metric of uniform convergence induced by $d_{\Omega}$, cf. \eqref{eq:d_omega_def}. Define then the \textit{evaluation map} $e_t:C([0,1]; \Omega_L) \rightarrow \Omega_L$ as 
\[e_t(v):=v(t) \ \text{ for every }t\in [0,1]\,.\]
Note also that by \cite[Theorem 10.6]{ambrosio2021lectures}, geodesics in $\mathcal{P}_2(\Omega_L)$ exist, as $\Omega_L$ is a geodesic space. To prove the $\lambda$-convexity, we first recall the \emph{metric dynamical formulation of optimal transport}, which is a major ingredient in the proof.
For every $\mu_0, \mu_1 \in\mathcal{P}_2(\Omega_L)$,  consider the minimization problem

\begin{equation}\label{eq:dyn_minimiz}
\min
\Big\{\int_{C([0,1];\Omega_L)}
\int_0^1 |\gamma'_t|^2\,dt\,d\eta(\gamma)\colon \eta\in \mathcal{P}(C([0,1];\Omega_L))\,, (e_0,e_1)_\#\eta = (\mu_0,\mu_1)\Big\}\,.
\end{equation}
We call any such admissible $\eta$ a \textit{dynamic transport plan}, the collection of which is denoted by $\mathrm{DTP}(\mu_0,\mu_1)$,  and any optimal one will be called an  \textit{optimal dynamic geodesic plan}. We denote by $\mathrm{OptGeo}(\mu_0,\mu_1)$ the collection of all optimal geodesic plans from $\mu_0$ to $\mu_1$, which is be supported on $\mathrm{Geo}(\Omega_L)$ (see \cite[Theorem 9.13]{ambrosio2021lectures}). 
Using \cite[Theorem 9.13]{ambrosio2021lectures} and Lemma \ref{lem:coord_geodesics}, it holds that 
\begin{align}\label{eq:dyn_ot}
\W_2^2(\mu_0, \mu_1)=\min_{\eta \in \mathrm{DTP}(\mu_0,\mu_1)}\int_{\mathrm{Geo}(\Omega_L)} \int_0^1\left|( c_t,u_t)^{\prime}\right|^2\, dt\, d \eta(c, u)\,. 
\end{align}
In addition, \cite[Theorem 10.6]{ambrosio2021lectures} guarantees that for any geodesic $\gamma_t \in \mathcal{P}_2(\Omega_L)$, there exists an $\eta \in \mathcal{P}(C([0,1]; \Omega_L)$ with $\eta \in \mathrm{OptGeo(\gamma_0,\gamma_1)}$ and $\mathrm{spt}\,\eta \subset \mathrm{Geo}(\Omega_L)$ such that
\begin{equation}\label{eq:gamma_t}
 \gamma_t = (e_t)_\#\eta\,.
\end{equation}
Thus, fixing any geodesic $\gamma_t\in \mathcal{P}_2(\Omega_L)$, we can write $\mathcal{J}(\gamma_t)$  as 
\begin{align}\label{eq:I_gamma_t_reform}
\begin{split}
\mathcal{J}(\gamma_t) &= \F\left(\int_{\Omega_L}c^2Ku\, d[(e_t)_\#\eta](c,u)\right) + \int_{\Omega_L}c^2\,d[(e_t)_\#\eta](c,u)\\
&= \F\left(\int_{\mathrm{Geo}(\Omega_L)}c^2_tKu_t\, d\eta(c,u)\right) + \int_{\mathrm{Geo}(\Omega_L)}c_t^2\,d\eta(c,u)\,.
\end{split}
\end{align}
\noindent

\begin{prop}\label{prop:semiconvexity_of_lifted}
Assume \ref{assmpt:a1}-\ref{assmpt:a8}. Then for every $L >0$, there exists a $\lambda(L) \leq 0$ such that $\mathcal{J}$ is $\lambda$-convex along geodesics on $\mathcal{P}_2(\Omega_L)$.
\begin{proof} 
Note that for every geodesic $( c_t,u_t)\in \mathrm{Geo}(\Omega_L)$ by \Cref{lem:coord_geodesics} we have that $c_t \in \mathrm{Geo}(\R_+)$ and $u_t\in \mathrm{Geo}(\B)$. Thus $c_t = (1-t)c_0 + t c_1$.
Therefore,
\begin{equation}\label{eq:d_dt_c_t_est}
\frac{d}{dt} c_t^2 =\ 2t(c_1-c_0)^2 + 2(c_1- c_0) c_0 \leq 2\big(|c_1-c_0|^2 + |c_1-c_0|c_0\Big), \ \text{ and } 
\frac{d^2}{dt^2}c_t^2 =\ 2(c_1- c_0)^2\,.
\end{equation}
Let now $\gamma_t \in \mathcal{P}_2(\Omega_L)$  be a geodesic.
Thanks to \cite[Theorem 10.6]{ambrosio2021lectures} there exists $\eta \in \mathcal{P}(\mathrm{Geo}(\Omega_L))$ with $\gamma_t = (e_t)_\# \eta$ and $\eta \in \mathrm{OptGeo}(\gamma_0,\gamma_1)$.
Therefore, by the Dominated Convergence Theorem,
\begin{align*}
\frac{d}{dt} \int_{\Omega_L} c^2 d\gamma_t =&\ \lim_{h\rightarrow 0} \left(\int_{\Omega_L} c^2 d\gamma_{t+h} - \int_{\Omega_L} c^2 d\gamma_{t}\right) / h = \lim_{h\rightarrow 0} \left(\int_{\Omega_L} c^2 d(e_{ t+h})_\#\eta - \int_{\Omega_L} c^2 d(e_{t})_\#\eta\right) / h \\
=&\ \lim_{h\rightarrow 0} \int_{\mathrm{Geo}(\Omega_L)}\Big(\frac{c^2_{t+h}-c^2_{t}}{h}\Big)\,
d\eta 
= \int_{\mathrm{Geo}(\Omega_L)} \frac{d}{dt} (c^2_t)\, d\eta\,,
\end{align*}
and similarly $\frac{d^2}{dt^2} \int_{\Omega_L} c^2 d\gamma_t = \int_{\mathrm{Geo}(\Omega_L)} \frac{d^2}{dt^2} (c^2_t)\, d\eta$.
Recalling \eqref{eq:metric_derivative} for geodesics, \eqref{eq:dyn_ot} and \eqref{eq:d_dt_c_t_est}, we can further estimate
\begin{align}\label{eq:second_der_regularizer_lifted}
\begin{split}
\frac{d^2}{dt^2} \int_{\Omega_L} c^2 d\gamma_t = &\ \int_{\mathrm{Geo}(\Omega_L)} 2( c_1 - c_0)^2 \,d\eta \leq\ 2 \int_{\mathrm{Geo}(\Omega_L)}\big( |c_1 -  c_0|^2 + d_\B^2(u_1,u_0)\big)\, d\eta \\
=&\ 2\int_{\mathrm{Geo}(\Omega_L)} \int_0^1|( c_t,u_t)'|^2\, dt\, d\eta = 2 \W_2^2(\gamma_0,\gamma_1)\,.
\end{split}
\end{align}
Moreover, by a similar application of the Dominated Convergence Theorem, we have
\begin{equation}\label{eq:d_dt_geo_1}
\frac{d}{dt} \int_{\Omega_L} c^2 Ku\, d\gamma_t =   \int_{\mathrm{Geo}(\Omega_L)} \Big(\frac{d}{dt}( c_t^2) Ku_t +  c_t^2 \frac{d}{dt} Ku_t\,,
\Big)\,d\eta\,        
\end{equation}
and 
\begin{equation}\label{eq:d_dt_geo_2}
\frac{d^2}{dt^2} \int_{\Omega_L}  c^2 Ku\, d\gamma_t = \int_{\mathrm{Geo}(\Omega_L)} \Big(\frac{d^2}{dt^2}( c_t^2)Ku_t + 2  \frac{d}{dt}( c_t^2)\frac{d}{dt} Ku_t +  c_t^2 \frac{d^2}{dt^2}Ku_t\Big) d\eta\,.\end{equation}
Arguing as in the proof of \Cref{prop:local_convexity}, using \eqref{eq:d_dt_geo_1},  Jensen's inequality for the probability measure $\eta$ on $\mathrm{Geo}(\Omega_L)$, \eqref{eq:d_dt_c_t_est}, \eqref{eq:bound_on_geodesics}, the Cauchy-Schwarz inequality, and for a constant $C:=C(K,\mathcal{B},L)>0$ that is allowed to vary from line to line, we estimate
\begin{align}\label{eq:d_dt_gamma_t_1}
\Big\| \frac{d}{dt}\int_{\Omega_L} c^2 Ku\, d\gamma_t \Big\|_Y^2 \leq&  \int_{\mathrm{Geo}(\Omega_L)} \big|\frac{d}{dt}  c_t^2\big|^2 \left\Vert Ku_t\right\Vert_Y^2+ 2 \big|\frac{d}{dt} c_t^2\big|| c_t^2| \|Ku_t\|_Y\big\|\frac{d}{dt}Ku_t\big\|_Y + | c_t^2|^2\left\Vert\frac{d}{dt}Ku_t\right\Vert_Y^2 d\eta \notag \\
\leq & C \Big(\int_{\mathrm{Geo}(\Omega_L)} |\frac{d}{dt}  c_t^2|^2 \left\Vert Ku_t\right\Vert_Y^2\,d\eta+  \int_{\mathrm{Geo}(\Omega_L)}   |c_t^2|^2\left\Vert\frac{d}{dt}Ku_t\right\Vert_Y^2\, d\eta\Big)\notag\\
\leq&\  C\Big( \int_{\mathrm{Geo}(\Omega_L)} \Big(\big(|c_1- c_0|^2 + |c_1-c_0| c_0\big)^2+(|c_0|+|c_1-c_0|)^4d^2_{\B}(u_0,u_1)\Big)\,d\eta\Big) \notag \\
\leq&\  C  \int_{\mathrm{Geo}(\Omega_L)}  \big(|c_1-c_0|^2 + d_\B^2(u_0,u_1)\big)\, d\eta = C \W_2^2(\gamma_0,\gamma_1)\,,
\end{align}
 where in the last equality we used again \eqref{eq:dyn_ot}. Analogously, using this time \eqref{eq:d_dt_geo_2} and again \eqref{eq:bound_on_geodesics}, we estimate  
\begin{align*}
\int_{\mathrm{Geo}(\Omega_L)}  c_t^2 \left\Vert \frac{d^2}{dt^2} Ku_t\right\Vert_Y d\eta \leq&\  C   \int_{\mathrm{Geo}(\Omega_L)} (| c_1- c_0|+ c_0)^2 d_\B^2(u_0,u_1)\, d\eta \\
\leq&\  C  \int_{\mathrm{Geo}(\Omega_L)} \big(|c_1- c_0|^2 + d_\B(u_0,u_1)^2\big) d\eta = C \W_2^2(\gamma_0,\gamma_1)\,,
\end{align*}
which directly leads to 
\begin{equation}\label{eq:2nd_derv_gamma_t_est}
\Big\|\frac{d^2}{dt^2} \int_{ \Omega_L}  c^2 Ku\, d\gamma_t \Big\|_Y \leq C\W_2^2(\gamma_0,\gamma_1)\,.
\end{equation}
Finally, considering the $\R$-valued map $[0,1]\ni t \mapsto \mathcal{J}(\gamma_t)$, and setting for brevity 
\begin{align*}
\underline{K}(t) := \int_{\mathrm{Geo}(\Omega_L)}  c_t^2 Ku_t\, d\eta\,,
\end{align*}
by \eqref{eq:I_gamma_t_reform}, \eqref{eq:second_der_regularizer_lifted}, \eqref{eq:d_dt_gamma_t_1}, and \eqref{eq:2nd_derv_gamma_t_est}, 
we get
 \begin{align*}
\frac{d^2}{dt^2} \mathcal{J}(\gamma_t) =&\ \frac{d^2}{dt^2} \F\left(\underline{K}(t) \right) + \frac{d^2}{dt^2} \int_{\mathrm{Geo}(\Omega_L)} c_t^2\,d\eta\,\\
\leq &\ \frac{d}{dt}\left( \nabla\F_{\underline{K}(t)}, \frac{d}{dt} \underline{K}(t) \right)_Y + 2 \W_2^2(\gamma_0,\gamma_1)\\
=&\ \nabla^2\F_{\underline{K}(t) } \left[\frac{d}{dt} \underline{K}(t),\frac{d}{dt} \underline{K}(t) \right] + \left(\nabla\F_{\underline{K}(t)}, \frac{d^2}{dt^2}\underline{K}(t) \right)_Y +2 \W_2^2(\gamma_0,\gamma_1) \\
\leq&\ ||\F||_{C^2 }  \left(\left\Vert \frac{d}{dt} \underline{K}(t)\right\Vert_Y^2  +  \left\Vert \frac{d^2}{dt^2} \underline{K}(t)\right\Vert_Y \right) + 2 \W_2^2(\gamma_0,\gamma_1) \\
%\leq&\ \BBB C\EEE||\F||_{C^2 }  \left[C(L,K,B) W_2^2(\gamma_0,\gamma_1) + C(L,K,B) W_2^2(\gamma_0,\gamma_1)\right] + C(L) W_2^2(\gamma_0,\gamma_1) \\
\leq&\  C  \W_2^2(\gamma_0,\gamma_1)\,,
\end{align*}
 where the constant $C>0$ in the last line depends also on $\mathcal{F}$.  Thus, the map $t \mapsto  \mathcal{J}(\gamma_t)$ is $\tilde \lambda \W_2^2(\gamma_0,\gamma_1)$-convex, for  some $\tilde \lambda:=\tilde \lambda(\F,K,\B,L)\leq 0$.  This exactly implies that $ \mathcal{J}$ is $ \tilde \lambda$-convex along geodesics in $\Omega_L$.
\end{proof}
\end{prop}

\begin{thm}\label{thm:suglifted}
Assume \ref{assmpt:a1}-\ref{assmpt:a8}. Then $|\partial \mathcal{J}|$ is a strong upper gradient for $\mathcal{J}$ on $\mathcal{P}_2(\Omega_L)$. In addition, every minimizing movement $\mu_t$ is a curve of maximal slope for $\mathcal{J}$ with regard to $|\partial \mathcal{J}|$.
\begin{proof}
First, \cite[Corollary 2.4.10]{AGS05} implies that for all $L>0$, $|\partial \mathcal{J}|$ is a strong upper gradient for $ \mathcal{J}$, since the latter is $\lambda$-convex by Proposition \ref{prop:semiconvexity_of_lifted} and lower-semicontinuous.
Next, let $\mu_0 \in \mathrm{Dom}(\mathcal{J})$ and $\mu_t$ a minimizing movement for $\mathcal{J}$, which exists by \Cref{thm:existence_of_gen_MMs}. Then since $\mathcal{J}$ is $\lambda$-convex on $\mathcal{P}_2(\Omega_L)$, \cite[Corollary 2.4.11]{AGS05} implies that $\mu_t$ is a curve of maximal slope for $\mathcal{J}$ with regard to $|\partial \mathcal{J}|$.
\end{proof}
\end{thm}

\begin{rem}
Note that here one cannot obtain uniqueness of the curves of maximal slope via standard methods, contrary to the setting for the non-lifted functional $J_n$. This is because in general, $\mathcal{P}_2(X)$ is not NPC, as geodesics can be non-unique. For $X$ being a Hilbert space, the uniqueness in $\mathcal{P}_2(X)$ still holds for  lower semicontinuous  functionals which are $\lambda$-convex along geodesics (see \cite[Theorem 11.1.4]{AGS05}), but follows from a variational characterization of the so-called Wasserstein subdifferential. If $X$ is just a metric space such as in our setting, this characterization and even the Wasserstein subdifferential are not available.
\end{rem}

\section{Relating the minimizing movements}
\label{sec:equivalence_gfs}
In this section we aim to relate the minimizing movements defined in Subsections  \ref{subsec:MM_AGF} and \ref{sec:minmovlifted}. 
In particular, we will show that gradient flows for $J_n$ induce, through lifting, gradient flows for the lifted functional $\mathcal{J}$. The precise statement is given below in Theorem \ref{thm:equivalence_of_curves_of_max_slope}. We remark that it is a generalization of \cite[Proposition B.1]{chizat2018global}. However, since we are here dealing with measures defined on general metric spaces, we cannot rely on the definition of Wasserstein gradient flow through the continuity equation, thus requiring a substantially different approach for the proof. Throughout this section, we will always assume \ref{assmpt:a1}-\ref{assmpt:a8}, \textit{i.e.}, the standard assumptions, the no loss of mass condition on the regularizer \eqref{eq:cond_no_loss_of_mass}, that $0\notin \B$ and the compatibility condition between $K$ and the metric in $\mathcal{B}$ (see also Remark \ref{rem:assumption8}).
For the next theorem, we again recall \eqref{eq:Omega_n_def}, \eqref{eq:discrete_version_of_energy}, and \eqref{eq:repr_I}.

%\textcolor{red}{Christian, have a look how we rephrase the assumption before}
%\chris{I've checked it and made it a bit more precise from my point of view; to me now it looks good}
%Moreover, in order to relate the two dynamics we need a further compatibility condition between $K$ and the geometry of $\mathcal{B}$. In particular, we consider the following additional Lipschitz assumption on $K$.
%\begin{enumerate}[label=\textnormal{(A\arabic*)}] 
%\setcounter{enumi}{7}
%\item \label{assmpt:a8} There exists $C>0$ such that for every $u,v \in \mathcal{B}$ it holds that 
%\begin{align}
%    \|Ku - Kv\|_Y \leq C d_{\B}(u,v)\,.
%\end{align}
%\end{enumerate}
%Note that \ref{assmpt:a8} is implied by \eqref{eq:bound_on_geodesics} with $m=1$ due the geodesically completeness of $\mathcal{B}$. 

\begin{thm}\label{thm:equivalence_of_curves_of_max_slope}
Suppose that \ref{assmpt:a1}-\ref{assmpt:a8} hold, and let $(\c_t,\u_t) \in AC([0,1],\Omega^n_L)$ be a curve of maximal slope for $J_n$ with  respect to the strong upper gradient $|\partial J_n|$. Then, 
\begin{equation}\label{eq:mu_t_def}
\mu_t := \frac{1}{n} \sum_{j=1}^n \delta_{(c_t^j,u_t^j)} \in \mathcal{P}_2(\Omega_L)
\end{equation}
is a curve of maximal slope for $ \mathcal{J}$ with regard to the strong upper gradient $|\partial \mathcal{J}|$.
\end{thm}
We first show that \Cref{thm:equivalence_of_curves_of_max_slope} is a consequence of the slope equality in \Cref{thm:gradient_equality_many_particle}, whose proof will be postponed to Subsections \ref{subsec:one_particle_case} and \ref{subsec:many_particle case}.
\begin{thm}\label{thm:gradient_equality_many_particle}
Suppose that \ref{assmpt:a1}-\ref{assmpt:a8} hold and let $n\geq 1$ and $\mu_t$ be as in \eqref{eq:mu_t_def}. Then,
\begin{equation}\label{eq:slope_equalities}
|\partial  \mathcal{J}|(\mu_t) = |\partial J_n|(\c_t,\u_t)\,,
\end{equation}
 where $L>0$ is arbitrary, but fixed. 
\end{thm}
\begin{proof}[Proof of \Cref{thm:equivalence_of_curves_of_max_slope}] 
Note that $|\partial J_n|$ and $|\partial \mathcal{J}|$ are strong upper gradients for $J_n$ and $\mathcal{J}$ respectively by Theorem \ref{thm:suglifted} and Lemma \ref{lem:sugJn}.
Recalling Definitions \ref{def:Strong_upper_gradient} and \ref{def:curves_max_slope}, we need to show that 
\begin{equation} \label{eq:lifted_c_o_m_s}
2 \frac{d}{dt}  \mathcal{J}(\mu_t) \leq - |\mu_t'|^2- |\partial \mathcal{J}|^2(\mu_t) \quad \text{for } \mathcal{L}^1-\mathrm{a.e. }\ t\in[0,1]\,,
\end{equation}
using the fact that $(\c_t,\u_t)$ is already a curve of maximal slope for $J_n$, so fulfills
\begin{equation}\label{eq:nonlifted_c_o_m_s2}
2 \frac{d}{dt} J_n(\c_t,\u_t) \leq - |(\c_t,\u_t)'|^2- |\partial J_n|^2(\c_t,\u_t) \quad \text{for } \mathcal{L}^1-\mathrm{a.e. }\ t\in [0,1]\,.
\end{equation}
For this, note first that  by \eqref{eq:mu_t_def}, \eqref{eq:repr_I}, \ref{assmpt:a1} and \eqref{eq:discrete_version_of_energy},
\[ \mathcal{J}(\mu_t)=\mathcal{F}\left(\frac{1}{n}\sum_{j=1}^n(c_t^j)^2Ku_t^j\right)+\frac{1}{n}\sum_{j=1}^n(c_t^j)^2= J_n(\c_t,\u_t)\,,\] 
therefore it holds that $\frac{d}{dt} \mathcal{J}(\mu_t) =  \frac{d}{dt} J_n(\c_t,\u_t)$. Also using \eqref{eq:metric_derivative} on $(\mathcal{P}_2(\Omega_L),\W_2)$ and \eqref{eq:Wasserstein_2} with the admissible transport plan
\[\gamma=\Big(\frac{1}{n}\sum_{j=1}^n\delta_{(c_t^j,u_t^j)}, \frac{1}{n}\sum_{j=1}^n\delta_{(c_s^j,u_s^j)}\Big)\,,\]
as well as \eqref{eq:d_Omega}, \eqref{eq: tau_dissip} and \eqref{eq:metric_derivative} on $(\Omega_L^n,d_n)$,  we obtain 
\begin{align*}
|\mu_t'| = \lim_{s\rightarrow t}\frac{\W_2(\mu_s,\mu_t)}{|s-t|} \leq   \lim_{s\rightarrow t}\frac{d_n((\c_s,\u_s),(\c_t,\u_t))}{|s-t|} = |(\c_t,\u_t)'|\,.
\end{align*}
Thus, using \eqref{eq:nonlifted_c_o_m_s2}, the above inequality and \eqref{eq:slope_equalities}  already gives
\begin{align*}\label{eq:estimate_curve_max_slope_1}
2 \frac{d}{dt} \mathcal{J}(\mu_t) &= 2 \frac{d}{dt} J_n(\c_t,\u_t) \leq - |(\c_t,\u_t)'|^2- |\partial J_n|^2(\c_t,\u_t)  \\[3pt]
&\leq -|\mu_t'|^2 - |\partial J_n|^2(\c_t,\u_t)=-|\mu_t'|^2 - |\partial \mathcal{J}|^2(\mu_t) \ \text{for } \mathcal{L}^1-\mathrm{a.e. }\ t\in[0,1]\,,
\end{align*}
 which shows \eqref{eq:lifted_c_o_m_s}, and concludes the proof of the theorem.
\end{proof}

 Hence, the rest of the section is devoted to the proof of the equality of the corresponding metric slopes of the functionals $J_n$ and $\mathcal{J}$.  

\subsection{The one-particle case $(n=1)$} \label{subsec:one_particle_case}
In the following, we prove \Cref{thm:gradient_equality_many_particle} first in the case $n=1$ by a simple localization argument, as the proof is significantly simpler than for $n>1$. For this, we need the following lemma, the proof of which we defer to  \Cref{app:A_2_proof_5_3} for better readability.
\begin{lem}\label{lem:fixing_of_signs}
Suppose that \ref{assmpt:a1}-\ref{assmpt:a8} hold. Let $n=1$ and $\mu_t = \delta_{(c_t,u_t)}$. For $\mathcal{L}^1$-a.e. $t\in [0,1]$, it holds that
\begin{equation}\label{eq:Wasserstein_limit_1}
\lim_{\nu\stackrel{*}{\rightharpoonup}\mu_t} \frac{\left|\int_{\Omega_L} \F(c^2 Ku)\, d\nu-\F\left(\int_{\Omega_L} c^2 Ku\,d\nu\right)\right|}{\W_2(\nu,\mu_t)} =0\,.
%\quad \text{ for } \nu\rightharpoonup^\ast\mu_t.
\end{equation}
\end{lem}
\noindent Then we can show the slope equality for $n=1$.
\begin{prop}
\label{prop:gradient_equality_one_particle}
In the setting of \Cref{thm:equivalence_of_curves_of_max_slope}  and \Cref{lem:fixing_of_signs},  one has $|\partial \mathcal{J}|(\mu_t) = |\partial J_1|(c_t,u_t)$.

\begin{proof}
Note that by the definitions \eqref{eq:local_slope}, 
\eqref{eq:discrete_version_of_energy}, and  \eqref{eq:repr_I}, 
\begin{align*}
|\partial \mathcal{J}|(\mu_t)=\limsup_{\tilde \mu\stackrel{*}{\rightharpoonup}\mu_t}\frac{(\mathcal{J}(\mu_t)-\mathcal{J}(\tilde\mu))^+}{\W_2(\mu_t,\tilde\mu)}&\geq \limsup_{(\tilde c,\tilde u)\to (c_t,u_t)}\frac{(\mathcal{J}(\delta_{( c_t,u_t)})-\mathcal{J}(\delta_{( \tilde c, \tilde u)}))^+}{\W_2(\delta_{(c_t,u_t)},\delta_{(\tilde c,\tilde u)})}\\
&=\limsup_{(\tilde c,\tilde u)\to (c_t,u_t)}\frac{(J_1(c_t,u_t)-J_1(\tilde c,\tilde u))^+}{d_\Omega((c_t,u_t),(\tilde c,\tilde u))}=|\partial J_1|(c_t,u_t)\,,
\end{align*}
\textit{i.e.}, the inequality $|\partial  \mathcal{J}|(\mu_t) \geq |\partial J_1|(c_t,u_t)$ is a direct consequence of the definitions.

For the reverse inequality, let $\nu\in \mathcal{P}_2(\Omega_L)$. By  \eqref{eq:repr_I}, and that $\mathcal{R}(u)=1$, 
we can estimate
\begin{align}\label{eq:1_particle_reverse}
 \mathcal{J}(\mu_t)- \mathcal{J}(\nu)  & = J_1(c_t,u_t) -  \int_{\Omega_L} J_1(c,u)\,d\nu + \int_{\Omega_L} J_1(c,u)\,d\nu -\F\left(\int_{\Omega_L} c^2 Ku\, d\nu\right) - \int_{\Omega_L} c^2\, d\nu \notag\\
&\leq \int_{\Omega_L} ( J_1(c_t,u_t) -   J_1(c,u))^+\,d\nu + \left|\int_{\Omega_L} \F(c^2Ku)\, d\nu-\F\left(\int_{\Omega_L} c^2 Ku\, d\nu\right)\right|\,. 
\end{align} 
Since
\begin{align*}
|\partial J_1|(c_t,u_t) = \limsup_{(c,u)\rightarrow (c_t,u_t)} \frac{\left(J_1(c_t,u_t) - J_1(c,u)\right)^+}{d_{\Omega}((c_t,u_t), (c,u))}\,,
\end{align*}
for every $\varepsilon >0$ there exists $\delta >0$ such that if $d_{\Omega}( (c_t,u_t), (c,u)) \leq \delta$, then 
\begin{align*}
\frac{ ( J_1(c_t,u_t) - J_1(c,u) )^+ }{d_{\Omega}((c_t,u_t), (c,u))} \leq |\partial J_1|(c_t,u_t) + \varepsilon\,.
\end{align*}
Note also that, thanks to \cite[Theorem 5.3.1]{AGS05} and Jensen's inequality, for $\pi\in \Gamma_0( \nu,\mu_t)$ (cf. the notation after \eqref{eq:d_Omega}),
\begin{align*}
 \W_2(\nu,\mu_t) = \left( \int_{\Omega_L}\int_{\Omega_L} d_\Omega^2((c,u),(\tilde{c},\tilde{u}))\, d\pi ((c,u),(\tilde{c},\tilde{u}))\right)^{1/2} &=  \left(\int_{\Omega_L} d_\Omega^{ 2}((c_t,u_t),(c,u)) \,d\nu(c,u)\right)^{1/2}\\
&\geq \int_{\Omega_L} d_\Omega((c_t,u_t),(c,u)) \,d\nu(c,u)\,.
\end{align*}
Therefore, denoting for brevity 
\begin{equation}\label{eq:delta_t_ball}
B_{\delta,t}:=B_\delta((c_t,u_t))\,,
\end{equation}
the latter intending the $\delta$-ball in $\Omega_L$ centered in $(c_t,u_t)$, we can easily estimate,
\begin{align}
\label{eq:fraction_w_2_leq1}
\frac{\int_{B_{\delta,t}} d_{\Omega}((c_t,u_t),(c,u))\, d\nu}{\W_2(\nu,\mu_t)} \leq \frac{\int_{\Omega_L} d_\Omega((c_t,u_t),(c,u))\, d\nu}{\W_2(\nu,\mu_t)} \leq 1\,.
\end{align}
By using \eqref{eq:fraction_w_2_leq1} we can therefore decompose 
\begin{align}\label{eq:slope_1_est}
\int_{\Omega_L} \frac{(J_1(c_t,u_t) - J_1(c,u) )^+}{\W_2(\nu, \mu_t)}\, d\nu & = \int_{ B_{\delta,t}} \frac{( J_1(c_t,u_t) - J_1(c,u))^+}{\W_2(\nu, \mu_t)}\, d\nu + \int_{\Omega_L\setminus B_{\delta,t}} \frac{(J_1(c_t,u_t) - J_1(c,u))^+}{\W_2(\nu, \mu_t)}\, d\nu\notag \\[3pt]
& \leq  
\frac{(|\partial J_1|(c_t,u_t) + \varepsilon)\int_{B_{\delta,t}} d_{\Omega}((c_t,u_t), (c,u))\, d\nu }{ \W_2(\nu,\mu_t)}  \notag\\[3pt]
& + \qquad \int_{\Omega_L\setminus B_{\delta,t} } \frac{(J_1(c_t,u_t) - J_1(c,u))^+}{ \W_2(\nu, \mu_t)}\, d\nu\notag \\
& \leq |\partial J_1|(c_t,u_t) + \varepsilon + 2 \Big(\sup_{(c,u)\in \Omega_L} J_1(c,u) \Big)\frac{\nu(\Omega_L\setminus B_{\delta,t})}{\W_2(\nu, \mu_t)}.
\end{align}
In addition, for every $\delta >0$ it holds that 
\begin{equation}\label{eq:n_conc_out_of_B_delta}
\lim_{\nu\rightarrow^\ast \mu_t} \frac{\nu( \Omega_L\setminus B_{\delta,t})}{ \mathcal{W}_2(\nu,\mu_t)} = 0\,.
\end{equation} 
 Indeed,  
\begin{align}
\label{eq:quantitative_estimate_1n}
\W^2_2(\nu,\mu_t)   \geq \int_{ \Omega_L\setminus B_{\delta,t}}  d^2_{\Omega}((c_t,u_t), (c,u))\,d\nu \geq \delta^2 \nu( \Omega_L\setminus B_{\delta,t})\,\implies \nu(\Omega_L\setminus B_{\delta,t}) \leq \frac{\W^2_2(\nu,\mu_t)}{\delta^2}\,,
\end{align}
from which \eqref{eq:n_conc_out_of_B_delta} follows. 
Therefore, if $(\nu_k)_{k\in \N}\subset \mathcal{P}_2(\Omega_L)$, with $\nu_k\stackrel{*}{\rightharpoonup}\mu_t$ as $k\to \infty$, is a sequence attaining the lim sup in the definition of $|\partial \mathcal{J}|(\mu_t)$, we can combine \eqref{eq:1_particle_reverse}, \eqref{eq:Wasserstein_limit_1}, \eqref{eq:slope_1_est} and \eqref{eq:n_conc_out_of_B_delta}, to obtain
\begin{align*}
|\partial \mathcal{J}|(\mu_t)&=\limsup_{k\to \infty}\frac{\left( \mathcal{J}(\mu_t)- \mathcal{J}(\nu_k)\right)^+}{\W_2(\nu_k, \mu_t)} \\
&
\leq |\partial J_1|(c_t,u_t) + \varepsilon + 2 \Big(\sup_{(c,u)\in \Omega_L} J_1(c,u) \Big) \limsup_{k\to \infty}\frac{\nu_k( \Omega_L\setminus B_{\delta,t})}{ \W_2(\nu_k, \mu_t)}
\\ 
&
\leq |\partial J_1|(c_t,u_t) + \varepsilon\,.
\end{align*}
and since $\varepsilon$ is arbitrary, this concludes the proof of the inequality $|\partial \mathcal{J}|(\mu_t)\leq |\partial J_1|(c_t,u_t)$, and thus of the proposition. 
\end{proof}
\end{prop}

\subsection{The many-particle case}
\label{subsec:many_particle case}
In this subsection we generalize \Cref{prop:gradient_equality_one_particle} to $n$ particles,  hence dealing now with the case $(\c_t,\u_t)\in \Omega_L^n$. To this end, one needs to generalize known facts about semi-discrete optimal transport to the setting of metric spaces, which for convenience of the reader we defer to in Subsection \ref{subsec:semi_discrete_ot}.
Throughout this section, we will always assume that the particles are distinct (see Remark \ref{eq:distinct}) and we take $\delta >0$ small enough such that the balls
\begin{equation}\label{eq:balls_disjoint}
B_\delta((c^j_t,u^j_t)) \ \text{are pairwise disjoint}\,.
\end{equation}

For brevity, with a slight abuse of notation in this  subsection, for fixed $(\c_t,\u_t)\in \Omega_L^n$, we will also write
\begin{align}
\label{eq:def_y_i_and_Y}
\begin{aligned}
\mathrm{(i)}& \quad  \y_t:= (\c_t,\u_t) \in \Omega_L^n\,,  \quad y_t^j := (c_t^j,u_t^j) \in \Omega_L \ \ \text{for } j\in\{1,\dots,n\}\,, \\[2pt] 
\mathrm{(ii)}& \quad \x:= (\c,\u) \in \Omega_L^n\,,  \qquad  x := (c,u) \in \Omega_L \,, \\[2pt] 
\mathrm{(iii)}& \quad Y := \{y_t^1,\dots,y_t^n\} = \{(c_t^1,u_t^1),\dots,(c_t^n,u_t^n)\} \subset \Omega_L,\\
\mathrm{(iv)}& \quad B_\delta(y_t^j):=  B_\delta((c_t^j,u_t^j)) \subset \Omega_L,  \text{ and }
B_\delta(\y_t):= B_\delta((\c_t,\u_t)) \subset \Omega_L^n\,,
\end{aligned}
\end{align}
where in the last shorthand notation we actually intend 
\begin{equation}\label{eq:n_ball_definition}
B_\delta(\y_t):=B_\delta(\c_t,\u_t):=B_\delta(c_t^1,u_t^1)\times\dots \times B_\delta(c_t^n,u_t^n)\,.
\end{equation}
To show \Cref{thm:gradient_equality_many_particle}, we first need to establish some auxiliary lemmata. For this purpose, we recall once again the definition of the measure $\mu_t$ in \eqref{eq:mu_t_def}. 

\begin{lem}[\textit{Mass correction}]
\label{lem:exchange_limiting_sequence}
Suppose that \ref{assmpt:a1}-\ref{assmpt:a8} hold. With the notation of \eqref{eq:def_y_i_and_Y}, let $\delta>0$ be such that $\eqref{eq:balls_disjoint}$ holds.   Then, 
 \begin{equation}\label{eq:reduced_convergence_equivalent}
\limsup_{\nu \stackrel{*}{\rightharpoonup}\mu_t} \frac{J_n(\c_t,\u_t)-\mathcal{J}(\nu)}{\W_2(\nu,\mu_t)} = \limsup_{\nu \stackrel{*}{\rightharpoonup}\mu_t} \frac{J_n(\c_t,\u_t)- \mathcal{J}(\tilde{\nu})}{\W_2(\nu,\mu_t)}\,,
\end{equation}
 with $\tilde{\nu}$ being defined as
\begin{equation}\label{eq:nu_decomp}
\tilde{\nu}:=\frac{1}{n} \sum_{j=1}^n \frac{\nu|_{B_\delta(y^j_t)}}{\nu(B_\delta(y^j_t))}\,.
\end{equation}
\end{lem}

\begin{rem}\label{eq:distinct}
 Before giving the proof of \Cref{lem:exchange_limiting_sequence}, we mention that without loss of generality we additionally assume the particles to be distinct, \textit{i.e.},
 \begin{equation*}
\delta_{y^i}\neq\delta_{y^j} \quad \text{for }i\neq j\,.
\end{equation*}
In case this does not hold, all of the statements still follow, but one needs to account for multiplicities. For instance, with approximating measures $\nu$ as in \eqref{eq:reduced_convergence_equivalent}-\eqref{eq:nu_decomp}, we have 
\begin{equation*}
\nu|_{B_\delta(y^j_{ t}))} \stackrel{*}{\rightharpoonup} \frac{k_j}{n} \delta_{y^j_{ t}}\,, \  \text{where } k_j:=\#\{i\colon y_t^i=y_t^j\}\,.
\end{equation*}
In view of the definitions of the \textit{Laguerre cells} and the corresponding dual maximizers in \Cref{subsec:semi_discrete_ot}, we would then have for all optimal couplings $\pi \in \Gamma_0(\nu,\mu_t)$ that
\begin{equation*}
\frac{k_j}{n}= \pi(A_j\times \{y^j\})\,,
\end{equation*}
allowing one to directly generalize the proofs. 
\end{rem}
\begin{proof}
It is enough to show that
\begin{equation}\label{eq:nu_tilde_nu_close}
\lim_{\nu \stackrel{*}{\rightharpoonup}\mu_t} \frac{\left|\mathcal{J}(\nu)-\mathcal{J}(\tilde{\nu})\right|}{\W_2(\nu,\mu_t)} = 0\,.
\end{equation}
 For this purpose, set  $\nu_R := \nu -\tilde{\nu}$.  By the definition of $\mathcal{J}$ in \eqref{eq:repr_I}, \eqref{eq:nu_decomp}, the triangle inequality and \ref{assmpt:a2}, we have 
\begin{align*}
\left| \mathcal{J}(\nu)-\mathcal{J}(\tilde{\nu})\right| &\leq  \left| \mathcal{F}\left(\int_{\Omega_L} c^2Ku\,d(\tilde{\nu}+\nu_R)\right)- \mathcal{F}\left(\int_{\Omega_L} c^2Ku\,d\tilde{\nu}\right)\right| + \left|\int_{\Omega_L}  c^2\,d(\tilde{\nu}+\nu_R)-\int_{\Omega_L}  c^2\,d\tilde{\nu}\right| \\
&\leq \sup_{y\in A_\nu}\|\nabla  \mathcal{F}\|(y) \cdot \left|\int_{\Omega_L} c^2Ku\,d(\tilde{\nu}+\nu_R)-\int_{\Omega_L} c^2Ku\,d\tilde{\nu}\right| + \left|\int_{\Omega_L}  c^2 \,d\nu_R\right| \\
&\leq \left(\sup_{y\in A_\nu}\|\nabla \mathcal{F}\|(y) \cdot \sup_{(c,u)\in\Omega_L}\|c^2Ku\|_Y+ L^2\right) \, |\nu_R|(\Omega_L)\,,
\end{align*}
where $A_\nu\subset Y$ is a compact set containing both $\int_{\Omega_L} c^2Ku \, d\nu$ and $\int_{\Omega_L} c^2Ku \, d\tilde{\nu}$ for all $\nu$ in the fixed sequence $\nu \stackrel{*}{\rightharpoonup} \mu_t$, so that $\mathcal{F}$ is Lipschitz on it, since it is globally twice Frech\'et differentiable.
Note that $A_\nu$ is compact since the map $\mu \mapsto \int_{\Omega_L}c^2Ku\, d\mu$ is continuous and $K$ is weak*-to-strong continuous, see \ref{assmpt:a5}. 
In view of the above chain of inequalities, in order to show \eqref{eq:nu_tilde_nu_close} we are only left with verifying that
\begin{align}
\label{eq:rest_term_conv_0}
\lim_{\nu\stackrel{*}{\rightharpoonup}\mu_t} \frac{|\nu_R|(\Omega_L)}{\W_2(\nu,\mu_t)}=0\,.
\end{align}
For this, notice first that by the definition of $\tilde{\nu}$ in \eqref{eq:nu_decomp}  and \eqref{eq:balls_disjoint}, setting 
 \begin{equation}\label{eq:B_notation}
 B_{\delta}:=\overset{n}{\underset{j=1}{\bigcup}}B_\delta(y^j_{t})\, \ \text{and }\ C_{ \delta}:= \Omega_L\setminus B_{ \delta}\,,
 \end{equation}
we can further decompose   
\begin{equation}\label{eq:nu_R_decomp}
\nu_{R} = \nu|_{C_{ \delta}}+ (\nu|_{B_{\delta}} -\tilde{\nu}) = \nu|_{C_{\delta}} +\sum_{j=1}^n \nu|_{B_\delta(y^j_{ t})} \cdot
\left(1-\frac{\nu(B_\delta(y^j_{ t}))^{-1}}{n} \right)\,.
\end{equation}
As $d^2_{\Omega}(z,y^j_{t}) \geq \delta^2$ for all $y^j_{t}\in Y$ and $z\in C_{\delta}$ (cf. \eqref{eq:def_y_i_and_Y}), for an optimal coupling $\pi \in\Gamma_0(\nu,\mu_t)$, we estimate 
\begin{align*}
\nu(C_{\delta})  = \pi\left(C_{\delta}\times Y\right) \leq \frac{1}{\delta^2}\int_{C_{\delta}\times Y} d_\Omega^2(z,y)\, d\pi(z,y) \leq \frac{\W_2^2(\nu,\mu_t)}{\delta^2}\,,
\end{align*}
so that
\begin{equation}
\label{eq:convergence_c_rest_term_0}
 \lim_{\nu\stackrel{*}{\rightharpoonup}\mu_t}  \frac{\nu(C_{\delta})}{\W_2(\nu,\mu_t)}  = 0\,.
\end{equation}
Therefore, in view of \eqref{eq:nu_R_decomp},  for the verification of \eqref{eq:rest_term_conv_0},  one only has to consider
\begin{align*}
\nu'_{R} :=\sum_{j=1}^n \nu|_{B_\delta(y^j_{ t})} \cdot
\left(1-\frac{\nu(B_\delta(y^j_{t}))^{-1}}{n} \right)\,.
 \end{align*}
As
\begin{align*}
|\nu'_{R}|(\Omega_L) = \sum_{j=1}^n \nu(B_\delta(y^j_{t})) \cdot
\left|1-\frac{\nu(B_\delta(y^j_{t}))^{-1}}{n} \right| = \sum_{i=1}^n  
\left|\nu(B_\delta(y^j_{t}))-\frac{1}{n} \right|\,,
\end{align*}
it is sufficient to show that for every fixed $j\in \{1,\dots,n\}$,
\begin{equation}
\label{eq:convergence_to_1}
 \lim_{\nu\stackrel{*}{\rightharpoonup}\mu_t} \frac{\left|\nu(B_\delta(y^j_{ t}))-\frac{1}{n} \right|}{\W_2(\nu,\mu_t)}  =0 \,.
\end{equation}
To this end, take again an optimal coupling $\pi \in \Gamma_0(\nu,\mu_t)$. Then for $Y_j := Y\setminus \{y^j_{t}\}$, by \Cref{lem:properties_semi_discrete_coupling}, one has
\begin{align*}
\nu(B_\delta(y^j_{ t})) &=\nu(B_\delta(y^j_{ t})\cap A_j) + \nu(B_\delta(y^j_{ t}) \cap A_j^c) = \pi((B_\delta(y^j_{ t})\cap A_j)\times \{y^j_{t}\}) + \pi((B_\delta(y^j_{ t}) \cap A_j^c)\times Y_j)\,.
\end{align*}
By \eqref{eq:mu_t_def} and Lemma \ref{lem:properties_semi_discrete_coupling}, we have that $\frac{1}{n}=\pi(A_j\times \{y_t^j\})$,  hence  
\begin{align*}\Big|\nu(B_\delta(y^j_t))-\frac{1}{n}\Big|&=  \left|\pi((B^c_\delta(y^j_{t})\cap A_j)\times \{y_t^j\}) - \pi((B_\delta(y^j_{ t}) \cap A_j^c)\times Y_j)\right|\\
&\leq \int_{(B^c_\delta(y^j_{t})\cap A_j)\times \{y^j_t\}}1 \, d\pi+\int_{(B_\delta(y^j_t) \cap A_j^c)\times Y_j}1\, d\pi\\
&\leq \frac{1}{\delta^2}\int_{(B^c_\delta(y^j_t)\cap A_j)\times \{y^j_t\}} d_\Omega^2(x,y)\, d\pi(x,y)+\frac{1}{\delta^2}\int_{(B_\delta(y^j_t) \cap A_j^c)\times Y_j} d^2(x,y)\, d\pi(x,y)\\
& \leq \frac{\W_2^2(\nu,\mu_t)}{\delta^2} + \sum_{i\neq j} \frac{1}{\delta^2}\int_{(B_\delta(y^j_t) \cap A_j^c)\times \{y_t^i\}} d^2(x,y)\, d\pi(x,y) \\
&\leq C \frac{\W_2^2(\nu,\mu_t)}{\delta^2}\,,
\end{align*}
where we have used that for all $i\neq j$, $B_\delta(y^j_t) \subset B^c_\delta(y^i_t)$, cf. \eqref{eq:balls_disjoint}. The above inequality implies  \eqref{eq:convergence_to_1}, which together with \eqref{eq:convergence_c_rest_term_0}, implies the desired  convergence \eqref{eq:rest_term_conv_0}. \qedhere
%As $1 \geq \nu(\cup_{i=1}^nB_\delta(y_i))$, the above is a consequence of the following:\\
%Take an optimal coupling $\pi$ and observe that
%\begin{align*}
%    1 - \nu(B) &= \pi(\Omega_L\times Y) - \pi(B\times Y) = \pi(C \times Y) \\
%    &= \int_{C\times Y} 1 \, d\pi(x,y) \leq \frac{1}{\delta^2}\int_{C\times Y} d^2(x,y) \, d\pi(x,y)  \leq \frac{\W_2^2(\nu,\mu)}{\delta^2}.
%\end{align*}
        
%Take an optimal coupling $\pi$ and observe that by \Cref{lem:crucial_integral_equality},
%\begin{align*}
%   1 &= \pi(\Omega_L \times Y) = \sum_{j=1}^n \nu(A_j\setminus \Sigma_j) + \pi(\Sigma\times Y),\\
%    \nu(\tilde{B}) &= \sum_{j=1}^n \nu(\tilde{B}\cap (A_j\setminus \Sigma_j)) + \nu(\tilde{B}\cap \Sigma) = \sum_{j=1}^n \nu(\tilde{B}\cap (A_j\setminus \Sigma_j)) + \pi((\tilde{B}\cap \Sigma) \times Y).
%\end{align*}
%Therefore,
%\begin{align*}
%    1 - \nu(\tilde{B}) &= \sum_{j=1}^n \nu((A_j\setminus \Sigma_j)\cap\tilde{B}^c) + \pi((\Sigma\cap\tilde{B}^c)\times Y) = \sum_{j=1}^n \int_{(A_j\setminus \Sigma_j)\cap\tilde{B}^c} 1\,d\nu(x) + \int_{(\Sigma \cap\tilde{B}^c)\times Y} 1\,d\nu(x) \\
%    &\leq \frac{1}{\delta^2} \sum_{j=1}^n \int_{(A_j\setminus \Sigma_j)\cap\tilde{B}^c} d^2(x,y_j)\,d\nu(x) + \frac{1}{\delta^2}\int_{(\Sigma \cap\tilde{B}^c)\times Y} d^2(x,y_j)\,d\nu(x) \\
%   &\leq \frac{1}{\delta^2} \sum_{j=1}^n \int_{A_j\setminus \Sigma_j} d^2(x,y_j)\,d\nu(x) + \frac{1}{\delta^2}\int_{\Sigma \times Y} d^2(x,y_j)\,d\nu(x)= \frac{\W_2(\nu,\mu_t)}{\delta^2},
%\end{align*}
\end{proof}
\noindent From now on we set
\begin{equation}\label{eq:notation_nu_i}
\nu_j:= \frac{\nu|_{B_\delta(y^j_t)}}{\nu(B_\delta(y^j_t))}\,,\quad \tilde{\nu}:= \frac{1}{n}\sum_{j=1}^n\nu_j\,,\quad \overline{\nu}^n := \nu_1\otimes\dots\otimes \nu_n\,,
\end{equation}
so that by \eqref{eq:nu_decomp}, $\nu= \frac{1}{n}\sum_{j=1}^n\nu_j+\nu_R$. Next, we also need an estimate on $\int_{B_\delta(\y_t)} d_n(\x,\y_t)\, d\overline{\nu}^n(\x)\,,$ where we recall the notation in \eqref{eq:def_y_i_and_Y}. Note that we also used a similar argument for the one-particle case in \Cref{prop:gradient_equality_one_particle}, however, since in that case there is only one Laguerre cell, the  following estimate  was trivial for $n=1$.
\begin{lem}[\textit{Quantitative Estimate}]
\label{lem:quantitative_estimates}
Suppose that \ref{assmpt:a1}-\ref{assmpt:a8} hold, and $\delta>0$ is small enough such that also \eqref{eq:balls_disjoint} holds. Then,  
\begin{align}
 \label{eq:gf_equal_leq1_inequality}
&\int_{B_\delta (\y_t)} \frac{d_n( \x, \y_t)}{\W_2(\nu,\mu_t)} d\overline{\nu}^n( \x) \leq \alpha_\nu, \text{ with } \lim_{\nu\stackrel{*}{\rightharpoonup} \mu_t}\alpha_\nu= 1\,.
\end{align}
\begin{proof} 
First of all, note that by \eqref{eq:n_ball_definition} and \eqref{eq:notation_nu_i}, one has $\overline{\nu}^n(B_\delta(\y_t))=1$. Then, by Jensen's inequality,  the choice of $\delta>0$, \eqref{eq: tau_dissip}, that $\nu_j(B_\delta(y_t^j))=1$ for all $j\in\{1,\dots,n\}$, and for $\pi \in \Gamma_0(\nu,\mu_t)$, we estimate 

\begin{align*}
\left(\int_{B_\delta(\y_t)} d_n(\mathbf{x}, \y_t) \,d\overline{\nu}^n(\textbf{x})\right)^2 \leq &\  \int_{B_\delta(\y_t)} d_n^2(\mathbf{x}, \y_t) \,d\overline{\nu}^n(\mathbf{x})\\
 = & \int_{\overset{n}{\underset{i=1}{\bigtimes}}B_\delta(y^i_t)} d_n^2((x^1, y_t^1), \ldots  , (x^n, y_t^n)) \,d\nu_1(x^1) \ldots d\nu_n(x^n)\\
 = & \int_{\overset{n}{\underset{i=1}{\bigtimes}}B_\delta(y^i_t)} \frac{1}{n} \sum_{j=1}^n d^2_{\Omega}(x^j, y_t^j) \,d\nu_1(x^1) \ldots d\nu_n(x^n) \\
= & \frac{1}{n} \sum_{j=1}^n \int_{\overset{n}{\underset{i=1}{\bigtimes}}B_\delta(y^i_t)} d^2_{\Omega}(x^j, y_t^j) \,d\nu_1(x^1) \ldots d\nu_n(x^n)\\
= & \frac{1}{n} \sum_{j=1}^n \int_{B_\delta(y^j_t)} d^2_{\Omega}(x^j, y_t^j) d\nu_{j}(x^j) \Bigg[\Pi_{i\neq j}  \int_{B_\delta(y_t^i)} d\nu_i(x^i) \Bigg]\\
 =&\ \frac{1}{n}\sum_{j=1}^n\int_{B_\delta( y_t^j)}  d^2_{\Omega}(x, y_t^j) \,d\nu_j(x) \\
=&\ \frac{1}{n}\sum_{j=1}^n\frac{1}{\nu(B_\delta( y_t^j))}\int_{B_\delta( y_t^j)}  d^2_{\Omega}(x, y_t^j) 
\,d\nu(x)\\
\leq&\ \frac{1}{n\cdot\underset{ i\in\{1,\dots,n\}}{\min}\nu(B_\delta( y_t^i)))}\, \sum_{j=1}^n\int_{B_\delta( y_t^j)\times Y}  d^2_{ \Omega}(x, y_t^j)\, d\pi(x,\tilde{y}).
  \end{align*}
Now for every $j\in \{1,\dots,n\}$, and every $x\in B_\delta(y_t^j)$, by assumption it holds that $d^2_{\Omega}(x,y_t^j) \leq d^2_{ \Omega}(x, 
\tilde y)$ for all $\tilde y \in  Y %B_\delta(y_t^j)
$, and  thus
\begin{align*}
\int_{B_\delta(y_t^j)\times Y}  d^2_{\Omega}(x,y_t^j)\, d\pi(x,\tilde{y}) \leq \int_{B_\delta(y_t^i)\times Y}  d^2_{\Omega}(x,\tilde{y})\, d\pi(x,\tilde{y})\,.
\end{align*}
Therefore, setting $\alpha_\nu := \frac{1}{n\cdot\underset{ i\in\{1,\dots,n\}}{\min}\nu(B_\delta(y_t^i))}$,  and recalling the notation in \eqref{eq:B_notation}, the previous estimates imply that 
\begin{align*}
\left(\int_{B_\delta(\y_t))} d_n(\mathbf{x}, \y_t) \,d\overline{\nu}^n(\textbf{x})\right)^2   \leq \alpha_\nu\, \int_{ \mathcal{B}\times Y}  d^2_{ \Omega}(x,y)\, d\pi(x,y)\,, 
\end{align*}
  and since $\pi \in \Gamma_0(\nu,\mu_t)$, we arrive at
\[\left(\int_{B_\delta(\y_t))} d_n(\mathbf{x}, \y_t) \,d\overline{\nu}^n(\textbf{x})\right)^2 \leq \alpha_\nu \W^2_2(\nu,\mu_t)\,,\]
 and $\alpha_\nu \rightarrow 1$ for $\nu \stackrel{*}{\rightharpoonup} \mu_t$, which concludes the proof of the lemma. 
\end{proof}
\end{lem}
Lastly, we need to generalize \Cref{lem:fixing_of_signs}, with the proof also proceeding similarly to \Cref{app:A_2_proof_5_3}.
\begin{lem}
\label{lem:exchange_f_with_int_jn}
 In the setting of \Cref{lem:quantitative_estimates}  %Suppose that Assumptions \ref{assmpt:a1}-\ref{assmpt:a8} hold. For $\delta>0$ small enough so that \eqref{eq:balls_disjoint} holds, 
and recalling \eqref{eq:notation_nu_i}, again for $\mathcal{L}^1$-a.e. $t\in [0,1]$, we have that 
\begin{equation}\label{eq:J_n_J_neglig}
\lim_{\nu\stackrel{*}{\rightharpoonup}\mu_t} \frac{\left|\int_{\Omega_L^n}J_n(\c,\u)\, d\overline{\nu}^n- \mathcal{J}(\tilde{\nu})\right|}{\W_2(\nu,\mu_t)} =0\,. 
\end{equation}
\begin{proof}
First, recalling \eqref{eq:discrete_version_of_energy}, \eqref{eq:repr_I} and  \ref{assmpt:a6}-\ref{assmpt:a7}, one has that
\begin{align*}
\int_{\Omega_L^n}   \mathcal{R} \left(\frac{1}{n} \sum_{j=1}^n (c^j)^2 u^j \right) \, d\bar \nu^n(\c,\u) & = \frac{1}{n} \sum_{j=1}^n \int_{\Omega_L^n}  (c^j)^2 d\bar \nu^n(\c,\u)= \frac{1}{n} \sum_{j=1}^n \int_{\Omega_L}  c^2 d\nu_j(c,u) =
   \int_{\Omega_L} c^2 \, d\tilde \nu(c,u)\,.
\end{align*}
Therefore,
\begin{align}\label{eq:1_J_n_minus_J}
\int_{\Omega_L^n}J_n(\c,\u)\, d\overline{\nu}^n(\c,\u)- \mathcal{J}(\tilde{\nu}) &= \int_{\Omega_L^n}  \F \left(\frac{1}{n}\sum_{j=1}^n(c^j)^2Ku^j\right)\, d\overline{\nu}^n (\c,\u) -  \F\left(\int_{\Omega_L} c^2Ku\, d\tilde{\nu} (c,u)\right) \nonumber\\
&=\underbrace{\int_{\Omega_L^n}  \F\left(\frac{1}{n}\sum_{j=1}^n(c^j)^2Ku^j\right)\, d\overline{\nu}^n- \F\left(\frac{1}{n}\sum_{j=1}^n(c^j_t)^2Ku^j_t\right)}_{=:I_\F(\nu,\mu_t)}\nonumber\\
&\quad + \underbrace{\F\left(\frac{1}{n}\sum_{j=1}^n(c^j_t)^2Ku^j_t\right)-  \F\left(\int_{\Omega_L} c^2Ku\, d\tilde{\nu}\right)}_{=:II_\F(\nu,\mu_t)}\,.
\end{align}
Then, since $ \F$ is  (twice)  Frech{\' e}t-differentiable,  and $\bar \nu^n$ is a probability measure,  it follows that
\begin{align}\label{eq:I_F}
 I_{\F}(\nu,\mu_t) =&\  \underbrace{\int_{\Omega_L^n}\Big(\nabla  \F\big(\frac{1}{n}\sum_{j=1}^n(c^j_t)^2Ku^j_t\big), \frac{1}{n}\sum_{j=1}^n \big((c^j)^2 Ku^j - (c^j_t)^2Ku^j_t\big) \Big)_Y\, d\overline{\nu}^n}_{ =:I^1_{\F}(\nu,\mu_t)}\nonumber\\
&+ \int_{\Omega_L^n}  g_I\Big(\frac{1}{n}\sum_{j=1}^n \big((c^j)^2 Ku^j - (c^j_t)^2Ku^j_t\big)\Big)\,d\overline{\nu}^n\,,
\end{align}
and
\begin{equation}\label{eq:II_F}       
 II_\F(\nu,\mu_t)  = \underbrace{\Big(\nabla  \F \big(\int_{\Omega_L} c^2 Ku\,d\tilde{\nu}\big),\frac{1}{n}\sum_{j=1}^n(c^j_t)^2Ku^j_t-\int_{\Omega_L} c^2Ku\,d\tilde{\nu}\Big)_Y}_{{=:II^1_{\F}(\nu,\mu_t)}}+  g_{II}\Big(\frac{1}{n}\sum_{j=1}^n(c^j_t)^2Ku^j_t-\int_{\Omega_L} c^2Ku\,d\tilde{\nu}\Big)\,, 
\end{equation}
 for functions $g_{I}$ and $g_{II}$ satisfying the same limiting behaviour at 0 as in \eqref{eq:g_1_2_o_t}.  Now by linearity  and \eqref{eq:notation_nu_i},
\begin{align*}
 I^1_\F(\nu,\mu_t)    &=\Big(\nabla  \F\big(\frac{1}{n}\sum_{j=1}^n(c^j_t)^2Ku^j_t\big), \int_{\Omega_L^n} \frac{1}{n} \sum_{j=1}^n \Big((c^j)^2 Ku^j - (c^j_t)^2Ku^j_t\big)\, d\overline{\nu}^n\Big)_Y \\
&= \Big(\nabla  \F\big(\frac{1}{n}\sum_{j=1}^n(c^j_t)^2Ku^j_t\big), \frac{1}{n}\sum_{j=1}^n\int_{ \Omega_L }   \big(c^2 Ku - (c_t^j)^2Ku^j_t\big)\, d\nu_j\Big)_Y
\end{align*}
and similarly,
\begin{align*}
 II^1_\F(\nu,\mu_t)=\Big(\nabla   \F\big(\int_{\Omega_L} c^2 Ku\,d\tilde{\nu}\big),\frac{1}{n}\sum_{j=1}^n\int_{\Omega_L} \big((c^j_t)^2Ku^j_t-c^2Ku\big)\,d\nu_j\Big)_Y\,.
\end{align*}
Note that by \ref{assmpt:a8}, recalling Remark \ref{rem:assumption8}, and by an application of \Cref{lem:quantitative_estimates}, it holds that
\begin{align}
\label{eq:order_of_rest_term}
\begin{aligned}
\frac{1}{n}\sum_{j=1}^n\int_{\Omega_L} \left\|(c^j_t)^2Ku^j_t-c^2Ku\right\|_Y\,d\nu_j &\leq \frac{C}{n}\sum_{j=1}^n\int_{\Omega_L} d_\Omega((c,u),(c^j_t,u^j_t))\,\,d\nu_j\\ &= C\int_{B_\delta(\c_t,\u_t)} d_n((\c,\u),(\c_t,\u_t))d\overline{\nu}^n( \c,\u)\\
&\leq C\alpha_\nu \W_2(\nu,\mu_t)\,.
\end{aligned}
\end{align}
Therefore, since $F$ is twice differentiable and by \eqref{eq:order_of_rest_term}, we obtain
\begin{align*}
 I^1_{\F}(\nu,\mu_t)+ II^1_{\F}(\nu,\mu_t)&= \Big(\nabla\F\big(\frac{1}{n}\sum_{j=1}^n(c^j_t)^2Ku^j_t\big)-\nabla\F\big(\int_{\Omega_L}c^2Ku\,d\tilde\nu\big), \frac{1}{n}\sum_{j=1}^n\int_{ \Omega_L}   \big(c^2 Ku - (c_t^j)^2Ku^j_t\big)\, d\nu_j\Big)_Y\\
& \leq\|\F\|_{C^2}\left\|\frac{1}{n}\sum_{j=1}^n(c^j_t)^2Ku^j_t-\int_{\Omega_L} c^2 Ku\,d\tilde{\nu}\right\|_Y\cdot \left\|\frac{1}{n}\sum_{j=1}^n\int_{\Omega_L} \big((c^j_t)^2Ku^j_t-c^2Ku\big)\,d\nu_j\right\|_Y \\
&=\|\F\|_{C^2} \left\|\frac{1}{n}\sum_{j=1}^n\int_{\Omega_L} \big((c^j_t)^2Ku^j_t-c^2Ku\big)\,d\nu_j\right\|_Y^2 \leq C_{\F}\alpha^2_\nu \W_2^2(\nu,\mu_t)\,.
\end{align*}
So again,  we only have to deal with the remainder terms, for which as in \eqref{eq:int_little_o_is_wasserstein_little_o}, we  claim first that 
\begin{equation}
\label{eq:g_Iint_little_o_is_wasserstein_little_o}
\lim_{\nu\stackrel{*}{\rightharpoonup}\mu_t} \frac{\left|\int_{\Omega^n_L} g_I\big(\frac{1}{n}\sum_{j=1}^n \big((c^j)^2 Ku^j - (c^j_t)^2Ku^j_t\big)\,d\overline{\nu}^n\,\right|}{\W_2(\nu,\mu_t)} =0\,.
\end{equation}
Indeed,  by definition and the continuity of $(\c,\u) \mapsto\frac{1}{n}\sum_{j=1}^n (c^j)^2 Ku^j$, for every $\varepsilon>0$  there exists a $\delta >0$ such that if $(\c,\u)\in B_\delta(\c_t,\u_t)  \subset \Omega_L^n$, then
\begin{equation*}
\Big|g_I\big(\frac{1}{n}\sum_{j=1}^n \big((c^j)^2 Ku^j - (c^j_t)^2Ku^j_t\big)\big)\Big| \leq \varepsilon \left\|\frac{1}{n}\sum_{j=1}^n \big((c^j)^2 Ku^j - (c^j_t)^2Ku^j_t\big)\right\|_Y \leq \frac{\varepsilon}{n}\sum_{j=1}^n \left\|(c^j)^2 Ku^j - (c^j_t)^2Ku^j_t)\right\|_Y\,.
\end{equation*}
Then,  using further \ref{assmpt:a8} and \eqref{eq:order_of_rest_term},
\begin{align*}
 \left|\int_{\Omega^n_L} g_I\big(\frac{1}{n}\sum_{j=1}^n \big((c^j)^2 Ku^j - (c^j_t)^2Ku^j_t\big)\,d\overline{\nu}^n\,\right| & \leq \int_{B_\delta(\c_t,\u_t)} \Big|g_I\big(\frac{1}{n}\sum_{j=1}^n \big((c^j)^2 Ku^j - (c^j_t)^2Ku^j_t\big)\Big|\,d\overline{\nu}^n\, \\
&\leq  \frac{\varepsilon}{n}\int_{B_\delta(\c_t,\u_t)}\sum_{j=1}^n \left\|(c^j)^2 Ku^j - (c^j_t)^2Ku^j_t)\right\|_Y\, d\overline{\nu}^n \\
& \leq \frac{C\varepsilon}{n}\sum_{j=1}^n  \int_{\Omega_L} d_\Omega((c,u),(c^j_t,u^j_t))\,\,d\nu_j  \leq \frac{C\varepsilon}{n}\alpha_\nu \W_2(\nu,\mu_t)\,.
\end{align*}
 Since $\lim_{\nu\stackrel{*}{\rightharpoonup} \mu_t}\alpha_\nu= 1,$ cf. \eqref{eq:gf_equal_leq1_inequality}, and $\varepsilon>0$ was arbitrary, the last estimate directly implies \eqref{eq:g_Iint_little_o_is_wasserstein_little_o}. % 
\  Lastly, we analogously show that 
\begin{equation}
\label{eq:g_II_little_o_is_wasserstein_little_o}
\lim_{\nu\stackrel{*}{\rightharpoonup}\mu_t} \frac{\left|g_{II}\big(\frac{1}{n}\sum_{j=1}^n(c^j_t)^2Ku^j_t-\int_{\Omega_L} c^2Ku\,d\tilde{\nu}\big)\,\right|}{\W_2(\nu,\mu_t)} =0\,,
\end{equation}
 which also follows by \eqref{eq:order_of_rest_term}. Indeed, again, for every $\varepsilon>0$, if $\nu$ is sufficiently close enough to $\mu_t$ in the weak$^*$-topology, then,
\begin{align*}
\left|g_{II}\big(\frac{1}{n}\sum_{j=1}^n(c^j_t)^2Ku^j_t-\int_{\Omega_L} c^2Ku\,d\tilde{\nu}\big)\,\right| &\leq  \varepsilon \left\|\frac{1}{n}\sum_{j=1}^n(c^j_t)^2Ku^j_t-\int_{\Omega_L} c^2Ku\,d\tilde{\nu}\right\|_Y \leq  C\varepsilon  \W_2(\nu,\mu_t)\,,
\end{align*}
 which again, by the arbitrariness of $\varepsilon>0$, yields \eqref{eq:g_II_little_o_is_wasserstein_little_o}, and concludes the proof of the lemma. 
\end{proof}
\end{lem}

We are now ready to prove the equality of slopes which implies that the lifting of our  AGF  is a curve of maximal slope in $\mathcal{P}_2(\Omega_L)$.
\begin{manualtheorem}{5.2}
Suppose that \ref{assmpt:a1}-\ref{assmpt:a8} hold, and let $n>1$ and $\mu_t = \frac{1}{n}\sum_{j=1}^n\delta_{(c_t^j,u_t^j)}$. For all $L>0$ % and $\Omega_L^{n} = [0,L]^n\times \B^n$, 
it holds that $|\partial \mathcal{J}|(\mu_t) = |\partial J_n|(\c_t,\u_t)$.
\end{manualtheorem}
\begin{proof}
As in the proof of \Cref{prop:gradient_equality_one_particle}, the inequality $|\partial  \mathcal{J}|(\mu_t) \geq |\partial J_n|(\c_t,\u_t)$ is a direct consequence of the definitions. For the reverse inequality, in the framework of the subsequent Subsection \ref{subsec:semi_discrete_ot}, let us choose a dual maximizer $\psi$ and thus Laguerre cells $(A_j)_{j\in \{1,\dots,n\}}$, cf. \eqref{eq:mu_atomic}--\eqref{eq:def_Sigma}. By definition of the slope in \eqref{eq:local_slope}, for every $\varepsilon >0$, there exists $\delta'>0$ such that for all $(\c,\u)\in \Omega_L^n$ with $d_n((\c_t,\u_t), (\c,\u)) \leq \delta'$, one has
\begin{equation}
\label{eq:slope_inequality}
\frac{( J_n(\c_t,\u_t) - J_n(\c,\u))^+}{d_n((\c_t,\u_t), (\c,\u))} \leq |\partial J_n|(\c_t,\u_t) + \varepsilon\,.
\end{equation}
Let then $0<\delta <\delta'$ such that $(B_\delta(c_t^j,u_t^j))_{j=1}^n$ are pairwise disjoint, as in \eqref{eq:balls_disjoint}, and
  \begin{align*}
B_\delta &:= \bigtimes_{j=1}^n B_\delta(c_t^j,u_t^j) \subset B_{\delta'}(\c_t,\u_t) \subset \Omega_L^n\,.
\end{align*}
Let now $\nu\stackrel{*}{\rightharpoonup} \mu_t$ and define $\tilde \nu$ as in \eqref{eq:nu_decomp} and $\nu_R := \nu - \tilde \nu$. 
%Next, \BBB for an \EEE arbitrary $\varepsilon>0$, fix the $\delta$ from above, and decompose
%\begin{align*}
%\nu =: \frac{1}{n} \sum_{j=1}^n \frac{\nu|_{B_\delta(c^j_t,u_t^j))}}{\nu(B_\delta(c^j_t,u_t^j))} + \BBB\nu\EEE_R\,,
%\end{align*}
%\BBB so that for $\nu\stackrel{*}{\rightharpoonup}$, the linearity of the weak$^\ast$-convergence, also implies that \[\nu_j :=\frac{\nu|_{B_\delta(c^j_t,u_t^j)}}{\nu(B_\delta(c^j_t,u_t^j))} \rightharpoonup^\ast \delta_{(c^j_t,u^j_t)}\quad \forall j\in\{1,\dots,n\},\ \text{ and } \nu_R \rightharpoonup^* 0\,.\]
Recalling also the notation in \eqref{eq:notation_nu_i}, \Cref{lem:exchange_limiting_sequence} in particular implies that,
\begin{equation}\label{eq:reducing_measure_spt}
\limsup_{\nu \stackrel{*}{\rightharpoonup}\mu_t} \frac{(J_n(\c_t,\u_t)-\mathcal{J}(\nu))^+}{\W_2(\nu,\mu_t)} = \limsup_{\nu \stackrel{*}{\rightharpoonup}\mu_t} \frac{( J_n(\c_t,\u_t)-\mathcal{J}(\tilde{\nu}))^+}{\W_2(\nu,\mu_t)},
\end{equation}
 so that it is enough to consider $\tilde{\nu}$ instead of $\nu$. By the triangle inequality, we can simply estimate 
\begin{align*}
\big| J_n(\c_t,\u_t) - \mathcal{J}(\tilde{\nu})\big| &= \left|\int_{\Omega_L^n}J_n(\c_t,\u_t)\, d\overline{\nu}^n(\c,\u) - \mathcal{J}(\tilde{\nu})\right|  \\
&\leq \left|\int_{\Omega_L^n}\big(J_n(\c_t,\u_t)-J_n(\c,\u)\big)\, d\overline{\nu}^n(\c,\u)\right| + \left|\int_{\Omega_L^n}J_n(\c,\u)\, d\overline{\nu}^n(\c,\u)-\mathcal{J}(\tilde{\nu})\right|\\
&= \Big|\int_{B_\delta(\c_t,\u_t)} \big(J_n(\c_t,\u_t) - J_n(\c,\u)\big) \,d\overline{\nu}^n(\c,\u)\Big| + \left|\int_{\Omega_L^n}J_n(\c,\u)\, d\overline{\nu}^n(\c,\u)-\mathcal{J}(\tilde{\nu})\right|\,.
\end{align*}
  
By \Cref{lem:exchange_f_with_int_jn} this time, it suffices to consider the first term in the right hand side of the above line for the slopes. Then, combining  \eqref{eq:slope_inequality}, \eqref{eq:reducing_measure_spt}, \eqref{eq:J_n_J_neglig}, and \Cref{lem:quantitative_estimates}, gives 
\begin{align*}
|\partial \mathcal{J}|(\mu_t) &=\limsup_{\nu \stackrel{*}{\rightharpoonup}\mu_t}\frac{(\mathcal{J}(\mu_t)-\mathcal{J}(\nu))^+}{\W_2(\nu,\mu_t)} \leq \limsup_{\nu \stackrel{*}{\rightharpoonup} \mu_t}\frac{\int_{B_\delta(\c_t,\u_t)} (J_n(\c_t,\u_t) - J_n(\c,\u))^+ \,d\overline{\nu}^n(\c,\u)}{\W_2(\nu,\mu_t)} \\[3pt]
\leq&\, \limsup_{\nu \stackrel{*}{\rightharpoonup} \mu_t}\left(|\partial J_n|(\c_t,\u_t) + \varepsilon\right) \int_{B_\delta(\c_t,\u_t)} \frac{d_n((\c_t,\u_t),(\c,\u))}{\W_2(\nu,\mu_t)}\, d\overline{\nu}^n(\c,\u)\\
&\leq \left(|\partial J_n|(\c_t,\u_t) + \varepsilon\right)\limsup_{\nu \stackrel{*}{\rightharpoonup}\mu_t} \alpha_\nu=|\partial J_n|(\c_t,\u_t) + \varepsilon\,,
\end{align*}
and as this holds for all $\varepsilon >0$, also the inequality $|\partial \mathcal{J}|(\mu_t) \leq |\partial J_n|(\c_t,\u_t)$ follows.
\end{proof}

\subsection{Semi-Discrete Optimal Transport on compact metric spaces}
\label{subsec:semi_discrete_ot}
Here, we recall some standard facts about semi-discrete optimal transport that can be found in  \cite{Merigot_Semi_Discrete} and in the classical treatise \cite{Aurenhammer1998}, or more recently in \cite{DieciOmarov}. Afterwards, we use them to generalize the necessary statements to our setting of a general compact metric space. \\
Let $(X,d)$ be a metric space, $\nu \in \mathcal{P}_2(X)$ and  consider an atomic measure 
\begin{equation}\label{eq:mu_atomic}
\mu := \sum_{j=1}^n\alpha_j\delta_{y^j}\,, \text{ with } \alpha_j>0\,, \ \sum_{j=1}^n \alpha_j =1, \text{ and } y^i\neq y^j \ \text{ for }i\neq j\,.
\end{equation}
Set also $Y := \{y^1,\dots,y^n\}$ and for every bounded Lipschitz function $\psi\colon X\to \R$, let us set  $\psi^j := \psi(y^j)$. Then, the  Kantorovich  duality gives 
\begin{align}\label{eq:Kantorovich}
\W_2^2(\nu,\mu)= \sup_{\phi(x)+\psi(y)\leq d^2(x,y)}\, \int_X \phi(x)\, d\nu(x) + \sum_{j=1}^n \alpha_j\psi^j.
\end{align}
Thus, one can consider the above  maximization  problem to be posed between $X$ and the discrete space $Y$. Then by, \textit{e.g.} \cite[Theorem 5.10]{villiani_old_new}, the supremum above is attained at a couple $(\psi^c,\psi)$, where
\begin{align}\label{eq:psi_c}
&\psi^c(x) := \min_{ j\in\{1,\dots,n\}} ( d^2(x,y^j) - \psi^j)\,. 
\end{align}
 Hence, 
\begin{align}\label{eq:Wasser_discr_dual}
\begin{split}
\W_2^2(\nu,\mu) = \sup_{(\psi^1,\ldots,\psi^n) \in \R^n}\int_X \psi^c(x)\, d\nu(x) + \sum_{j=1}^n \alpha_j\psi^j\,.
\end{split}
\end{align}
This duality motivates defining the \emph{Laguerre cells} $A_j$ and the 
\emph{tie sets} $\Sigma_j$  as follows.
\begin{defin}[\textit{Laguerre cells}]
 Let $(X,d)$ be a metric space, $\nu \in \mathcal{P}_2(\Omega_L)$, $\mu$ as in \eqref{eq:mu_atomic}  and $\psi \in \R^n$ be a dual maximizer  of \eqref{eq:Wasser_discr_dual}. For every $j\in\{1,\dots,n\}$, one defines the \emph{Laguerre cell} $A_j$ and \emph{tie sets} or also \emph{boundaries of the Laguerre cells} $\Sigma_j$ as
\begin{align}\label{eq:Laguerre_and_tie}
\begin{split}
\mathrm{(i)}&\quad A_j:= \{x\in X: d^2(x,y^j) - \psi^j \leq d^2(x,y^i) - \psi^i\text{ for } i\neq j\}\,,\\[2pt]
\mathrm{(ii)}&\quad  \Sigma_j := \bigcup_{i \neq j} (A_j\cap A_i)\,.
\end{split}
\end{align}
Define also the \emph{interior of the Laguerre cells} as $A_j\setminus\Sigma_j$, and the \emph{tie set} $\Sigma$ as
\begin{equation}\label{eq:def_Sigma}
\Sigma := \bigcup_{j=1}^n \Sigma_j\,.
\end{equation}
\end{defin}
\noindent A simple statement about the interior of the Laguerre cells follows directly from the definition.
\begin{lem}
\label{lem:laguerre_interior}
 For every $j\in\{1,\dots,n\}$,  it holds that 
\begin{align*}
A_j \setminus \Sigma_j = \{x\in X: d^2(x,y^j) - \psi^j < d^2(x,y^i) - \psi^i \text{ for } i\neq j\}\,.
\end{align*}
\end{lem}

%\BBB In view of \eqref{eq:Wasser_discr_dual} and \Cref{lem:laguerre_interior}, \EEE we may therefore, for a choice of dual maximizer $\psi$, write
%\begin{align*}
%\W_2^2(\nu,\mu)=&\int_X \left[\min_{\BBB j\in\{1,\dots,n\}\EEE} d^2(x,y^j) - \psi^j\right]\, d\nu(x) + \sum_{j=1}^n \alpha_j\psi^j\\  =&\,  \sum_{j=1}^n \int_{A_j\setminus\Sigma_j} \BBB\big(\EEE d^2(x,y^j) - \psi^j\BBB\big)\EEE\,d\nu(x) + \sum_{j=1}^n\int_{\tilde{\Sigma}_j}\BBB(\EEE d^2(x,y^j) - \psi^j\BBB)\EEE\,d\nu(x)+ \sum_{j=1}^n \alpha_j\psi^j\,.
%\end{align*}
%for $\tilde{\Sigma}_1 := \Sigma_1$, %and \BBB for $j\geq 2$, \EEE $\tilde{\Sigma}_j := \Sigma_j \setminus \cup_{\ell=1}^{j-1} \Sigma_\ell$. 
\noindent Thus, for every $i\in\{1,\dots,n\}$ and $x\in A_i$, it holds that
\begin{equation}
\psi^c(x) = \min_{j\in\{1,\dots,n\}} (d^2(x,y^j) - \psi^j) =d^2(x,y^i) - \psi^i\,,
\end{equation}
and in the interior, if $x\in A_i\setminus \Sigma_i$, by Lemma \ref{lem:laguerre_interior},
\begin{equation}
\label{eq:strict_inequality_outside_of_tie_set}
\psi^c(x) < d^2(x,y^j) - \psi^j \quad \text{ for all } j\neq i\,.
\end{equation}
This can be combined with known facts on Kantorovitch duality to derive the desired properties.

\begin{lem}[\textit{Properties of semi-discrete optimal couplings}]\label{lem:sem}
\label{lem:properties_semi_discrete_coupling}
In the above setting,  consider an optimal coupling $\pi \in\Gamma_0(\nu,\mu)$. Then in fact $\pi \in \mathcal{P}_2(X\times Y)$, and moreover, for a choice of dual optimizer $\psi\in \R^n$ and Laguerre cells $(A_j)_{j\in\{1,\dots,n\}}$, it holds that
\begin{align} 
\label{eq:support_semi_discrete_coupling}
\begin{aligned}
\mathrm{(i)}&\quad \mathrm{spt}\,\pi \cap (A_j\times Y) \subset A_j\times \{y^j\}, \\
\mathrm{(ii)}&\quad \mathrm{spt}\,\pi \cap (A_j^c\times Y)\subset A_j^c\times (Y\setminus\{y^j\})\,,
\end{aligned}
\end{align}
and
\begin{equation*}
\alpha_j= \pi(A_j \times \{y^j\})\,.
\end{equation*}
\begin{proof}
First of all, note that cylindrical sets %of the type $A\times B\subset X\times X$ 
are generators of the Borel $\sigma-$algebra in $X\times X$, and for any such $A\times B\subset X\times X$, with $B \subset Y^c$, by \eqref{eq:Gamma} and \eqref{eq:mu_atomic},  one has
\begin{equation*}
0\leq \pi(A\times B)\leq \pi( X%\Omega_L
 \times B)= \mu(B) = 0\,.
\end{equation*}
Therefore, it is immediate that %also 
$\mathrm{spt}\,\pi \subset X\times Y$, and $\pi \in \mathcal{P}_2( X\times Y)$. For the second part, \cite[Theorem 5.10]{villiani_old_new} implies that
\begin{equation}
\label{eq:duality_support_property}
\mathrm{spt}\pi \subset \{(x,y)\in X \times Y: \psi^c(x) + \psi(y) = d^2(x,y)\}\,.
\end{equation}
Now, by definition of the Laguerre cells, for $j\in\{1,\dots,n\}$ and $x\in A_j^c$, there must exist an $i_0 \neq j$ such that
\begin{equation*}
d^2(x,y^j) - \psi^j > d^2(x,y^{i_0}) - \psi^{i_0}\,.
\end{equation*}
But then, by \eqref{eq:strict_inequality_outside_of_tie_set}, also 
\begin{equation*}
\psi^c(x) %= \min_{j=1,\dots,n} d^2(x,y_j) - \psi_j 
< d^2(x,y^j) - \psi^j \implies \psi^c(x) + \psi(y^j) < d^2(x,y^j) \quad \forall (x,y^j) \in A^c_j\times \{y^j\}\,.
\end{equation*}
Therefore, by \eqref{eq:duality_support_property}, $\pi(A^c_j\times \{y^j\}) =0$, which implies \eqref{eq:support_semi_discrete_coupling},  and moreover, by \eqref{eq:mu_atomic} we have
\begin{equation*}
\alpha_j = \mu(\{y^j\}) = \pi( X \times \{y^j\}) = \pi(A_j \times \{y^j\}) + \pi(A_j^c \times \{y^j\}) = \pi(A_j \times \{y^j\})\,,
\end{equation*}
which concludes the proof.
\end{proof}

\end{lem}
In the next lemma we discuss properties of optimal couplings when restricted on the interior of the Laguerre cells.

\begin{lem}
\label{lem:crucial_integral_equality}
For $\pi \in \Gamma_0(\nu,\mu)$,  $j\in \{1,\dots,n\}$, and $L_j:= (A_j\setminus \Sigma_j) \times X$, it holds that
\begin{equation*}
\pi|_{L_j} = \nu|_{(A_j \setminus \Sigma_j)} \otimes \delta_{y^j},
\end{equation*}
so that
\begin{equation*}
\int_{L_j} d^2(x,y) \, d\pi(x,y) = \int_{A_j\setminus\Sigma_j} d^2(x,y^j)\, d\nu(x)\,.
\end{equation*}
\begin{proof}
For every $j\in\{1,\dots,n\}$, one has by Lemma \ref{lem:sem} that
\begin{align}
\label{eq:support_cap_Li}
L_j \cap \mathrm{spt}\,\pi \subset (A_j\setminus \Sigma_j) \times \{y^j\}\,.
\end{align}
By disintegrating $\pi|_{L_j}$ with respect to $\nu|_{(A_j \setminus \Sigma_j)}$, we have that 
\[\pi|_{L_j} =  \nu|_{(A_j \setminus \Sigma_j)}\otimes \mu_x\,,\] where $x\in A_j\setminus \Sigma_j$. From \eqref{eq:support_cap_Li} it follows that 
\begin{equation}\label{eq:spt_mu_x}
{\rm spt}\, \mu_x \subset \{y^j\}\ \mathrm{for}\ \nu \text{-a.e.}\ x\in A_j \setminus \Sigma_j\,.
\end{equation}
Note now that for every Borel subset $D\subset A_j\setminus \Sigma_j$, we have 
\begin{align*}
\nu( D%A_j \setminus \Sigma_j
) = \pi(D\times X%L_j
) =  \int_{D%A_j \setminus \Sigma_j
} \mu_x(X)\, d\nu(x)\,, 
\end{align*}
 implying that $\mu_x$ is a probability measure for $\nu$-a.e. $x$ in $A_j \setminus \Sigma_j$. 
 %\chris{For this implication I think one needs to give more reasoning- K: Localization of the above equality should be enough? M: Agreed, it is more clear with the localization} 
 Therefore, in view also of \eqref{eq:spt_mu_x}, it holds that $\mu_x = \delta_{y_j}$ concluding the proof. \qedhere

%\begin{align*}
%\int_{L_j} d^2(x,y) \, d\pi(x,y) = \int_{(A_j\setminus\Sigma_j)\times \{y^j\}} d^2(x,y)\, d\pi(x,y)=\int_{A_j\setminus\Sigma_j} \int_{\{y^j\}} d^2(x,y)\, d\pi_x(y) \, d\nu(x).
%\end{align*}
%Now as $\mathrm{spt}\,\pi \cap L_j = (A_j\setminus\Sigma_j)\times \{y^j\}$, for $\nu-$a.e. $x\in A_j\setminus\Sigma_j$, it holds that $\pi_x = \delta_{y^j}$: \\
%\BBB Indeed, it is immediate in this case that \EEE $\mathrm{spt}\pi_x \subset Y$, and since \EEE $\pi_x$ is also a probability measure, %and $\mathrm{spt}\,\pi \cap L_j = (A_j\setminus \Sigma_j) \times \{y^j\}$, for $\nu$-a.e. $x\in A_j\setminus \Sigma_j$,
%\begin{align*}
%1 = \pi_x(X) = \pi_x(Y) = \pi_x(\{y^j\})\,.
%\end{align*}
%Thus, $\pi|_{(A_j\setminus \Sigma_j) \times \BBB X\EEE} = \nu \otimes \delta_{y^j}$, and
%\begin{align*}
%\int_{A_j\setminus\Sigma_j} \int_{\{y^j\}} d^2(x,y)\, d\pi_x(y) \, d\nu(x) = \int_{A_j\setminus\Sigma_j} d^2(x,y^j) d\nu(x)\,,
%\end{align*}
%\BBB which concludes the proof. \EEE
\end{proof}
\end{lem}
Thanks to the previous lemma we can give a structural statement about the Wasserstein distance between $\nu$ and $\mu$ using the semi-discrete theory,  for which we also recall the notation in \eqref{eq:def_Sigma}.

%\chris{We do not really refer to the proposition or use it as far as I can see, but I would still not delete it. Maybe I can write a short sentence or two afterwards connecting it with the proofs beforehand \textcolor{red}{Ok!}}
\begin{prop}
\label{prop:structural_semi_discrete_wasserstein}
Let $(X,d)$ be a metric space, $\mu$ as in \eqref{eq:mu_atomic}, $\nu \in \mathcal{P}_2( X)$ and $\psi \in \R^n$ a corresponding dual maximizer. For an optimal coupling $\pi \in \Gamma_0(\nu,\mu)$, one has
\begin{equation*}
\W_2^2(\nu,\mu) = \sum_{j=1}^n\int_{A_j\setminus\Sigma_j} d^2(x,y^j)\, d\nu(x) + \int_{\Sigma\times Y} d^2(x,y)\, d\pi(x,y)\,.
\end{equation*}
\begin{proof}
 The statement follows from \Cref{lem:crucial_integral_equality}, since, recalling \eqref{eq:def_Sigma},
  \begin{align*}
\W_2^2(\nu, \mu) &= \int_{X\times Y} d^2(x,y)\, d\pi(x,y)= \sum_{j=1}^n\int_{(A_j\setminus\Sigma_j)\times Y} d^2(x,y)\, d\pi(x,y) + \int_{\Sigma\times Y} d^2(x,y)\, d\pi(x,y) \\
&= \sum_{j=1}^n\int_{A_j\setminus\Sigma_j} d^2(x,y^j)\, d\nu(x) + \int_{\Sigma\times Y} d^2(x,y)\, d\pi(x,y)\,.
\end{align*}
\end{proof}
\end{prop}

\section{Examples}
\label{sec:examples}
In this section we consider concrete examples that fall into  our general framework.

\medskip

\noindent \textbf{Finite dimensional settings.} As a preliminary observation, we note that on $\R^n$, the sublevel sets of $\mathcal{R}(u) := \|u\|_1$ and $\mathcal{R}(u) := \|u\|_{\infty}$
possess isolated extremal points. Consequently, the only geodesically connected admissible sets $\mathcal{B}$ reduce to singletons, and thus AGFs are not meaningful in this setting. In contrast, for $ \mathcal{R}(u) := \|u\|_p$ with $1 < p < \infty$,
the set of extremal points coincides with the boundary of the $p$-balls. In this case, AGFs are essentially equivalent to Riemannian gradient flows on $\partial \{ u : \|u\|_p \leq 1 \}$.

\medskip

\noindent \textbf{Strictly convex Banach spaces.} In case $\mathcal{M}$ is a strictly convex Banach space and $\mathcal{R}(u) := \|u\|_{\mathcal{M}}$, it is immediate to check that the extremal points of the unit ball of $\mathcal{R}$ coincide with $\partial \{ u : \|u\|_{\mathcal{M}} \leq 1 \}$. Therefore, if $\mathcal{M}$ admits a separable predual, then AGFs are metric gradient flows in the weak*-closure of $\partial \{ u : \|u\|_{\mathcal{M}} \leq 1 \}$ endowed with any metric metrizing the weak*-topology.

\subsection{Total variation regularization in the space of measures}
\label{subsec:tv_regularization}
As mentioned in the introduction, AGFs applied to optimization problems in the space of measures regularized with the total variation  penalization  recover the PGF-setting analyzed in \cite{chizat2018global, C22}. Let $X \subset \R^d$ be a convex, compact set and set
$\C:=C(X)$. In this case $\M$ is the space of finite signed Radon measures on $X$. Choosing as regularizer $\RR:=\|\cdot\|_{TV}$, one can prove that $
\tilde{\B}=\{\pm\delta_x:\ x\in X\}$.
As weakly*-closed, geodesically connected subset of $\tilde{\mathcal{B}}$ one can choose for example
\begin{align*}
\B_+:=\{\delta_x:\ x\in X\} \quad \text{or} \quad \B_-:=\{-\delta_x:\ x\in X\}\,.
\end{align*} 
   \noindent
  Note that the Wasserstein distance $\W_2$ metrizes the narrow convergence on $\B_\pm$. Moreover, the spaces \((\B_\pm,\W_2)\) are flat and in any case isometric to $(\B_+,\W_2)$. 
  %\begin{prop}
  %     $(\B_\pm,\W_2)$ is NPC, and in fact equality holds in \eqref{eq:convexity_of_dist_along_curve}.
  %      \begin{proof}
  %   For $x,y\in\R^n$ one has $\W_2(\delta_x,\delta_y)=|x-y|$, and the $\W_2$–geodesic between $\delta_x$ and $\delta_y$ is
%$t\mapsto \delta_{(1-t)x+ty}$. Hence, $(\R^n,|\cdot|)$ is isometric to $(\B_+,\W_2)$ via $x\mapsto\delta_x$, so $(\B_+,\W_2)$ is NPC and the inequality in  \eqref{eq:convexity_of_dist_along_curve} is an equality along these geodesics.
%For $\B_-$, define $d_{\B_-}(-\delta_x,-\delta_y):=\W_2(\delta_x,\delta_y)$. Then, $x\mapsto -\delta_x$ is an isometry, giving the same conclusion.
%        \end{proof}
%    \end{prop}
%    \noindent
As a consequence, uniqueness for the minimizing movement scheme and contraction estimates follow from the application of  \Cref{thm: uniqueness}.\\
   % \begin{rem}
   %         This example shows that the metric structure of the set of extreme points $\B$ of the regularizer $\mathcal{R}$ determines the properties of the gradient flow. In the Dirac case, $(\B_\pm,\W_2)$ is non-compact and isometric to $\R^n$, yielding a flat NPC geometry with euclidean geodesics, whereas $(\B,d_\B)$ (with $d_\B$ inducing the weak$^*$ topology) is compact and more delicate to analyze.
   %     This shows that there are topological subtleties at play for $\B$, but this is to be expected considering the generality of our setting. However, here restricting to a bounded convex domain $\Omega_L\subset\R^n$ restores compactness on the Wasserstein side, with the decomposition into the sign components $\B_\pm$ being the same as considering geodesically complete subsets of $\tilde{\B}$ in \Cref{def:omega_n}
   % \end{rem}
\noindent
Now, we can concretely examine how the AGF  framework defined in  \Cref{sec:AGFminmov} looks like in this particular scenario for the choice of $\mathcal{B}_+$,  which can be repeated verbatim for $\mathcal{B}_-$.
The discretized functional $J_n$ can be written as 
\begin{align}
J_n\big(((c^j),(\delta_{x^j}))^n_{j=1}\big)= \F\!\left(\frac{1}{n}\sum_{j=1}^n (c^j)^{2}\,K\delta_{x^j}\right)+\frac{1}{n}\sum_{j=1}^n (c^j)^{2}\,,
\end{align}
recovering the discrete problem in \cite[Equation (2)]{chizat2018global}.
Moreover, since $\mathcal{B}_+$ is isometric to $X$, the  AGF  of  $J_n$, defined as metric gradient flow in $\Omega_L^n$, is equivalent to the finite dimensional gradient flow of the functional $((c^j),(\delta_{x^j}))^n_{j=1} \mapsto J_n\big(((c^j),(\delta_{x^j}))^n_{j=1}\big)$, which is precisely how PGFs are defined.
The lifted functional defined as in \eqref{eq:repr_I}, for $\nu \in \mathcal{P}(\Omega_L)$, is written as 
\begin{align}\label{eq:liftex-1}
 \mathcal{J}(\nu) = \mathcal{F} \left(\int_{\Omega_L} c^2 K\delta_x\, d\nu(c,\delta_x)\right) + \int_{\Omega_L} c^2 d\nu(c,\delta_x)\,.
\end{align}
Note that $\mathcal{P}([0,L] \times \mathcal{B}_+)$ is isometric to $\mathcal{P}([0,L] \times X)$ through the maps $T(\nu) := \mathcal{T}_\# \nu$, where 
\[ \mathcal{T}(c,\delta_x) := (c,x)\,.\]
Indeed, $T$ is an isometry since, due to the fact that $\mathcal{T}$ is a bijective isometry, we have that
\begin{align*}
\mathcal{W}_2(T(\nu_1),T(\nu_2)) = \mathcal{W}_2(\mathcal{T}_\#\nu_1, \mathcal{T}_\#\nu_2) = \mathcal{W}_2(\nu_1, \nu_2)\,.
\end{align*}
Therefore, the metric gradient flow of the lifted functional \eqref{eq:liftex-1} is equivalent to a Wasserstein gradient flow of 
\begin{align}
\nu \mapsto \mathcal{F} \left(\int_{\Omega_L} c^2 K \delta_x \,d\nu(c,x) \right) + \int_{\Omega_L} c^2\, d\nu(c,x)\,,
\end{align}
i.e. precisely to the lifted dynamics of PGFs.

%giving the explicit form of \(J_n\) for Diracs and its connection to the lifting. We recall that 
%\begin{align}
%J_n(c,u)=J\left(\tfrac{1}{n}\sum_{j=1}^n c_j^{\,2}u_j\right),\qquad
%(c,u)\in\Omega_L^n=[0,L]^n\times\B^{\,n},
%\end{align}
%as in \eqref{eq:discrete_version_of_energy}.  Here, $\B\subset\B$ is closed and geodesically connected, and $c_j\in[0,L]$ by construction. Under the no loss of mass assumption on $\B$,
%\begin{align}
%R\!\left(\sum_{j=1}^n \alpha_j u_j\right)=\sum_{j=1}^n \alpha_j R(u_j)=\sum_{j=1}^n \alpha_j
%\quad\text{for }\alpha_j\geq0,\ u_j\in\B.
%\end{align}
%In particular, for $R=\|\cdot\|_{TV}$ and Diracs $u_j=\delta_{x_j}$, this identity holds whenever
%$\B$ is contained in a single sign component $\B_\pm$ (it may fail if $\B$ meets both signs due to cancellations). Hence, we get
%\begin{align}
%J_n\big((c_j),(x_j)\big)=F\!\left(\tfrac1n\sum_{j=1}^n c_j^{\,2}\,K\delta_{x_j}\right)+\tfrac1n\sum_{j=1}^n c_j^{\,2}.
%\end{align}
    
\subsection{One-dimensional BV functions}
Following \cite{carioni2023general} (see also \cite{trautmann2024fast}), we consider  $\mathcal{M}:=L^{\infty}((0,1))$ and $Y:=L^2((0,1))$. We recall that $L^{\infty}((0,1))$ is a Banach space, whose predual is $\mathcal{C}:=L^1((0,1))$, which is a separable space. To enforce zero boundary conditions, for $0<\varepsilon<\frac{1}{2}$ we define the set
\begin{equation*}
\mathcal{D}_{ \varepsilon} :=\left\{u \in L^{\infty}((0,1)): u(x)=0 \text { for a.e. } x \in(0, \varepsilon) \cup(1-\varepsilon, 1)\right\}\,.
\end{equation*}
We denote by
\begin{align*}
|D u|((0,1)):=\sup \left\{\int_0^1 u(x) \operatorname{div} \varphi(x) \,d x: \varphi \in C_c^1((0,1)),\ \|\varphi\|_{\infty} \leqslant 1\right\}  
\end{align*}
the $BV$-seminorm of  a function $u \in L^\infty ((0,1))$, and we choose the regularizer 
\begin{equation*}
\mathcal{R}(u):= \begin{cases}|D u|((0,1)) & \text { if } u \in B V((0,1)) \cap \mathcal{D_{\varepsilon}}\,, \\ +\infty & \text { otherwise\,. }\end{cases}
\end{equation*}
It is straightforward that $\mathcal{R}$ satisfies  \ref{assmpt:a4} and \ref{assmpt:a5}  in Section \ref{sec:setting_prelim}. By \cite[Proposition 6.11]{carioni2023general} one can also characterize the weak*-closure of the extremal points set as
\begin{equation}
 \tilde{\mathcal{B}}=\left\{\sigma \frac{1}{2} \mathds{1}_{[a, b]}: a, b \in[\varepsilon, 1-\varepsilon],\ a \leq b \text { and } \sigma \in\{-1,1\}\right\},
\end{equation}
where $\mathds{1}_A(t)$ denotes the indicator function of the set $A$. As weakly*-closed, geodesically connected subset of $\tilde{\mathcal{B}}$, we can choose either 
\begin{align*}
\mathcal{B}_+ = \left\{\frac{1}{2} \mathds{1}_{[a, b]}: a, b \in[\varepsilon, 1-\varepsilon], a \leq b \right\}\quad \text{or} \quad \mathcal{B}_- = \left\{-\frac{1}{2} \mathds{1}_{[a, b]}: a, b \in[\varepsilon, 1-\varepsilon], a \leq b\right\}.
\end{align*}
In $\mathcal{B}_{\pm }$, the weak*-convergence is equivalent to the convergence of the endpoints of the indicator function,  \textit{i.e.}, $\frac{1}{2} \sigma  \mathds{1}_{\left[a_k, b_k\right]} \stackrel{*}{\rightharpoonup} \frac{1}{2} \sigma \mathds{1}_{[a, b]}$ for $\sigma \in \{-1,1\}$, if and only if $a_k \rightarrow a$ and $b_k \rightarrow b$ as $k\to \infty$. In particular, by choosing 
\begin{align*}
d_\B\left(\frac{1}{2} \sigma \mathds{1}_{[a, b]}, \frac{1}{2} \sigma \mathds{1}_{[a', b']}\right) := \sqrt{|a-a'|^2+|b-b'|^2}\,,
\end{align*}
$d_\B$ metrizes the weak*-convergence in $\mathcal{B}_{\pm}$. With this choice of the metric, both $\mathcal{B}_+$ and  $\mathcal{B}_-$ are isometric to \[Q_\varepsilon :=\{(a,b) \in [(0,\varepsilon) \cup (1-\varepsilon, 1)]^2 : a \leq b\} \subset \R^2\,,\] and thus are flat and NPC. As a consequence, uniqueness for the AGF and the contraction estimates follow again from Theorem \ref{thm: uniqueness}. The discretized functional $J_n$ can be written via 
\begin{align}
J_n\big(((c^j),(a_j,b_j))_{j=1}^n\big)= \mathcal{F}\!\left(\frac{1}{2n}\sum_{j=1}^n (c^j)^{2}\,K\mathds{1}_{[a_j, b_j]}\right)+\frac{1}{n}\sum_{j=1}^n (c^j)^{2}\,, 
\end{align}
and the lifted functional is defined as in \eqref{eq:repr_I}, for $\nu \in \mathcal{P}(\Omega_L) = \mathcal{P}([0,L] \times \mathcal{B}_+)$, via  
\begin{align}\label{eq:liftex}
 \mathcal{J}(\nu) := \mathcal{F} \left(\int_{\Omega_L} c^2 K\mathds{1}_{[a, b]}\, d\nu(c,\mathds{1}_{[a, b]})\right) + \int_{\Omega_L} c^2\, d\nu(c,\mathds{1}_{[a, b]})\,.
\end{align}
Reasoning similarly to Subsection \ref{subsec:tv_regularization}, the space $\mathcal{P}([0,L] \times \mathcal{B}_+)$ is isometric $\mathcal{P}([0,L] \times Q_\varepsilon)$. Therefore, the gradient flow of the lifted functional is equivalent to a Wasserstein gradient flow in $\mathcal{P}([0,L] \times Q_\varepsilon)$ of 
\begin{align}
\nu \mapsto \mathcal{F} \left(\int_{[0,L] \times Q_\varepsilon} c^2 K \mathds{1}_{[a, b]} \,d\nu(c,(a,b)) \right) + \int_{[0,L] \times Q_\varepsilon} c^2 d\nu(c,(a,b))\,.
\end{align}

\subsection{Benamou-Brenier dynamical formulation of optimal transport}
  
Introduced in  \cite{benamou2000computational}, the \textit{Benamou-Brenier formula} allows to compute an optimal transport between two probability measures $\rho_0$ and $\rho_1$ on a closed, bounded domain $X \subset \mathbb{R}^d$ through the minimization of the \emph{kinetic energy}
\begin{equation}\label{eq:Kinetic}
(\rho,v) \mapsto \frac{1}{2} \int_0^1\int_X \left|v_t(x) \right|^2 \mathrm{d}\rho_t(x)\, dt\,,    
\end{equation}  
where $\rho_t \in M^{+}(X)$ is a time-dependent probability measure interpolating between $\rho_0$ and $\rho_1$, and $v_t: [0,1]\times X \rightarrow \mathbb{R}^d$ is a vector field, and the pair $\left(\rho_t, v_t\right)$ satisfies the \emph{continuity equation} 
\[\partial_t \rho +\operatorname{div}\left(\rho v_t\right)=0\] in the sense of distributions. 
One can reformulate the Benamou-Brenier energy as a convex functional on the space of Borel measures $\mathcal{M}:=M^{+}([0,1]\times X) \times M\left([0,1]\times X; \mathbb{R}^d\right)$  via
\begin{equation}\label{Convex_kinetic}
\mathbf{B}(\rho, m):=\begin{cases}\frac{1}{2} \int_0^1 \int_{X}\left|\frac{d m}{d \rho}\right|^2d \rho(t, x)\,, \ \text{if } \rho \geq 0, m \ll \rho\\[3pt]
+ \infty\,,\qquad \qquad \qquad \quad  \text{ otherwise}\,.
\end{cases}
\end{equation}
 In this case, the continuity equation becomes the linear constraint $\partial_t \rho+\operatorname{div} m=0$. In \cite{schmitzer2019dynamic, bredies2023generalized, bredies2020optimal} the Benamou-Brenier energy was used as a convex regularizer to solve dynamic inverse problems, and its sparsity properties were subsequently investigated \cite{bredies2021extremal, bredies2023generalized, duval2024dynamical, carioni2025sparsity}. This was done by
considering the functional
\begin{equation}
 \mathcal{R}_{\alpha, \beta}(\rho, m):=\beta \mathbf{B}(\rho, m)+\alpha\|\rho\|_{TV %M(\Omega)
}\,, \quad \text { subjected to } \quad \partial_t \rho+\operatorname{div} m=0
\end{equation}
defined for all $(\rho, m) \in \mathcal{M}$ and $\alpha > 0, \beta>0$. 
Here, AGFs apply directly to composite optimization problems regularized with $\mathcal{R} := \mathcal{R}_{\alpha,\beta}$ since one can show that $\mathcal{R}_{\alpha,\beta}$ satisfies \ref{assmpt:a4}-\ref{assmpt:a5}.
Moreover, as proved in \cite{bredies2021extremal}, it is possible to characterize the extremal points of 
\[B_{\alpha, \beta}:=\left\{(\rho, m) \in \mathcal{M}:  \mathcal{R}_{\alpha, \beta}(\rho, m) \leqslant 1, \text{subjected to } \partial_t \rho+\operatorname{div} m=0\right\}.\] 
\begin{thm}[Theorem 6, \cite{bredies2021extremal}]
Let $\alpha > 0, \beta>0$. It holds that 
\begin{align}
{\rm Ext}(B_{\alpha,\beta}) = \{(\rho,m) : \rho=C_\gamma\, d t \otimes \delta_{\gamma(t)}, \quad m=\dot{\gamma}(t) C_\gamma\, d t \otimes \delta_{\gamma(t)}\} \cup \{(0,0)\}\,,
\end{align}
where $\gamma:[0,1] \rightarrow X$ is a curve in $H^1([0,1]; X)$ and $C_\gamma:=\left(\frac{\beta}{2} \int_0^1|\dot{\gamma}(t)|^2 d t+\alpha\right)^{-1}$.
\end{thm}
\noindent To define AGFs, one can consider the weak*-closed set $\mathcal{B} \subset \tilde{\mathcal{B}}$ defined as
\begin{align}
\mathcal{B} := \{(\rho,m) : \rho= b\, d t \otimes \delta_{\gamma(t)}, \quad m=\dot{\gamma}(t) b\, d t \otimes \delta_{\gamma(t)}, \text{for } b \in [0,C_\gamma]\}\,.
\end{align}
Moreover, the metric $d((\rho_1,m_1), (\rho_2,m_2)) = \|\gamma_1 - \gamma_2\|_{H^1([0,1]; X)}$
metrizes the weak* convergence on $\mathcal{B}$ on curves of bounded speed ($b > 0$).
Therefore since $H^1([0,1]; X)$ is an Hilbert space and any Hilbert space is NPC, uniqueness and contractivity estimates for AGFs again follow from the NPC framework. \\
Here, the functional $J_n$ can be written (on curves of bounded speed) as 
\begin{align*}
J_n\big(((c^j),(\gamma_j))_{j=1}^n \big)= \mathcal{F}\left(\frac{1}{n}\sum_{j=1}^n (c^j)^{2}\,K(C_{\gamma_j} dt \otimes \delta_{\gamma_j(t)}, C_{\gamma_j} \dot \gamma_j(t)dt \otimes \delta_{\gamma_j(t)}) \right)+\frac{1}{n}\sum_{j=1}^n (c^j)^{2}\,, 
\end{align*}
and AGFs correspond to the $([0,L] \times H^1([0,1]; X))^n$ gradient flow of such functional. The lifting can be written similarly as in the previous sections as a gradient flow in $\mathcal{P}_2([0,L] \times H^1([0,1]; X))$. Interestingly such algorithm has been implemented (without mathematical investigation) as an acceleration step in \cite{bredies2023generalized}.

\subsection{Further examples}

In this section we list with less details several further examples that can be handled by AGFs. We refer the reader to the references for more insights about extremality properties of the given problems.

\medskip

\noindent \textbf{Kantorovich-Rubinstein-norms and Wasserstein distances on balanced signed measures.} As shown in the recent papers \cite{carioni2025extremal, bartolucci2024lipschitz, carioni2023general} Kantorovich-Rubinstein norms can be used as regularizers for inverse problems and mean-field neural networks training. The ${\rm KR}$-norm can be defined as follows, cf. \cite{bartolucci2024lipschitz}:
\begin{align}
\|\mu\|_{\rm KR} := |\mu(X)| + \sup\left\{\int_X f(z) \, d\mu(z) : f(e) = 0, L(f) \leq 1\right\}\,,
\end{align}
where $\mu \in M_1(X)$ is a signed measure with finite first moment, $e\in X$ is a base point and $L(f)$ is the Lipschitz constant of $f$.
In case of balanced measures, \textit{i.e.}, when $\mu_+(X) = \mu_-(X)$, the KR norm is simply the $\W_1$-distance between the positive and the negative part of $\mu$.
Alternatively, other possibilities to include unbalanced measures are possible through infimal convolution approaches, cf. \cite{carioni2025extremal, }.
As shown in \cite{bartolucci2024lipschitz} (see also \cite{carioni2023general} and \cite{carioni2025extremal} for corresponding and similar results), the extremal points of the ball of the ${\rm}$ KR-regularizer can be characterized. In case $X$ is a compact set, a simple consequence of the previously mentioned results yields that extremal points of 
\[\{\mu \text{ balanced}: \|\mu\|_{\rm KR} + \|\mu\|_{TV} \leq 1\}\,,\] are dipoles of the form $\mu_{\rm dip} = \frac{\delta_x - \delta_y}{d(x,y) + 1}$. 
Therefore, AGFs applied to optimization problems on balanced signed measures regularized with the ${\rm KR}$-norms are equivalent to Euclidean gradient flows in $([0,L] \times X \times X)^n$. Note that such algorithm have also been implemented (without mathematical investigation) in \cite[Section 8]{bartolucci2024lipschitz}.

\medskip

\noindent \textbf{PDE-regularized optimization problems and splines.}
Optimization problems regularized with the total variation norm of a scalar differential operator $L$ in $\mathbb{R}^d$, \textit{i.e.}, $\mathcal{R}(u) := \|Lu\|_{TV}$ also fit the AGF framework. Indeed, the extremal points of the ball of $\mathcal{R}$ can be characterized (up to elements in the null-space of $L$) as fundamental solutions translated by $x \in \R^{d}$, cf. \cite{bredies2020sparsity, unser2017splines, unser2019native}. Therefore, depending on the regularity of the fundamental solution, AGFs applied to such problems are Euclidean gradient flows in $([0,L] \times \mathbb{R}^d)^n$.

\medskip

\noindent \textbf{BV functions in higher dimensions.} Optimization problems on BV functions in $\mathbb{R}^d$ regularized with the $BV$-seminorm can also be studied through AGFs. The extremal points of the unit ball of the $BV$-seminorm have been shown to be indicator functions of simple sets, cf. \cite{bredies2020sparsity, ambrosio2001connected, fleming1957functions, cristinelli2026conditional}. However, in this case, it is challenging to characterize the weak* distance on $\mathcal{B}$ and the properties of the resulting metric space \cite{de2024exact,de2021towards}.   
\section*{Acknowledgements}
C.A. and M.C.'s research is supported by the NWO-M1 grant \emph{Curve Ensemble Gradient Descents for Sparse Dynamic Problems} (Grant Number OCENW.M.22.302) and NWO-Vidi grant \emph{SPARGO: Exploring and Exploiting the Geometric Landscape of Infinite-Dimensional Sparse Optimization} (Grant Number VI.Vidi.243.200). K.Z. was funded by the Deutsche Forschungsgemeinschaft (DFG, German Research Foundation) – CRC 1720 – 539309657 and also acknowledges previous support by the Hausdorff Center for Mathematics (HCM) under
Germany’s Excellence Strategy -EXC-2047/1-390685813.

\printbibliography

\appendix
\large{\section{Appendix}
\renewcommand{\thesection}{A} 

\normalsize

\subsection{Minimizing movements and curves of maximal slope}\label{sec:met_th}

In order to formulate the concept of gradient flows in our setting, we have chosen (generalized) minimizing movements (see Section \ref{sec:AGFminmov} and Section \ref{sec:3_lifted}). For completeness we recall the classical setup from \cite[Chapter 2]{AGS05}.
\begin{defin}[\textit{Discrete Scheme}]\label{def:discrete_scheme}
Let $(X,d)$ be a complete metric space and $F:X\rightarrow (-\infty,+\infty]$. The discrete scheme is defined as follows.
Given a partition of the time interval $[0,+\infty)$ by a sequence of time steps ${\boldsymbol\tau}:=(\tau_n)_{n\in \N}\subset (0,+\infty)$ with $|{\boldsymbol\tau}|:= \sup_{n\in \N} \tau_n <+\infty$, let us set
\begin{align}\label{eq:time_partition}
\begin{split}
\mathcal{P}_{\boldsymbol\tau}&:=\{0=:t^0_{\boldsymbol\tau}<t^1_{\boldsymbol\tau}<\dots <t^n_{\boldsymbol\tau}<\dots\}\,,\quad I_{\boldsymbol\tau}^n:=(t_{\boldsymbol\tau}^{n-1},t_{\boldsymbol\tau}^{n}]\,,\\
\tau_n&=:t_{\boldsymbol\tau}^{n}-t_{\boldsymbol\tau}^{n-1}\,,\quad \lim_{n\to \infty}t_{{\boldsymbol\tau}}^n=\sum_{k=1}^\infty \tau_k=+\infty\,, 
\end{split}
\end{align}
and consider the functionals
\begin{align}
\label{eq:discrete_scheme}
G(\tau_n,u_{\boldsymbol\tau}^{n-1};v) := F(v)+\frac{1}{2\tau_n}d^2(v,u_{\boldsymbol\tau}^{n-1})\,.
\end{align}
Given an admissible initialization $u_{\boldsymbol\tau}^0 \in \mathrm{Dom}(F)$, we define successively $(u^n_{\boldsymbol\tau})_{n\in \N}\subset X$, with the property that
\begin{align}
\label{eq:discrete_functional}
u^n_{\boldsymbol\tau} \in \argmin_{v\in X}G(\tau_n,u_{\boldsymbol\tau}^{n-1};v)\, \ \ \forall n\geq 1\,.
\end{align}
One then defines accordingly the piecewise constant interpolation 
\begin{equation}\label{eq:piecewise_constant_int}
\overline{u}_{\boldsymbol\tau}(t) := \left\{\begin{IEEEeqnarraybox}[
\IEEEeqnarraystrutmode
\IEEEeqnarraystrutsizeadd{2pt}
{2pt}
][c]{rCl}
& u_{\boldsymbol\tau}^0\,, & \quad \text{ if } t=0\,, \\
& u_{\boldsymbol\tau}^n\,, & \quad \text{ if } t\in (t^{n-1}_{\boldsymbol\tau},t^{n}_{\boldsymbol\tau}]\, \ \ \forall n\geq 1\,.
\end{IEEEeqnarraybox}
\right.
\end{equation}
We call $\overline{u}_{\boldsymbol\tau}$ of \eqref{eq:piecewise_constant_int} a \textit{discrete solution} corresponding to the partition $\mathcal{P}_{\boldsymbol\tau}$ of \eqref{eq:time_partition}.
\end{defin}

\begin{rem}
The operator which provides all the solutions to the minimization problem \eqref{eq:discrete_functional} (given $u_{\boldsymbol\tau}^{n-1}$) is generically multi-valued, and is often called the \textit{resolvent operator}. For a general $\tau>0$ and $u\in X$ it is defined via
\begin{equation*}
J_{\tau}(u):=\argmin_{v\in X}G(\tau,u;v)\,,\ \textit{i.e.}, \quad u_{\tau}\in J_{\tau}(u)\iff G(\tau, u;{u_\tau})\leq G(\tau,u;v) \ \ \forall v\in X\,.
\end{equation*}
\noindent
Thus, Definition \ref{def:discrete_scheme} can equivalently be phrased as follows: $\bar u_{\boldsymbol\tau}$ is a discrete solution iff $u_{\boldsymbol\tau}^n\in J_{\tau_n}(u_{\boldsymbol\tau}^{n-1})$ for every $n\geq 1$.
\end{rem}

\begin{defin}[\textit{Minimizing Movements}]
\label{def:minimizing_movements}
Let $(X,d)$ be a complete metric space, $F:X\rightarrow (-\infty,+\infty]$ and $G$ be defined as in \eqref{eq:discrete_scheme}. We call a curve $x_t:[0,+\infty)\rightarrow X$ a \textit{minimizing movement} for $G$ starting from $x_0 \in \mathrm{Dom}(F)$ iff there exists a sequence of partitions $({\boldsymbol\tau}_k)_{k\in \N}$  with $|{\boldsymbol\tau}_k|\rightarrow 0$ as $k\to \infty$, and a sequence of discrete solutions $(\overline{x}_{{\boldsymbol\tau}_k})_{k\in \N}$ as in \eqref{eq:piecewise_constant_int} such that
\begin{align}\label{def:gener_minimizing_movement}
\begin{split}
\mathrm{(i)}&\quad \lim_{k\rightarrow \infty} F(x^0_{{\boldsymbol\tau}_k}) = F(x_0)\,,\quad \limsup_{k\to \infty}d(x^0_{{\boldsymbol\tau}_k},x_0) < +\infty\,, \\
\mathrm{(ii)}&\quad \overline{x}_{{\boldsymbol\tau}_k,t} \rightarrow x_t \text{ for all } t\geqslant 0\,, \ \text{as }k\to +\infty\,.
\end{split}
\end{align}
More precisely, in \cite[Definition 2.0.6]{AGS05}, such $x_t$ are called \emph{generalized minimizing movements}, and denoted by $x_t\in \mathrm{GMM}(G,x_0)$. A related concept is the one of  \emph{minimizing movements} that can be defined as follows.
A curve $x_t:[0,+\infty)\rightarrow X$ is a \textit{minimizing movement} for $G$ starting from $x_0 \in \mathrm{Dom}(F)$, and one writes $x_t\in {\rm MM}(G,x_0)$, iff for every partition ${\boldsymbol\tau}:=(\tau_n)_{n\in \N}$, there exists a discrete solution $\overline{x}_{\boldsymbol\tau}$ as in \eqref{eq:discrete_functional}-\eqref{eq:piecewise_constant_int}, such that
\begin{align}\label{def:minimizing_movement}
\begin{split}
\mathrm{(i)}&\quad\lim_{|{\boldsymbol\tau}|\downarrow 0} F(x^0_{\boldsymbol\tau}) = F(x_0)\,,\quad \limsup_{|{\boldsymbol\tau}|\downarrow 0}d(x^0_{\boldsymbol\tau},x_0) < +\infty\,,\\
\mathrm{(ii)}&\quad\overline{x}_{\boldsymbol\tau,t} \rightarrow x_t\text{ for all } t\geqslant 0\,, \ \text{as }|{\boldsymbol\tau}|\rightarrow 0\,.
\end{split}
\end{align}
\end{defin}

We also recall and combine some standard results on uniqueness of gradient flows on metric spaces of non-positive curvature in the following, using as a reference \cite[Chapter 4]{AGS05}.
\begin{thm}
\label{thm:uniqueness+regularity_mm_in_npc_case}
Let $(X,d)$ be a complete metric space of non-positive curvature cf. \Cref{def:NPC}, and $F:X\rightarrow (-\infty,+\infty]$ be $\lambda$-convex for some $\lambda\in \R$.
Then for every initial point $x_0\in X$, there exists a unique $x_t\in\mathrm{GMM}(G,x_0)$ that is locally Lipschitz and fulfills the following contractivity property: for every $x_0,\tilde{x}_0\in X$,
\begin{equation}\label{eq:contract_prop}
d\big(x_t,\tilde{x}_t\big) \leqslant  e^{-\lambda t}d\big(x_0,\tilde{x}_0\big) \quad \text{for } \mathcal{L}^1\text{-a.e. }t>0\,.
\end{equation}
 Given $x_0\in X$, the unique corresponding $ x_t\in\mathrm{GMM}(G,x_0)$ is also the unique solution to the following evolution variational inequality $(\mathrm{EVI})_\lambda$:
\begin{equation}\label{eq:EVI}
\frac{1}{2}\frac{d}{dt}d^2(x_t,y)+\frac{\lambda}{2}d^2(x_t,y) + F(x_t) \leq F(y)\, \quad \forall y\in \mathrm{Dom}(F)\,,
\end{equation}
and is a curve of maximal slope from $x_0$ for $F$, hence fulfills the inequality:
\begin{equation}\label{eq:nonlifted_c_o_m_s}
2\frac{d}{dt}F(x_t) \leq - |x_t'|^2 - |\partial F|^2(x_t)\,.
\end{equation}
In addition, for all absolutely continuous curves $\tilde{x}_t$ starting from $x_0$, $x_t$ is the unique one that fulfills \eqref{eq:nonlifted_c_o_m_s}, and therefore it is also the unique one with this property among all curves of maximal slope from $x_0$ with respect to $F$.
\begin{proof}
Since $F$ is $\lambda$-convex and $(X,d)$ is NPC, by \cite[Remark 4.0.2]{AGS05} it follows that  $v\mapsto G(\tau,w;v)$ is $(\tau^{-1}+\lambda)$-convex (see \cite[Assumption 4.0.1]{AGS05}. Therefore, existence, uniqueness and regularity of $u\in {\rm MM}(G,x_0)$ as well as \eqref{eq:contract_prop} follow from \cite[Theorem 4.0.4]{AGS05}.\\
To prove that ${\rm MM}(G,x_0)={\rm GMM}(G,x_0)$, note that the solutions to the discrete scheme \eqref{eq:discrete_scheme} are unique for $|\boldsymbol{\tau}|$ small enough, as then \cite[Theorem 4.1.2]{AGS05} is applicable because then $\lambda |\boldsymbol{\tau}| >-1$. By Definition \ref{def:minimizing_movements}, one has that for all partitions $\boldsymbol{\tau}$ the unique discrete solution $\overline{x}_{\boldsymbol{\tau}}$ of \eqref{eq:piecewise_constant_int} satisfies 
$$\overline{x}_{\boldsymbol{\tau},t}\rightarrow x_t \ \text{ for } \ |{\boldsymbol{\tau}}|\rightarrow 0\,.$$ 
Now, if $\tilde{x}_t \in {\rm GMM}(G,x_0)$, by the definition of generalized minimizing movements, $\tilde{x}_t$ is obtained by a specific choice of partitions $({\boldsymbol{\tau}}_k)_{k\in \N}$ and discrete solutions $(\overline{x}_{{\boldsymbol{\tau}}_k})_{k\in \N}\subset X$, for which
$\overline{x}_{{\boldsymbol{\tau}}_k,t} \rightarrow\tilde{x}_t$ as $k\rightarrow \infty$. 
But then it must also hold that $\overline{x}_{{\boldsymbol{\tau}}_k,t} \rightarrow x_t$ as $k\rightarrow \infty$, hence $x_t = \tilde{x}_t$, showing the uniqueness. The uniqueness among all curves of maximal slope follows from the existence of  $(\mathrm{EVI})_\lambda$ solutions, cf. \cite[Theorem 4.2]{Muratori2020}.
\end{proof}
\end{thm}

\subsection{Proof of Lemma \ref{lem:fixing_of_signs}}
\label{app:A_2_proof_5_3}

\label{lem:fixing_of_signs_appendix}
Note that thanks to \ref{assmpt:a8} and recalling Remark \ref{rem:assumption8} there exists $C>0$ such that 
for every $(c',u'),  (\tilde c,\tilde u)  \in \Omega_L$, it holds that
\begin{equation}\label{eq:lipschitz_K_n}
\|(c')^2Ku' - \tilde{c}^2K\tilde{u}\|_Y \leq C d_\Omega((c',u'),(\tilde{c},\tilde{u}))\,.
\end{equation}

Then, since $ \mathcal{F}$ is Frech{\'e}t-differentiable  and $\nu\in\mathcal{P}_2(\Omega_L)$, one has that 
\begin{align}\label{eq:F_minus_average_1}
\int_{\Omega_L}  \F(c^2 Ku)\, d\nu-\F\left(\int_{\Omega_L}  c^2 K u\,d\nu\right) =&\ \int_{\Omega_L}  \Big[  \F(c^2 Ku) - \F(c_t^2Ku_t) + \F(c_t^2Ku_t) -\F\Big(\int_{\Omega_L}  c^2 K u\, d\nu\,\Big              )\Big]\,d\nu \notag \\
=&\ \int_{\Omega_L} \big((\nabla \F(c_t^2Ku_t), c^2 Ku - c^2_tKu_t)_Y +  g_1(c^2Ku-c_t^2Ku_t)\big)\,d\nu \notag \\
&+   \big(\nabla \F\Big(\int_{\Omega_L}  c^2 K u\,d\nu\Big),c^2_tKu_t-\int_{\Omega_L}  c^2K u\,d\nu\big)_Y\, \notag \\
&+  g_2\Big(c_t^2Ku_t-\int_{\Omega_L} c^2K u\,d\nu\Big)\,,
\end{align}
 for functions $g_i\colon Y\to \R$, $i=1,2$, with 
\begin{equation}\label{eq:g_1_2_o_t}
\lim_{y\to 0}\frac{|g_i(y)|}{\|y\|_Y}=0\,\ \text{for }i=1,2\,.
\end{equation}
 In view of \eqref{eq:F_minus_average_1} and using that $\F\in C^2$ (cf. \ref{assmpt:a2}), the Cauchy-Schwartz inequality, \eqref{eq:lipschitz_K_n}, Jensen's inequality, and \eqref{eq:Wasserstein_2}--\eqref{eq:d_Omega},   we first estimate 
\begin{align*}
&\int_{\Omega_L} \Big(\big(\nabla  \F(c_t^2 K u_t), c^2 Ku - c^2_tKu_t \big)_Y \,d\nu +  \big(\nabla \F\Big(\int_{\Omega_L}  c^2 K u\,d\nu\Big),c^2_tKu_t-\int_{\Omega_L}  c^2K u\,d\nu\big)_Y\, \\
& =\int_{\Omega_L} \big(\nabla \F(c_t^2 K u_t)-\nabla \F\Big(\int_{\Omega_L} c^2 K u\,d\nu\Big), c^2 Ku - c^2_tKu_t \big)_Y \,d\nu\, \\
& \leq \|\F\|_{C^{2}} \left\|\int_{\Omega_L}( c^2 Ku - c^2_tKu_t )\,d\nu\right\|_Y^2\leq  \|\F\|_{C^{2}}\left(\int_{\Omega_L} d_\Omega((c,u),(c_t,u_t))\, d\nu\right)^2\\
&\leq \|\F\|_{C^{2}} \int_{\Omega_L} d^2_\Omega((c,u),(c_t,u_t))\, d\nu\leq  \|\F\|_{C^{2}} \W_2^2(\nu,\mu_t)\,.
\end{align*}
Therefore, in order to prove \eqref{eq:Wasserstein_limit_1}, it suffices to deal with the $g_1,g_2$-terms in \eqref{eq:F_minus_average_1}, \textit{i.e.}, to prove that 
\begin{equation}
\label{eq:int_little_o_is_wasserstein_little_o}
\lim_{\nu\stackrel{*}{\rightharpoonup}\mu_t} \frac{\left|\int_{\Omega_L} g_1(c^2Ku-c_t^2Ku_t)\, d\nu\right|}{\W_2(\nu,\mu_t)} =\lim_{\nu\stackrel{*}{\rightharpoonup}\mu_t} \frac{\left| g_2(c_t^2Ku_t-\int_{\Omega_L}c^2Ku\,d\nu)\,  \right|}{\W_2(\nu,\mu_t)}=0\,.
\end{equation}
Indeed, for the first limit, by \eqref{eq:lipschitz_K_n} and \eqref{eq:g_1_2_o_t}, and recalling the notation in \eqref{eq:delta_t_ball}, for every $\varepsilon>0$ there exists $\delta>0$, such that if $(c,u)\in B_{\delta,t}$, then 
\[|g_1(c^2Ku-c_t^2Ku_t) \leq \varepsilon \|c^2Ku-c_t^2Ku_t\|_Y\,,\] 
and once again by \eqref{eq:lipschitz_K_n}, 
\begin{equation}\label{eq:Wasserstein_simple_est}
\int_{\Omega_L}\|c^2Ku-c_t^2Ku_t\|_Y\,d\nu\leq C\int_{\Omega_L}d_\Omega((c,u),(c_t,u_t))\,d\nu\leq C\W_2(\nu,\mu_t)\,.
\end{equation}
Since by \eqref{eq:Wasserstein_2}--\eqref{eq:d_Omega} we have $\W_2^2(\nu,\mu_t)\geq \delta^2\nu(\Omega_L\setminus B_{\delta,t})$ (cf. also \eqref{eq:quantitative_estimate_1n}), we estimate 
\begin{align*}
\left|\int_{\Omega_L}  g_1(c^2Ku-c_t^2Ku_t)\,d\nu \right| &\leq   \int_{ B_{\delta, t}} \Big|g_1(c^2Ku-c_t^2Ku_t)\Big|\,d\nu + \int_{ \Omega_L\setminus B_{\delta, t}} \Big| g_1(c^2Ku-c_t^2Ku_t)\Big|\,d\nu \\
&\leq  \varepsilon \int_{ B_{\delta,t}} \|c^2Ku-c_t^2Ku_t\|_Y\,d\nu + C \nu( \Omega_L\setminus B_{\delta,t}) \\
&\leq  \varepsilon  \W_2(\nu,\mu_t) + \frac{C}{\delta^2} \W_2^2(\nu,\mu_t).
\end{align*}
Sending $\nu \stackrel{*}{\rightharpoonup} \mu_t$ here, gives 
\begin{equation*}
 \lim_{\nu\stackrel{*}{\rightharpoonup}  \mu_t}\frac{\big|\int_{\Omega_L} g_1(c^2Ku-c_t^2Ku_t)\,d\nu\big|}{\W_2(\nu,\mu_t)} \leq \varepsilon\,,
\end{equation*}
 and since $\varepsilon>0$ was arbitrary, the first part of \eqref{eq:int_little_o_is_wasserstein_little_o} follows. For the second one,
 by \eqref{eq:Wasserstein_simple_est} it holds
\begin{equation}\label{eq:Wass_distance_easy}
\left\|c_t^2Ku_t-\int_{\Omega_L} c^2K u\,d\nu\right\|_Y \leq C \W_2(\nu,\mu_t)\to 0 \ \text{ as }\nu\stackrel{*}{\rightharpoonup}  \mu_t\,,
\end{equation}
and therefore, using \eqref{eq:Wass_distance_easy} and also \eqref{eq:g_1_2_o_t}, we deduce that 
\begin{align*}
\lim_{\nu\stackrel{*}{\rightharpoonup}\mu_t} \frac{| g_2(c_t^2Ku_t-\int_{\Omega_L}c^2Ku\,d\nu)\, |}{\W_2(\nu,\mu_t)} \leq C \frac{\big|g_2(c_t^2Ku_t-\int_{\Omega_L}c^2Ku\,d\nu)\big|}{\big\|c_t^2Ku_t-\int_{\Omega_L}c^2Ku\,d\nu\big\|_Y}\,=0\,,
\end{align*}
which concludes the proof. 
\end{document}